\documentclass[final]{siamltex}
\usepackage{latexsym,amsbsy,amssymb,amsmath}
\usepackage{color}
\usepackage{subcaption}
\usepackage{graphicx}
\usepackage{psfrag}
\usepackage{float}
\usepackage{t1enc}
\usepackage{amsfonts}
\usepackage{stmaryrd}
\usepackage{multirow}
\usepackage[table]{xcolor}

\usepackage{epstopdf}
\usepackage{graphics}
\usepackage{epsf}
\usepackage{xfrac}
\usepackage{amscd}
\usepackage{wrapfig}
\usepackage[mathscr]{euscript}
\usepackage[english]{babel}
\usepackage{epsfig}
\usepackage{dsfont}
\everymath{\displaystyle}
\usepackage{setspace}
\usepackage{subcaption, algorithmicx}
\usepackage[]{algorithm}
\usepackage[titletoc,title]{appendix}
\usepackage{enumitem}
\numberwithin{equation}{section}
\newtheorem{remark}[theorem]{Remark}

\newcommand{\bw}{\pmb{w}}
\newcommand{\bx}{\pmb{x}}

\newcommand{\bU}{\pmb{U}}

\newcommand{\bL}{\pmb{L}}

\newcommand{\bH}{\pmb{H}}

\newcommand{\bff}{\pmb{f}}
\newcommand{\bu}{\pmb{u}}
\newcommand{\baru}{\overline{u}}

\newcommand{\dtildeu}{\widetilde{\widetilde{u}}}

\newcommand{\ih}{\mathfrak{h}}
\newcommand{\iC}{\mathcal{C}}

\newcommand{\mR}{\mathbb{R}}

\newcommand{\ahalf}{\sfrac{1}{2}}

\definecolor{OliveGreen}{rgb}{0,0.6,0}

\title{High Order Explicit Local Time-Stepping Methods For Hyperbolic Conservation Laws\footnotemark[1]}

\author{Thi-Thao-Phuong Hoang\footnotemark[2]
\and Lili Ju\footnotemark[3]\ 
\and Wei Leng\footnotemark[4]\ 
\and Zhu Wang\footnotemark[3]}

\begin{document}

\maketitle

\renewcommand{\thefootnote}{\fnsymbol{footnote}}
\footnotetext[1]{This work is partially supported by US Department of Energy under grant number DE-SC0016540 and US National Science Foundation under grant number DMS-1818438.}
\footnotetext[2]{Department of Mathematics and Statistics, Auburn University, AL 36849, USA. Email: \texttt{tzh0059@auburn.edu}.}
\footnotetext[3]{Department of Mathematics, University of South Carolina, Columbia, SC 29208, USA. Email: \texttt{ju@math.sc.edu}, \texttt{wangzhu@math.sc.edu}.}
\footnotetext[4]{State Key Laboratory of Scientific and Engineering Computing, Chinese Academy of Sciences, Beijing 100190, China. Email: \texttt{wleng@lsec.cc.ac.cn}.}

\renewcommand{\thefootnote}{\arabic{footnote}}

\begin{abstract}
In this paper we present and analyze a general framework for constructing high order explicit local time 
stepping (LTS) methods for hyperbolic conservation laws. In particular, we consider the model problem 
discretized by Runge-Kutta discontinuous Galerkin (RKDG) methods and design LTS algorithms based on 
strong stability preserving Runge-Kutta (SSP-RK) schemes, that allow spatially variable time step sizes to 
be used for time integrations in different regions. The proposed algorithms are of predictor-corrector type, 
in which the interface information along the time direction is first predicted based on the SSP-RK 
approximations and Taylor expansions, and then the fluxes over the region of interface are corrected to 
conserve mass exactly at each time step. Following the proposed framework, we detail the corresponding 
LTS schemes with accuracy up to the fourth order, and prove their conservation property and nonlinear 
stability for the scalar conservation laws. Numerical experiments are also presented to demonstrate 
excellent performance of the proposed LTS algorithms. 
\end{abstract} 


\pagestyle{myheadings}

\section{Introduction} 
%
%
%
%
Numerical methods for hyperbolic conservation laws are a subject of great interest and importance as these laws are extensively used for modeling a wide range of physical phenomena such as gas dynamics, shallow water flow, advection of contaminants, traffic flows, etc. It is well known that these problems are often highly nonlinear and may have discontinuous solutions with sharp and moving fronts/shocks. To obtain  accurate and stable numerical solutions to hyperbolic conservation laws, it is popular to use conservative high resolution methods in space together with explicit time stepping. Examples of such spatial discretization include the MUSCL (monotonic upwind scheme for conservation laws) \cite{Leer79}, the ENO (essentially nonoscillatory) and WENO (weighted ENO) schemes \cite{HO87, HJCP87, Liu94, Shu96}, and the RK-DG (Runge-Kutta discontinuous Galerkin) methods \cite{CSI, CSII, CSIII, CSIV}. Note that to guarantee numerical stability, the time step size needs to satisfy the CFL condition, which is determined by the spatial mesh size and wave speed.  The use of local spatial refinements is efficient in resolving the sharp, moving fronts. However, as the CFL condition needs to hold everywhere, the step size for time integration would be controlled by the smallest cell size, or by the highest wave speed, which certainly increases the computational cost as a small time step size has to be used globally. Thus, to improve computational efficiency, the global CFL condition could be replaced by a local one so that the different time step sizes can be used in different regions: smaller time step sizes where the mesh is fine or the wave speed is high, and larger time step sizes where the mesh is coarse or the wave speed is low. \vspace{1pt}

Explicit local time-stepping (LTS) algorithms have a long tradition. To the best of our knowledge, the first LTS algorithm for hyperbolic conservation laws was introduced in \cite{Osher83} for one-dimensional scalar case based on the forward Euler method in time. It is of predictor-corrector type and is first order accurate in both space and time. Extension to high resolution schemes with slope limiters for advection equations was presented in~\cite{Dawson95}, and to second order in time for hyperbolic conservation laws in~\cite{Dawson01}. The numerical results on two-dimensional test problems confirm that these LTS schemes are very competitive to the global time-stepping with respect to the accuracy in time. The application of LTS schemes to the shallow water equations was investigated in \cite{Sanders08} with a Godunov-type finite volume discretization in space and later in \cite{Dawson12} using the RK-DG finite element methods. Note that the LTS scheme in~\cite{Sanders08} is only first order accurate in time, while the one in \cite{Dawson12} is second order accurate in time on regions away from the LTS interface but its accuracy degrades to first order at the interface.  
The LTS scheme in \cite{Dawson12} is based on the second-order strong stability preserving Runge-Kutta (SSP-RK) method, which is also known as a total variation diminishing (TVD) method introduced in \cite{Shu87a, Shu89}. Higher order RK-based explicit LTS methods were introduced for conservation laws in \cite{Kr10,Amaster} and for wave propagation in \cite{Grote15}. In~\cite{NERK09}, a space-time fully adaptive multi-resolution method based on natural continuous extensions for RK methods was proposed, whose accuracy is  of second order in both space and time. 
Other works related to LTS include the adaptive mesh refinement (AMR) method \cite{Berger84, Berger98}, the multirate time-stepping method \cite{Sandu07,Sandu09} and the Implicit-Explicit (IMEX) based LTS methods \cite{Helmig15, Gupta16}. Among them, the AMR method involves the refinement in both space and time, i.e., small time step sizes are taken on the refined mesh and large time step sizes on the coarse mesh. It is different from our approach in the way that refined grids are placed over regions of the coarse grid and information is exchanged between the grids by means of injection and interpolation. The multirate time-stepping method allows different time step sizes in different regions but it requires buffer regions to accommodate the time scale transition between regions. An overview of LTS techniques over the last two decades can be found in \cite{GH13Review}.\vspace{1pt}

In \cite{HJCP19}, inspired by the first order predictor-corrector scheme in \cite{Osher83}, we have designed  conservative  second and third order explicit LTS algorithms, incorporating with SSP-RK, for the rotating shallow water equations. The model is discretized in space by a C-grid staggering finite volume method, namely the TRiSK scheme~\cite{Todd2009, Todd2010}, on orthogonal primal and dual meshes. Numerical results with parallel implementation show excellent performance of the LTS algorithms in terms of stability, accuracy, efficiency and scalability. 
In this work, we extend the approach to construct, in a systematic way, a framework of high order LTS algorithms for hyperbolic conservative laws. In order to derive high order LTS algorithms, the key idea is to find high order approximations on the interface at intermediate time levels to handle the coupling between coarse and fine time steppings. Our proposed schemes are also of predictor-corrector type: 
we derive the predictors based on Taylor series expansions of the solution at the current time level and the SSP-RK stepping algorithms at each intermediate time level. Our approach thus is different from the one proposed in~\cite{Kr10,Amaster} where the predictors are based on RK time-stepping and interpolating polynomials. We present up to fourth order predictors within this framework, and show that the proposed LTS schemes preserve the accuracy in time over the entire domain. Concerning the corrector, it is designed to balance the fluxes from the regions with small time step sizes to the ones with large time step sizes. As high order SSP-RK methods consist of multiple stages, the fluxes at the same stage are accumulated over all the intermediate time levels to update the interface solution associated with that stage. As a consequence, the total mass is well conserved, though the corrector is no longer convex combinations of forward Euler steps as in the global SSP-RK methods. Nevertheless, we rigorously prove that the proposed LTS schemes for scalar conservation laws are total variation bounded (TVB). Such nonlinear stability is a crucial feature of any effective numerical method for hyperbolic conservation laws because it guarantees that the schemes can capture moving shocks without introducing nonphysical oscillations. Various numerical experiments are carried out to validate the accuracy, conservation and stability of our LTS schemes. Since  time advancement of the simulations in the fine regions and in the coarse ones can be implemented in parallel (this will be discussed further in Section~\ref{sec:LTS}), the proposed LTS schemes preserve the natural parallelism of explicit stepping schemes. 

Consider  the initial value problem for hyperbolic conservation laws: \vspace{-0.2cm}
\begin{equation} \label{eq:model}
\begin{array}{rll}
\frac{\partial \bu }{\partial t} + \sum_{i=1}^{d} \frac{\partial \bff_{i}}{\partial x_{i}}(\bu) &\hspace{-0.3cm}= \pmb{0}, & \text{on} \; \mR^{d} \times (0,T), \\
\bu(\bx, 0) & \hspace{-0.3cm}=\bu_{0}(\bx), & \text{in} \; \mR^{d},
\end{array} \vspace{-0.2cm}
\end{equation}
where $\bu(\bx) := \left (u_{1}(\bx), \ldots, u_{m}(\bx)\right )$ is an $m-$dimensional vector of unknowns and each flux function $ \bff_{i}: \mR^{m}  \longrightarrow \mR^{m}$ defined by
\begin{equation*}
\begin{array}{rcl} 
\bu  &\mapsto &\bff_{i}(\bu) := \left (f_{i1}(\bu), \ldots, f_{im}(\bu)\right ), 
\end{array}
\end{equation*}
is vector-valued and is of $m$ components. Since we focus on the time discretization techniques in this paper, we shall only consider the one-dimensional case, $d=1$. In particular, our model problem is the following scalar hyperbolic conservation law, resulting from  the system~\eqref{eq:model} with $d=m=1$:
\begin{equation} \label{eq:CL}
\begin{array}{rll}
\frac{\partial u}{\partial t} + \frac{\partial}{\partial x} f(u)& \hspace{-0.3cm}= 0, & \text{on} \; \mR \times (0,T), \vspace{4pt}\\
u(x,0) & = \;u_{0}, & \text{in} \; \mR.
\end{array} 
\end{equation}
We shall construct and analyze high order Runge-Kutta discontinuous Galerkin algorithms with local time-stepping for \eqref{eq:CL}. The proposed LTS algorithms can be straightforwardly extended to the case of one-dimensional systems of conservation laws ($m>1$), which will be presented in the numerical results, as well as to the higher dimensional problems ($d>1$). 

The rest of this paper is structured as follows. In Section~\ref{sec:RK-DG} we briefly introduce the RK-DG methods for scalar conservation laws~\eqref{eq:CL}.
High order LTS algorithms are carefully derived in Section~\ref{sec:LTS}, and their conservation and stability properties are then proved in Section~\ref{sec:anal}. Numerical results for various test cases are given in Section~\ref{sec:NumRe} to demonstrate the performance of the proposed LTS schemes. Additionally, coefficients of the SSP-RK methods used in the paper are given in Appendix~\ref{SSPRKappend}, and detailed derivation of the predictors for the proposed LTS schemes is  presented in Appendix~\ref{PREDappend}.
%
%
%
%
%
%
%
%
%
%
%
%
\section{Runge-Kutta discontinuous Galerkin methods} \label{sec:RK-DG}
We first introduce the RK-DG methods and refer to \cite{CSII} for a complete presentation of the methods. 
Within the framework of RK-DG, we first discretize equation~\eqref{eq:CL} in space by the discontinuous Galerkin method, then integrate it in time by SSP-RK schemes, and finally apply a slope limiter to achieve stable and high order accurate numerical solutions.
\subsection{Spatial discretization by the discontinuous Galerkin}
Assume a partition of the real line $\mR$ to have the $j$-th intervals as $I_{j}=\left (x_{j-\ahalf}, x_{j+\ahalf}\right )$ and define
$\Delta_{j}=x_{j+\ahalf}-x_{j-\ahalf}$ and $h=\max_{j}\Delta_{j}.$ Let $V_{h}$ be the finite dimensional space consisting of discontinuous, piecewise polynomial functions:
\begin{equation*}
V_{h}=V_{h}^{k}=\left \{ v \in L^{1}(\mR): v\mid_{I_{j}} \in \mathcal{P}^{k}(I_{j}), \; \forall\, j\right \} \; \not \subset H^{1}(\mR),
\end{equation*}
where $\mathcal{P}^{k}(I_{j})$ is the space of polynomials of degree at most $k$ on $I_{j}$. Consider a weak formulation of \eqref{eq:CL} obtained from testing it by any function $v_{h} \in V_{h}$ over $I_{j}$: 

\noindent
\hspace{1mm}
\mbox{For a.e. $ t \in (0,T) $, find $ u_{h}(t) \in V_{h}$ such that: $\forall\, j$ and  $\forall\, v_{h} \in V_{h}$}
\begin{equation}  \label{eq:variational}
\begin{array}{l}
\int_{I_{j}} \partial_{t} u_{h} (x,t) \, v_{h}(x) \, dx - \int_{I_{j}} f(u_{h}(x,t)) \,  \partial_{x} v_{h}(x) \, dx \vspace{0.1cm} \\
\hspace{1.5cm} + h(u_{h})_{j+\ahalf}(t)\, v_{h}(x_{j+\ahalf}^{-}) - h(u_{h})_{j-\ahalf}(t)\, v_{h}(x_{j-\ahalf}^{+})  = 0,  \vspace{0.2cm}\\
\int_{I_{j}} u_h(x,0) v_{h}(x) \, dx  = \int_{I_{j}} u_{0}(x) v_{h}(x) \, dx.
\end{array}
\end{equation}
Note that we have replaced the nonlinear flux $f(u(x_{j+\ahalf},t))$ in \eqref{eq:variational} by a \emph{Lipschitz, consistent, monotone numerical flux} $h(u)_{j+\ahalf}(t)$ which depends on the two values of~$u$ at $x_{j+\ahalf}$: 
$$ h(u)_{j+\ahalf}(t)=h\left (u(x_{j+\ahalf}^{-}, t), u(x_{j+\ahalf}^{+}, t)\right ).
$$
The numerical flux $h(\cdot, \cdot)$ is required to satisfy the following properties: 
i) locally Lipschitz continuous; 
ii) consistent with the flux $f$, that is, $h(u,u)=f(u)$; and 
iii) nondecreasing in the first argument and nonincreasing in the second argument. 
Examples of such a flux include the Godunov flux, Engquist-Osher flux, Lax-Friedrichs flux and Roe flux. 

A local orthogonal basis of $V_{h}$ consists of functions $\varphi_{j}^{(l)}$ defined as, for any $j$, 
\begin{equation*}
\varphi_{j}^{(l)} := P_{l}\left (\frac{2(x-x_{j})}{\Delta_{j}} \right ), \quad \text{for } l=0,1, \ldots, k,
\end{equation*}
in which $P_{l}$ is the Legendre polynomial of degree $l$ and $x_{j}$ is the middle point of $I_{j}$. Consequently, the approximate solution $u_{h}$ is expressed uniquely as \vspace{-0.2cm}
\begin{equation} \label{eq:uh}
u_{h}(x,t) = \sum_{l=0}^{k} u_{j}^{(l)}(t) \, \varphi_{j}^{(l)} (x), \quad \text{for} \; x \in I_{j},
\end{equation}
where the degrees of freedom $u_{j}^{(l)}(t)$ are determined by
\begin{equation*}
u_{j}^{(l)}(t):=\frac{2l+1}{\Delta_{j}} \int_{I_{j}} u(x,t) \varphi_{j}^{(l)}(x) \, dx, \quad \text{for } l=0,1, \ldots, k.
\end{equation*}
Note that $u_{j}^{(0)}$ is the cell average of $u$ in $I_{j}$. By taking $v_{h}=\varphi_{j}^{(l)}$ in~\eqref{eq:variational}, we obtain the following ODE for $u_{j}^{(l)}$ for any $j$: 
\begin{equation} \label{eq:ODE}
\begin{array}{l}
\left (\frac{1}{2l+1}\right )\frac{d u_{j}^{(l)}(t)}{dt}-\frac{1}{\Delta_{j}} \int_{I_{j}} f(u_{h}(x,t)) \, \partial_{x} \varphi_{j}^{(l)} \, dx \vspace{0.1cm} \\
\hspace{1cm} +\frac{1}{\Delta_{j}} \left [ h(u_{h})_{j+\ahalf}(t) - (-1)^{l} h(u_{h})_{j-\ahalf}(t) \right ] = 0, \quad \forall\, l=0,1, \ldots, k,
\end{array}
\end{equation}
with the initial condition
\begin{equation*}
u_{j}^{(l)}(0)  = \frac{2l+1}{\Delta_{j}} \int_{I_{j}} u_{0}(x) \varphi_{j}^{(l)}(x) \, dx.
\end{equation*}
Note that in \eqref{eq:ODE} we have used the following properties of Legendre polynomials:  
$$\varphi_{j}^{(l)}(x_{j+\ahalf}^{-}) =P_{l}(1)=1, \quad \varphi_{j}^{(l)}(x_{j-\ahalf}^{+})=P_{l}(-1)=(-1)^{l}.  $$ 
The numerical flux $h$ is computed by 
 $h(u_{h})_{j+\ahalf}(t)=h\left (u_{j+\ahalf}^{-}(t), u_{j+\ahalf}^{+}(t)\right ),$
where $u_{j+\ahalf}^{\pm}(t)=u_{h}(x_{j+\ahalf}^{\pm},t)$ are defined by 
\begin{equation*} \label{eq:Uedge}
u_{j+\ahalf}^{-}(t) = \sum_{l=0}^{k} u_{j}^{(l)}, \quad  u_{j-\ahalf}^{+}(t)  = \sum_{l=0}^{k} (-1)^{l} u_{j}^{(l)}. 
\end{equation*}
Approximating the integral in \eqref{eq:ODE} by Gauss-Lobatto quadrature rules that involve the two endpoints of the interval yields (using the definition of $u_{h}$ in \eqref{eq:uh}):
\begin{equation*} \label{eq:GLquad}
\int_{I_{j}} f(u_{h}) \, \partial_{x} \varphi_{j}^{(l)} \, dx = \int_{I_{j}} f\left (u_{j-\ahalf}^{+}, (u_{j}^{(l)})_{0\leq l \leq k}, u_{j+\ahalf}^{-}\right ) \, \partial_{x} \varphi_{j}^{(l)} \, dx. \vspace{-0.2cm}
\end{equation*}
The system of ODEs \eqref{eq:ODE} can be recast in an autonomous form as follows:
\begin{equation} \label{eq:auto}
\frac{d  \bU_{h} }{dt}= \bL_{h}(\bU_{h}), \quad \bU_{h}(0)=\bU_{h0},
\end{equation}
where $\bU_{h}=(\bu_{j})_{\forall\, j}$ with $\bu_{j}=(u_{j}^{(l)})_{l=0,\ldots,k}$, the right hand side \vspace{-0.1cm}
$$\bL_{h}(\bU_{h}) = \left (L_{h,j}^{(l)}(u_{j-\ahalf}^{\pm},\bu_{j}, u_{j+\ahalf}^{\pm})\right )_{\forall\, j, \, l=0,1,\ldots,k}, \vspace{-0.2cm}$$ with
\begin{equation} \label{eq:Lhj}
\begin{array}{ll}
 L_{h,j}^{(l)}(u_{j-\ahalf}^{\pm},\bu_{j}, u_{j+\ahalf}^{\pm})=&\hspace{-0.2cm}\frac{2l+1}{\Delta_{j}} \bigg \{ \int_{I_{j}} f\left (u_{j-\ahalf}^{+}, \bu_{j}, u_{j+\ahalf}^{-}\right ) \, \partial_{x} \varphi_{j}^{(l)} \, dx \vspace{3pt}\\
 &- \left [ h(u_{j+\ahalf}^{-},u_{j+\ahalf}^{+}) - (-1)^{l}  h(u_{j-\ahalf}^{-},u_{j-\ahalf}^{+})\right ] \bigg \},
\end{array} \hspace{-0.6cm}
\end{equation}
 and the initial data $\bU_{h0}=\Big[\sfrac{(2l+1)}{\Delta_{j}} \int_{I_{j}} u_{0}(x) \varphi_{j}^{(l)}(x) \, dx\Big]_{\forall\, j, \, l=0,1,\ldots,k}$.
Next, we solve \eqref{eq:auto} explicitly in time by the SSP-RK methods \cite{SO88, GS98}. 

\subsection{Strong stability preserving Runge-Kutta time discretization} 
The SSP-RK methods have been proved to be effective for solving hyperbolic conservation laws with discontinuous solutions. Given a uniform partition of $(0,T)$, 
$0=t_{0} < t_{1} < \ldots < t_{N-1} < t_{N}=T$, 
with the time step size $\textstyle \Delta t = \sfrac{T}{N}$. The $s$-stage, $r$th-order SSP-RK methods, referred to as SSP-RK$(s,r)$, for solving the autonomous system~\eqref{eq:auto} read as follows:  
for $n=0, \ldots, N-1$, compute  \vspace{-0.1cm}
\begin{equation} \label{eq:TVDRK}
\bU_{h}^{n, (i)}=\sum_{\nu=0}^{i-1} \alpha_{i\nu} \bU_{h}^{n, (\nu)} + \beta_{i\nu} \Delta t \bL_{h}(\bU_{h}^{n, (\nu)}), \quad \forall\, i=1, \ldots, s,
\end{equation}
where $\bU_{h}^{n, (0)}=\bU_{h}^{n}$ and set  
$\bU_{h}^{n+1}=\bU_{h}^{n, (s)}$.  
It is required that all the weights  $\alpha_{i\nu}, \beta_{i\nu} \geq 0$.
To measure stability of RK-DG methods, we denote by $\bu^{(l),n} = (u_{j}^{(l),n})_{\forall j}$ and define the total variation of numerical solutions by 
\begin{equation*}
TV(\bu^{(l),n})=\sum_{j} \left \vert u_{j+1}^{(l),n} - u_{j}^{(l),n} \right \vert, \quad \forall\,  l=0, \ldots, k, \; \text{and} \; n=0, \ldots, N-1.
\end{equation*}
A numerical method is {\em total variation diminishing (TVD)} if  
\begin{equation*}
TV(\bu^{(l), n+1}) \leq  TV(\bu^{(l),n}), \quad \forall\,  l=0, \ldots, k \; \text{and} \; n=0, \ldots, N-1, 
\end{equation*}
and is {\em total variation bounded (TVB)} if 
\begin{equation*}
TV(\bu^{(l), n+1}) \leq  TV(\bu^{(l),0}) + BT, \quad \forall\,  l=0, \ldots, k \; \text{and} \; n=0, \ldots, N-1,
\end{equation*}
for some constant $B$ independent of the time step size. The stability of the SSP-RK schemes is given by the following lemma. 
\begin{lemma}[\cite{GS98}] If the forward Euler method $\; \bU_{h}^{n+1}=\bU_{h}^{n} + \Delta t \bL_{h}(\bU_{h}^{n})$
is TVD under the CFL condition $\Delta t \leq \Delta t_{FE} $, then the SSP-RK$(s,r)$ scheme \eqref{eq:TVDRK} is TVD under the modified CFL condition:
$\Delta t\leq \iC \Delta t_{FE}, $
where $\iC:= \min_{i,\nu}\frac{\alpha_{i\nu}}{ \beta_{i\nu} }$ is the SSP coefficient.
\end{lemma} 
We present  some commonly used SSP-RK schemes such as SSP-RK(2,2), SSP-RK(3,3) and SSP-RK(5,4)  in detail in Appendix~\ref{SSPRKappend}.

\subsection{TVB corrected slope limiter} \label{subsec:limiter}
In order to handle moving shocks while preserving high order accuracy in smooth regions, we follow \cite{Shu87a} and define the TVB corrected \emph{minmod} function $\widetilde{m}$: \vspace{-0.1cm}
\begin{equation*} \label{eq:mminmod}
\widetilde{m}(a_{1}, \ldots, a_{\nu}) = \left \{ \begin{array}{ll} a_{1}, & \text{if} \; \vert a_{1} \vert \leq C_M h^{2}, \\
m(a_{1}, \ldots, a_{\nu}), & \text{otherwise}. 
\end{array} \right . 
\end{equation*}
where $C_M>0$ is a constant and $m$ is the usual \emph{minmod} function \cite{H84}: 
\begin{equation} \label{eq:minmod}
m(a_{1}, \ldots, a_{\nu}) = s \min_{1 \leq i \leq \nu} \vert a_{i} \vert, \;\; { \text{ with  } } s= \left \{ \begin{array}{ll} \text{sign}(a_1), & \text{if} \;  \text{sign}(a_{1})= \ldots  = \text{sign}(a_{\nu}), \\
0, & \text{otherwise}. 
\end{array} \right . 
\end{equation}
The corrected limiter leads to high order accuracy in any region where the solution is smooth, even at local extrema. The resulting scheme is no longer TVD, instead it is TVB.  Next, we define the $(k+1)$th-order limiter $\Lambda \Pi^{k}_{h}$ as in \cite{C01}. 
When $k=1$, we have  
\begin{equation*}
\Lambda \Pi^{1}_{h} (u_{h})\vert_{I_{j}} = u_{j}^{(0)} + \widetilde{m}(u_{j}^{(1)}, u_{j+1}^{(0)}-u_{j}^{(0)}, u_{j}^{(0)}-u_{j-1}^{(0)}) \varphi_{j}^{(1)}(x).  
\end{equation*}
For $k>1$, we first compute  
\begin{equation*}
\begin{array}{ll}
u_{j+\ahalf}^{- \text{(mod)}} &\hspace{-0.2cm} = u_{j}^{(0)} + \widetilde{m}(u_{j+\ahalf}^{-}-u_{j}^{(0)}, u_{j+1}^{(0)}-u_{j}^{(0)}, u_{j}^{(0)}-u_{j-1}^{(0)}), \vspace{3pt}\\
u_{j-\ahalf}^{+ \text{(mod)}} &\hspace{-0.2cm} = u_{j}^{(0)} - \widetilde{m}(u_{j}^{(0)}-u_{j-\ahalf}^{+}, u_{j+1}^{(0)}-u_{j}^{(0)}, u_{j}^{(0)}-u_{j-1}^{(0)}),
\end{array} \vspace{-0.1cm}
\end{equation*}
then define
\begin{equation*}
\Lambda \Pi^{k}_{h} (u_{h})\vert_{I_{j}} = \left \{ \begin{array}{ll} u_{h}\vert_{I_{j}}, & \text{if $u_{j+\ahalf}^{- \text{(mod)}} = u_{j+\ahalf}^{-}$ and $u_{j-\ahalf}^{+ \text{(mod)}}=u_{j-\ahalf}^{+}$}, \vspace{0.2cm}\\
\Lambda \Pi^{1}_{h} (u_{h})\vert_{I_{j}}, & \text{otherwise}. 
\end{array} \right .
\end{equation*}
We finally make the following notation  \vspace{-0.2cm}
$$u_{h}^{\text{(mod)}} \vert_{I_{j}} := \Lambda \Pi^{k}_{h} (u_{h})\vert_{I_{j}} = \sum_{l=0}^{k} u_{j}^{(l) \text{(mod)}} \varphi_{l},$$
and \vspace{-0.1cm}
$$ \Lambda \Pi_{h}^{k}(\bU_{h}) := \bU_{h}^{\text{(mod)}} = \left [u_{j}^{(l) \text{(mod)}}\right ]_{\forall\, j, \, l=0,1, \ldots, k}. $$
The complete RK-DG method with the TVB {\em minmod} limiter is given in Algorithm~\ref{algo}, in which $r=(k+1)$ to match the accuracy in space and in time, and $s \geq r$ is the number of stages in SSP-RK.
\begin{algorithm}\small
		\caption{Runge-Kutta local projection discontinuous Galerkin method}  \label{algo}
		\begin{algorithmic}[1] 
			\item Compute $\bU_{h}^{0 \text{(mod)}} = \Lambda \Pi_{h}^{k}(\bU_{h0})$.
			\item For each $n=0, 1, \ldots, N-1,$ 
			\begin{enumerate}
			\item Set $\bU_{h}^{n, (0) \text{(mod)}} = \bU_{h}^{n \text{(mod)}}$.
			\item For $i=1, \ldots, s$, compute the solution at stage $i$: 
			\begin{equation*}
			\bU_{h}^{n, (i) \text{(mod)}} = \Lambda \Pi_{h}^{k} \left (\sum_{\nu=0}^{i-1} \alpha_{i\nu} \bU_{h}^{n, (\nu) \text{(mod)}} + \beta_{i\nu} \Delta t \bL_{h}\left(\bU_{h}^{n, (\nu) \text{(mod)}}\right) \right ).
			\end{equation*}
			\item Set $\bU_{h}^{n+1 \text{(mod)}} = \bU_{h}^{n, (s) \text{(mod)}}$. 
			\end{enumerate}
		\end{algorithmic} 
	\end{algorithm} \vspace{-0.2cm}
%
%
%
%
%
%
%
%
%
%
%
%
\section{Local time stepping algorithms} \label{sec:LTS}
In this section, we present high order LTS algorithms incorporated with the RK-DG methods for conservation laws. Given the solution $\bU_{h}^{n \text{(mod)}}$ at $t^{n}$, possibly with moving shocks, we approximate the solution at $t^{n+1}$. To this end, we divide the domain into coarse and fine regions, and assume shocks only appear in the fine regions. This could be made possible by varying the LTS interfaces with time. Consequently, we can use spatially variable time steps: large step sizes in the coarse regions and small step sizes in the fine regions. 

For simplicity of presentation, we decompose the domain into a coarse region $\Omega_{c}^{n}$ and a fine region $\Omega_{f}^{n}$. Extension to more complicated configurations with multiple subdomains is straightforward. 
Denoted by $x_{j_{0}^{n}+\ahalf}$ the interface point at $t^{n}$, $\Omega_{c}^{n} = \left \{ I_{j}: \, j \leq j_{0}^{n} \right \}$ the coarse region, and $\Omega_{f}^{n}=\left \{ I_{j}: \, j \geq j_{0}^{n}+1 \right \}$ the fine region.  
As depicted in Figure~\ref{fig:lts}, we enforce a larger time step $\Delta t_{\text{coarse}}=\Delta t$ in $\Omega_{c}^{n}$ and a smaller time step $\Delta t_{\text{fine}}=\sfrac{\Delta t}{M}$ in $\Omega_{f}^{n}$. We remark that the coarse time increment must be a union of fine time increments: 
$$\left [t^{n},t^{n+1}\right )=\bigcup_{p=0}^{M-1} \left [t^{n,p},t^{n,p+1}\right ). \vspace{-0.1cm}$$
\begin{figure}[ht!]
\centering 
\includegraphics[scale=0.6]{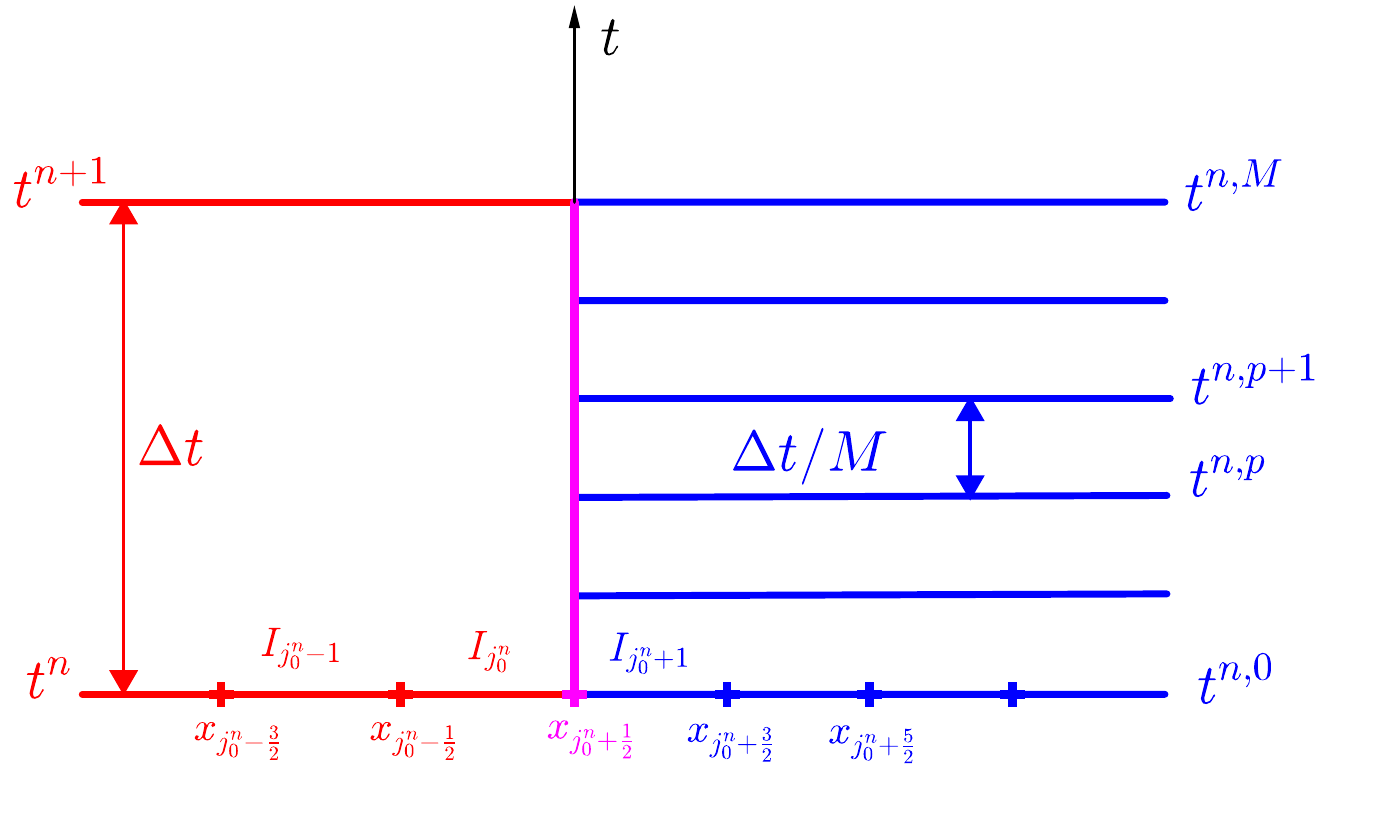} \vspace{-0.2cm}
\caption{Partition in time with local time-stepping.} \label{fig:lts}  \vspace{-0.4cm}
\end{figure}

To proceed in time in the fine region, one needs to find $(k+1)$th-order in time approximation of the flux at the interface at intermediate time levels $t^{n,p}$ for $p=1, \ldots, M-1$. This is obtained via a predictor based on $k$th-order Taylor expansions and the $(k+1)$th-order SSP-RK algorithm, assuming that the solution is smooth enough near the LTS interface. After advancing in the fine region to $t^{n+1}$, we will correct the flux at the interface in order to conserve mass exactly. The derivation of the predictors up to fourth order accuracy are presented in Appendix~\ref{PREDappend}. 
The proposed  LTS algorithm of order $(k+1)$ consists of the following three steps: 
%
%
%
%

\indent \textbf{Step 1: Predicting the interface values.} We first compute the solution of the first $(s-1)$ stages of the SSP-RK$(s,k+1)$ scheme on the interface cell $I_{j_{0}^{n}}$ with a coarse time step: 
\begin{equation*} \label{eq:predUcoarse}
\bu_{j_{0}^{n}}^{n,(i)}=\left (u_{j_{0}^{n}}^{(l), n, (i)}\right )_{\forall\, j, \, l=0, \ldots, k}, \quad  \forall\, i=1, \ldots, s-1. \vspace{-0.1cm}
\end{equation*}
It is important to note that we compute $\bu_{j_{0}^{n}}^{n,(i)}$ {\em locally} on $I_{j_{0}}$ by enforcing $u_{j_{0}^{n}-\ahalf}^{n,(i),-}=u_{j_{0}^{n}-\ahalf}^{n,(i),+}$ and $u_{j_{0}^{n}+\ahalf}^{n,(i),+}=u_{j_{0}^{n}+\ahalf}^{n,(i),-}$ in \eqref{eq:Lhj}. This is obtained under the assumption that the solution near the LTS interface is continuous (for $k \leq 1$) or sufficiently smooth (for $k>1$). Thus, limiter is not necessary in this case and we have
$\bu_{j_{0}^{n}}^{n,(i) \text{(mod)}}=\bu_{j_{0}^{n}}^{n,(i)}.$
We then use these values to predict the solution on the interface $x_{j_{0}^{n}+\ahalf}$ at intermediate time levels $t^{n,p}$: 
\begin{equation} \label{eq:predUedge}
u_{j_{0}^{n}+\ahalf}^{n,p,(i), - \text{(mod)}}=u_{j_{0}^{n}+\ahalf}^{n,p,(i), - } = \sum_{l=0}^{k} u_{j_{0}^{n}}^{(l),n,p,(i)}, \quad \forall\, p=0,1, \ldots, M-1, \vspace{-0.1cm}
 \end{equation} 
where $u_{j_{0}^{n}}^{(l), n,p,(i)}$ are computed by the formulas in Appendix~\ref{PREDappend}. In particular:\vspace{3pt}

\emph{For second order SSP-RK(2,2)}: \vspace{-0.1cm}
\begin{equation} \label{eq:2ndpredictor}
\begin{array}{ll}
u_{j_{0}^{n}}^{(l),n,p,(0)} = (1-\theta_{p}) u_{j_{0}^{n}}^{(l),n, (0)} + \theta_{p} u_{j_{0}^{n}}^{(l),n, (1)}, \vspace{0.1cm}\\
u_{j_{0}^{n}}^{(l),n,p, (1)} = (1-\eta_{p}) u_{j_{0}^{n}}^{(l),n, (0)} + \eta_{p} u_{j_{0}^{n}}^{(l),n, (1)},
\end{array} \vspace{-0.1cm}
\end{equation}
for $l=0,1,$ where $\textstyle \theta_{p}=\frac{p}{M}$ and $\textstyle \eta_{p}=\frac{p+1}{M}$ for $p=0, 1, \ldots, M-1$. \vspace{3pt}

\emph{For third order SSP-RK(3,3)}: 
\begin{equation}
\begin{array}{ll}
u_{j_{0}^{n}}^{(l),n,p, (0)} = (1-\theta_{p}-\widehat{\theta}_{p}) u_{j_{0}^{n}}^{(l),n, (0)} + (\theta_{p}-\widehat{\theta}_{p}) u_{j_{0}^{n}}^{(l),n, (1)} +2\widehat{\theta}_{p} u_{j_{0}^{n}}^{(l),n, (2)} , \vspace{0.1cm}\\
u_{j_{0}^{n}}^{(l),n,p, (1)} = (1-\eta_{p}-\widehat{\eta}_{p}) u_{j_{0}^{n}}^{(l),n, (0)} + (\eta_{p}-\widehat{\eta}_{p}) u_{j_{0}^{n}}^{(l),n, (1)} +2\widehat{\eta}_{p} u_{j_{0}^{n}}^{(l),n, (2)}, \vspace{0.1cm}\\
u_{j_{0}^{n}}^{(l),n,p, (2)} = (1-\gamma_{p}-\widehat{\gamma}_{p}) u_{j_{0}^{n}}^{(l),n, (0)} + (\gamma_{p}-\widehat{\gamma}_{p}) u_{j_{0}^{n}}^{(l),n, (1)}+2\widehat{\gamma}_{p} u_{j_{0}^{n}}^{(l),n, (2)},
\end{array} 
\end{equation}
for $l=0,1,2,$ with $\theta_{p}$ and $\eta_{p}$ as above, and $\textstyle\widehat{\theta}_{p}=\frac{p^{2}}{M^{2}}, \, \widehat{\eta}_{p}=\frac{p(p+2)}{M^{2}}, \; \gamma_{p}= \frac{2p+1}{2M} $ and $\textstyle\widehat{\gamma}_{p}=\frac{2p^{2}+2p+1}{2M^{2}}$ for $p=0, 1, \ldots, M-1$. \vspace{3pt}

\emph{For fourth order SSP-RK(5,4)}: we approximate $u_{j_{0}^{n}}^{(l),n,p, (i)}$, for $p=0, 1, \ldots, M-1$ as linear combinations of $u_{j_{0}^{n}}^{(l),n,(i)}$ for $l,i=0,\ldots, 4,$ as presented in Appendix~\ref{subsec:pRK4}. \vspace{0.3cm}\\
%
%
%
%
\indent \textbf{Step 2: Advancing in the coarse and fine regions in parallel.} \vspace{0.1cm}\\
Step 2a). Advancing the coarse region excluding the interface cell: with the solution at the current time level, we advance solution to the next time level by running the SSP-RK with a coarse time step. 
\vspace{0.1cm}

\noindent  For all the cells $I_j$ with $j < j_{0}^{n}$, we perform:  
\begin{enumerate}
\item For $i = 1, \hdots, s,$
\begin{equation} \label{eq:coarse}
\begin{array}{ll}
u_{j}^{(l),n, (i)} &\hspace{-0.2cm} =\sum_{\nu=0}^{i-1} \alpha_{i\nu} u_{j}^{(l),n,(\nu)} + \beta_{i\nu} \Delta t \, L_{h,j}^{(l)}\left (u_{j-\ahalf}^{n,(\nu),\pm},\bu_{j}^{n,(\nu)}, u_{j+\ahalf}^{n,(\nu),\pm}\right ).
\end{array}
\end{equation}
\item Set $\bu_{j}^{n+1}=\bu_{j}^{n, (s)}$ for all $j<j_{0}^{n}$. 
\end{enumerate}
%
%
{\em Step 2b). Advancing in the fine region:} with the predicted values on the interface, we evaluate the interface flux $h(u_{j_{0}^{n}+\ahalf}^{n,p,-},u_{j_{0}^{n}+\ahalf}^{n,p,+})$ at the intermediate time levels, and consequently obtain the solution $u_{j}^{n,p}$ for all the cells $I_j$ with $j > j_{0}^{n}$ in the fine region. The TVB limiter is performed to obtain $u_{j}^{n,p \text{(mod)}}$ for $j>j_{0}^{n}$ and the predicted values are updated on the interface after limiting. 
\vspace{0.1cm}

\noindent For all the cells $I_j$ with $j \geq j_{0}^{n}+1$, we perform: \vspace{0.1cm}

For $p=0, \ldots, M-1$, \vspace{0.1cm}
\begin{enumerate}
\item Set $u_{j}^{(l),n,p, (0) (mod)} = u_{j}^{(l),n,p (mod)} $, for $l=0, \ldots, k$.\vspace{0.2cm}
\item For $i=1, \ldots, s$,  we compute the solution at stage $i$: \vspace{-0.2cm}
\begin{equation*} \label{eq:fine}
\begin{array}{ll}
\hspace{1cm} u_{j}^{(l),n,p, (i)}\;=&\sum_{\nu=0}^{i-1} \alpha_{i\nu}u_{j}^{(l),n,p, (\nu) \text{(mod)}} \vspace{3pt}\\
& \hspace{-0.3cm}+ \beta_{i\nu} \left (\frac{\Delta t}{M}\right ) L_{h,j}^{(l)}\left (u_{j-\ahalf}^{n,p,(\nu),\pm \text{(mod)}},\bu_{j}^{n,p,(\nu)\text{(mod)}}, u_{j+\ahalf}^{n,p,(\nu),\pm \text{(mod)}}\right ), 
\end{array}
\end{equation*}
for $l=0, \ldots, k$.
If $p<M-1$, limit the solution in the fine region 
\begin{equation*}
\bu_{j}^{n,p, (i) \text{(mod)}} = \Lambda \Pi_{h}^{k}\left (\bu_{j^{\prime}\geq j_{0}^{n}}^{n,p, (i)}\right )\vert_{I_{j}},
\end{equation*} 
and update the predicted interface value $u_{j_{0}^{n}+\ahalf}^{n,p,(i), - \text{(mod)}}$ (cf. \eqref{eq:predUedge}) after limiting:
\begin{equation*}
\begin{array}{ll}
u_{j_{0}^{n}+\ahalf}^{n,p,(i), - \text{(mod)}}  \;= &u_{j_{0}^{n}}^{(0),n,p, (i)} + \widetilde{m}\Big(u_{j_{0}^{n}+\ahalf}^{n,p,(i), - \text{(mod)}}-u_{j_{0}^{n}}^{(0),n,p,(i)}, \vspace{3pt}\\
& \hspace{0.5cm}  u_{j_{0}^{n}+1}^{(0), n,p,(i)}-u_{j_{0}^{n}}^{(0),n,p,(i)}, u_{j_{0}^{n}}^{(0),n,p,(i)}-u_{j_{0}^{n}-1}^{(0),n, (i)}\Big).
\end{array} 
\end{equation*}

\item For all $j>j_{0}^{n}$, set:
\begin{equation*}
\left \{ \begin{array}{rll}
u_{j}^{(l),n,p+1 \text{(mod)}}& = u_{j}^{(l), n,p, (s) \text{(mod)}}, & \text{if} \; p<M-1, \vspace{3pt}\\
u_{j}^{(l),n+1}& = u_{j}^{(l), n,p, (s)}, & \text{if} \; p=M-1.
\end{array} \right .
\end{equation*}
\end{enumerate}   \vspace{0.2cm}
%
%
%
%

 \textbf{Step 3: Correcting the interface solution and limiting the global solution at $t^{n+1}$ locally.} 
With the predicted interface value $u_{j_{0}^{n}+\ahalf}^{n,p,(\nu),-}$, we calculate the flux at the interface $ x=x_{j_{0}^{n}+\ahalf}$. Together with the flux at $x=x_{j_{0}^{n}-\ahalf}$, which is frozen over $[t^{n}, t^{n+1})$, we correct the solution of the interface cell $I_{j_0^n}$. Finally, a TVB limiter is applied, which can be implemented in parallel as \cite{CS01JSC}, to limit the solution on $I_j, \, \forall j$ in which only information on elements sharing edges with $I_j$ is necessary.  
\begin{enumerate}
\item For $i=1, \ldots, s$, we compute the solution at stage $i$ at the interface:
\begin{equation} \label{eq:corrIF}
\begin{array}{ll}
\widehat{u}_{j_{0}^{n}}^{(l),n,(i)}\;=&\sum_{\nu=0}^{i-1} \alpha_{i\nu} \widehat{u}_{j_{0}^{n}}^{(l),n, (\nu)} \vspace{3pt}\\
& + \beta_{i\nu} \frac{\Delta t}{M}  \sum_{p=0}^{M-1} L_{h,j_{0}^{n}}^{(l)}\left (u_{j_{0}^{n}-\ahalf}^{n,(\nu),\pm \text{(mod)}},\bu_{j_{0}^{n}}^{n,(\nu)\text{(mod)}}, u_{j_{0}^{n}+\ahalf}^{n,p,(\nu),\pm \text{(mod)}}\right ),
\end{array} \hspace{-0.8cm}
\end{equation}
where $\widehat{u}_{j_{0}^{n}}^{(l),n, (0)}=u_{j_{0}^{n}}^{(l),n, (0) \text{(mod)}}$.
 \vspace{0.2cm}
\item Set $u_{j_{0}^{n}}^{(l),n+1}=\widehat{u}_{j_{0}^{n}}^{(l),n,(s)}$ and perform the limiter:
$\bU_{h}^{n+1 \text{(mod)}}= \Lambda\Pi_{h}^{k}(\bU_{h}^{n+1}).$
\end{enumerate}
%
%
%
%
%

\section{Properties of LTS schemes} \label{sec:anal}
First, we notice that the proposed LTS schemes preserve the accuracy in time of the corresponding global SSP-RK methods due to the construction of the predictor and the corrector (see also Remark~\ref{rmrk:pred}). In the following, we prove that the LTS schemes conserve mass exactly and importantly, they satisfy the TVB stability. 
\subsection{Conservation}
Mass conservation of the proposed LTS schemes is obtained via the construction of the corrector. For simplicity, we assume that the solutions are obtained after performing the limiter defined in Subsection~\ref{subsec:limiter} and write $u_{j}^{(l),n}$ for $u_{j}^{(l), n \text{(mod)}}$. 
\begin{theorem} \label{lmm:conv}
The LTS schemes exhibit exact conservation of mass:
\begin{equation*}
\int_{\mR} u_{h}^{n+1} = \int_{\mR} u_{h}^{n}, \quad \forall\, n=0, \ldots, N-1.
\end{equation*}
\end{theorem}
\begin{proof}
We only need to show that mass is conserved in the region of the LTS interface $x=x_{j_{0}^{n}+\ahalf}$, $I_{j_{0}^{n}} \cup I_{j_{0}^{n}+1}$, under the assumption that no flux is imposed at $x_{j_{0}^{n}-\ahalf}$ and $x_{j_{0}^{n}+\sfrac{3}{2}}$:
\begin{equation} \label{eq:conv}
\int_{I_{j_{0}^{n}} \cup I_{j_{0}^{n}+1}} u_{h}^{n+1} = \int_{I_{j_{0}^{n}} \cup I_{j_{0}^{n}+1}} u_{h}^{n}.
\end{equation}
Next, we prove \eqref{eq:conv} for the second order LTS scheme based on SSP-RK$(2,2)$ (cf. Equations~\eqref{eq:SSPRK2}). The proof for the third and fourth order LTS schemes can be done in a similar manner; in fact, the result holds for any high order LTS schemes with the corrector defined by~\eqref{eq:corrIF}. \vspace{0.2cm}\\
For the fine cell $I_{j_{0}^{n}+1}$, the second order LTS algorithm reads:
\begin{equation*}
\begin{array}{ll}
u_{j_{0}^{n}+1}^{(l), n+1} &\hspace{-0.2cm}= \frac{1}{2} u_{j_{0}^{n}+1}^{(l),n,M-1} + \frac{1}{2} \bigg [u_{j_{0}^{n}+1}^{(l),n,M-1,(1)} \vspace{3pt}\\
& \hspace{1cm} +  \left (\frac{\Delta t}{M}\right ) L_{h,j_{0}^{n}+1}^{(l)}\left (u_{j_{0}^{n}+\ahalf}^{n,M-1,(1),\pm},\bu_{j_{0}^{n}+1}^{n,M-1, (1)}, u_{j_{0}^{n}+\sfrac{3}{2}}^{n,M-1,(1),\pm}\right )\bigg ] \vspace{3pt}\\
&\hspace{-0.2cm}=u_{j_{0}^{n}+1}^{(l),n,M-1} + \frac{1}{2} \frac{\Delta t}{M} \sum_{\nu=0}^{1} L_{h,j_{0}^{n}+1}^{(l)}\left (u_{j_{0}^{n}+\ahalf}^{n,M-1,(\nu),\pm},\bu_{j_{0}^{n}+1}^{n,M-1, (\nu)}, u_{j_{0}^{n}+\sfrac{3}{2}}^{n,M-1,(\nu),\pm}\right ).
\end{array}
\end{equation*}
Thus, by recursion, we obtain:
\begin{equation} \label{eq:conv1}
u_{j_{0}^{n}+1}^{(l), n+1}=u_{j_{0}^{n}+1}^{(l),n} + \frac{1}{2} \frac{\Delta t}{M}\sum_{p=0}^{M-1} \sum_{\nu=0}^{1} L_{h,j_{0}^{n}+1}^{(l)}\left (u_{j_{0}^{n}+\ahalf}^{n,p,(\nu),\pm},\bu_{j_{0}^{n}+1}^{n,p, (\nu)}, u_{j_{0}^{n}+\sfrac{3}{2}}^{n,p,(\nu),\pm}\right ).
\end{equation}
Taking $v_{h}=1$ in \eqref{eq:variational}, using \eqref{eq:conv1} and the definition of $L_{h, j}^{(l)}$ in \eqref{eq:Lhj}, we have
\begin{equation} \label{eq:conv_fineb}
\int_{I_{j_{0}^{n}+1}} u_{h}^{n+1} = \int_{I_{j_{0}^{n}+1}} u_{h}^{n} + \frac{1}{2} \frac{\Delta t}{M}\sum_{p=0}^{M-1} \left ( - h(u_{j_{0}^{n}+\ahalf}^{n,p,-}, u_{j_{0}^{n}+\ahalf}^{n,p,+})-  h(u_{j_{0}^{n}+\ahalf}^{n,p,(1),-}, u_{j_{0}^{n}+\ahalf}^{n,p,(1),+})\right ).
\end{equation}
as no flux is imposed at $x_{j_{0}^{n}+\sfrac{3}{2}}$. \vspace{2pt}

For the interface cell $I_{j_{0}^{n}}$, the corrector \eqref{eq:corrIF} associated with SSP-RK$(2,2)$ is given by
\begin{equation*}
\begin{array}{ll}
\widehat{u}_{j_{0}^{n}}^{(l), n,(1)}=u_{j_{0}^{n}}^{(l),n} +  \frac{\Delta t}{M}\sum_{p=0}^{M-1}L_{h,j_{0}^{n}}^{(l)}\left (u_{j_{0}^{n}-\ahalf}^{n,\pm },\bu_{j_{0}^{n}}^{n}, u_{j_{0}^{n}+\ahalf}^{n,p,\pm}\right ), \vspace{3pt}\\
u_{j_{0}^{n}}^{(l), n+1}=\frac{1}{2}u_{j_{0}^{n}}^{(l),n}  + \frac{1}{2} \left [ \widehat{u}_{j_{0}^{n}}^{(l), n,(1)} + \frac{\Delta t}{M}\sum_{p=0}^{M-1} L_{h,j_{0}^{n}}^{(l)}\left (u_{j_{0}^{n}-\ahalf}^{n,(1),\pm},\bu_{j_{0}^{n}}^{n,(1)}, u_{j_{0}^{n}+\ahalf}^{n,p,(1),\pm}\right ) \right ],
\end{array}
\end{equation*}
from which we deduce that
\begin{equation} \label{eq:conv2}
u_{j_{0}^{n}}^{(l), n+1}=u_{j_{0}^{n}}^{(l),n} + \frac{1}{2}  \frac{\Delta t}{M}\sum_{p=0}^{M-1} \sum_{\nu=0}^{1}L_{h,j_{0}^{n}}^{(l)}\left (u_{j_{0}^{n}-\ahalf}^{n,(\nu),\pm},\bu_{j_{0}^{n}}^{n,(\nu)}, u_{j_{0}^{n}+\ahalf}^{n,p,(\nu),\pm}\right ).
\end{equation}
As for the fine cell $j=j_{0}^{n}+1$, we choose $v_{h}=1$ in \eqref{eq:variational} and use \eqref{eq:conv2} to obtain
\begin{equation} \label{eq:conv_IF}
\int_{I_{j_{0}^{n}}} u_{h}^{n+1} = \int_{I_{j_{0}^{n}}} u_{h}^{n} + \frac{1}{2} \frac{\Delta t}{M}\sum_{p=0}^{M-1} \left (  h(u_{j_{0}^{n}+\ahalf}^{n,p,-}, u_{j_{0}^{n}+\ahalf}^{n,p,+}) + h(u_{j_{0}^{n}+\ahalf}^{n,p,(1),-}, u_{j_{0}^{n}+\ahalf}^{n,p,(1),+})\right ),
\end{equation}
noting that no flux at $x_{j_{0}^{n}-\ahalf}$ is assumed. 
Thus, the proof is completed by adding \eqref{eq:conv_fineb} and \eqref{eq:conv_IF} together. 
\end{proof}
%
%
\subsection{Stability} Numerical methods for conservation laws need to satisfy certain nonlinear stability requirements in order to prevent spurious oscillations when the solution is discontinuous. In \cite{Osher83}, the first order LTS scheme  based on forward Euler is proved to be TVD with the predictor obtained by freezing the value at $t^{n}$: $$u_{j_{0}^{n}+\ahalf}^{n,p} = u_{j_{0}^{n}+\ahalf}^{n}, \quad \forall p=0, \hdots, M-1.$$ 
For higher order LTS schemes as proposed in Section~\ref{sec:LTS}, multiple stage time-stepping algorithms are employed and the predictors are obtained by taking linear combinations of the interface solution at different stages with a coarse time step size. Therefore, the proof of nonlinear stability for high order LTS schemes is not an obvious generalization from the first order one. Additionally, the corrector designed to conserve mass is not a convex combination of forward Euler steps as in the case of the global SSP-RK. As a consequence, the high order LTS schemes are not TVD anymore, instead they are TVB. 

We next prove the stability of the second order LTS scheme by first showing that it is TVBM (total variation bounded in the means). The generalization to higher order LTS schemes can be done in a similar manner. We introduce some notation to be used in the proof. Denoted by $\Delta_{+}$ and $\Delta_{-}$ the forward and backward finite difference operators, respectively: \vspace{-0.1cm}
\begin{equation*}
\Delta_{+} u_{j}  = u_{j+1} - u_{j}, \; \text{and} \quad \Delta_{-} u_{j}=u_{j}-u_{j-1}. \vspace{-0.1cm}
\end{equation*}
Following \cite{CSII}, we decompose the interface values $u_{j+ \ahalf}^{\pm}$ as  \vspace{-0.1cm}
\begin{equation*}
u_{j+\ahalf}^{-} = \baru_{j} + \widetilde{u}_{j}, \quad u_{j-\ahalf}^{+}=\baru_{j} -\dtildeu_{j}, \vspace{-0.1cm}
\end{equation*}
where $\baru_{j} := u_{j}^{(0)}$ is the mean value of $u$ on the cell $I_{j}$. As in \cite{CSI}, we denote: \vspace{-0.1cm}
\begin{equation} \label{eq:cd}
C_{j+\ahalf}=-h_{2} \cdot \left (1-\frac{\Delta_{+} \dtildeu_{j}}{\Delta_{+} \baru_{j}}\right ), \; \text{and} \quad  D_{j-\ahalf} =h_{1} \cdot \left (1+ \frac{\Delta_{-} \widetilde{u}_{j}}{\Delta_{-} \baru_{j}} \right ),  \vspace{-0.2cm}
\end{equation}
where  \vspace{-0.1cm}
\begin{align*}
h_{1}=\frac{h(u_{j+\ahalf}^{-}, u_{j-\ahalf}^{+}) - h(u_{j-\ahalf}^{-}, u_{j-\ahalf}^{+})}{u_{j+\ahalf}^{-}-u_{j-\ahalf}^{-}}, \quad
h_{2}=\frac{h(u_{j+\ahalf}^{-}, u_{j+\ahalf}^{+})-h(u_{j+\ahalf}^{-}, u_{j-\ahalf}^{+})}{u_{j+\ahalf}^{+}-u_{j-\ahalf}^{+}}. 
\end{align*}
Note that $h_{1}$ and $-h_{2}$ are nonnegative due to the monotonicity of $h(\cdot, \cdot)$. Then the flux associated with the mean value $\baru_{j}$ (cf. Equation~\eqref{eq:Lhj} with $l=0$) can be rewritten equivalently as \vspace{-0.1cm}
\begin{equation*} \label{eq:fluxCD}
-\left (h(u_{j+\ahalf}^{-}, u_{j+\ahalf}^{+})-h(u_{j-\ahalf}^{-}, u_{j-\ahalf}^{+})\right )= C_{j+\ahalf} \, \Delta_{+} \baru_{j} - D_{j-\ahalf} \, \Delta_{-} \baru_{j}. \vspace{-0.1cm}
\end{equation*}
Using the above notation, the second order LTS scheme as presented in Section~\ref{sec:LTS} for the mean value $\baru_{j}$ reads as follows: for $n=0, \ldots, N-1$, 
\begin{enumerate}[leftmargin=*]
\item Compute the predicted mean on the interface cell at the intermediate time levels from the solutions with a coarse time step: for $p=0, \ldots, M-1$,
\begin{equation} \label{eq:meanPred}
\begin{array}{ll}
\baru_{j_{0}^{n}}^{n,p}&\hspace{-0.3cm}=\left (1-\frac{p}{M}\right ) \baru_{j_{0}^{n}}^{n} + \frac{p}{M} \baru_{j_{0}^{n}+\ahalf}^{n,(1)}, \vspace{3pt}\\
\baru_{j_{0}^{n}}^{n,p,(1)}&\hspace{-0.3cm}=\left (1-\frac{p+1}{M}\right ) \baru_{j_{0}^{n}}^{n} + \frac{p+1}{M} \baru_{j_{0}^{n}+\ahalf}^{n,(1)}.
\end{array}
\end{equation}
\item Advance in the coarse region, for all $j<j_{0}^{n}$:
\begin{equation} \label{eq:coarseLTS2}
\begin{array}{ll}
\baru_{j}^{n,(1)}&\hspace{-0.3cm}=\baru_{j}^{n} +\frac{\Delta t}{\Delta x_{j}} \left (C_{j+\ahalf}^{n} \Delta_{+} \baru_{j}^{n} -D_{j-\ahalf}^{n} \Delta_{-} \baru_{j}^{n}\right ),  \vspace{3pt}\\
\baru_{j}^{n+1}&\hspace{-0.3cm}=\frac{1}{2} \baru_{j}^{n} + \frac{1}{2} \left [ \baru_{j}^{n,(1)}+\frac{\Delta t}{\Delta x_{j}} \left (C_{j+\ahalf}^{n,(1)} \Delta_{+} \baru_{j}^{n,(1)} -D_{j-\ahalf}^{n,(1)} \Delta_{-} \baru_{j}^{n,(1)}\right )\right ],
\end{array}
\end{equation}
and in the fine region, for all $j>j_{0}^{n}$: for $p=0, \ldots, M-1$,
\begin{equation} \label{eq:fineLTS2}
\begin{array}{ll}
\baru_{j}^{n,p,(1)}&\hspace{-0.3cm}=\baru_{j}^{n,p} +\frac{\Delta t}{M\Delta x_{j}} \left (C_{j+\ahalf}^{n,p} \Delta_{+} \baru_{j}^{n, p} -D_{j-\ahalf}^{n,p} \Delta_{-} \baru_{j}^{n,p}\right ), \vspace{3pt}\\
\baru_{j}^{n,p+1}&\hspace{-0.3cm}=\frac{1}{2} \baru_{j}^{n,p} + \frac{1}{2} \bigg [\baru_{j}^{n,p,(1)}+\frac{\Delta t}{M\Delta x_{j}} \big  (C_{j+\ahalf}^{n,p,(1)} \Delta_{+} \baru_{j}^{n, p,(1)}  - D_{j-\ahalf}^{n,p,(1)} \Delta_{-} \baru_{j}^{n,p,(1)}\big )\bigg ].
\end{array}
\end{equation}
Note that the interface $u_{j_{0}^{n}+\ahalf}^{n,p,(i),-}$, for $i=1,2,$ and $p=0, \ldots, M-1$ are computed by using the second order predictor~\eqref{eq:meanPred} and \eqref{eq:predUedge}. \vspace{0.2cm}\\
\item Correcting the interface values for which the flux at $x=x_{j_{0}^{n}-\ahalf}$ is frozen over $[t^{n}, t^{n+1})$:
\begin{equation} \label{eq:interfaceLTS2}\small
\begin{array}{ll}
\widehat{\baru}_{j_{0}^{n}}^{n,(1)}&\hspace{-0.4cm}=\baru_{j_{0}^{n}}^{n} -\frac{\Delta t}{M\Delta x_{j_{0}^{n}}} \sum_{p=0}^{M-1} \left (h(u_{j_{0}^{n}+\ahalf}^{n,p,-},u_{j_{0}^{n}+\ahalf}^{n,p,+})-h(u_{j_{0}^{n}-\ahalf}^{n,-},u_{j_{0}^{n}-\ahalf}^{n,+})\right ),  \vspace{3pt}\\
\baru_{j_{0}^{n}}^{n+1}&\hspace{-0.4cm}=\frac{1}{2} \baru_{j_{0}^{n}}^{n} + \frac{1}{2} \bigg [\widehat{\baru}_{j_{0}^{n}}^{n,(1)}-\frac{\Delta t}{M\Delta x_{j_{0}^{n}}} \sum_{p=0}^{M-1} \bigg (h(u_{j_{0}^{n}+\ahalf}^{n,p,(1),-},u_{j_{0}^{n}+\ahalf}^{n,p,(1),+})  -h(u_{j_{0}^{n}-\ahalf}^{n,(1),-},u_{j_{0}^{n}-\ahalf}^{n,(1),+}) \bigg )\bigg ].
\end{array}
\end{equation}
\end{enumerate}

The flux term in the right-hand side of \eqref{eq:interfaceLTS2} can be rewritten as
\begin{equation*} \label{eq:fluxIFh}
\begin{array}{ll}
&-\left (h(u_{j_{0}^{n}+\ahalf}^{n,p,-},u_{j_{0}^{n}+\ahalf}^{n,p,+})-h(u_{j_{0}^{n}-\ahalf}^{n,-},u_{j_{0}^{n}-\ahalf}^{n,+})\right ) = C_{j_{0}^{n}+\ahalf}^{n,p} \Delta_{+} \baru_{j_{0}^{n}}^{n, p} \vspace{3pt}\\
& \hspace{1cm} -D_{j_{0}^{n}-\ahalf}^{n} \Delta_{-} \baru_{j_{0}^{n}}^{n} - \left (h(u_{j_{0}^{n}+\ahalf}^{n,p,-},u_{j_{0}^{n}-\ahalf}^{n,p,+})-h(u_{j_{0}^{n}+\ahalf}^{n,-},u_{j_{0}^{n}-\ahalf}^{n,+})\right ),
\end{array}
 \end{equation*} 
where $u_{j_{0}^{n}-\ahalf}^{n,p,+}$ is computed by the same predictor as $u_{j_{0}^{n}+\ahalf}^{n,p,-}$. In addition, we write
 \begin{equation*}
 \begin{array}{ll}
  &- \left (h(u_{j_{0}^{n}+\ahalf}^{n,p,-},u_{j_{0}^{n}-\ahalf}^{n,p,+})-h(u_{j_{0}^{n}+\ahalf}^{n,-},u_{j_{0}^{n}-\ahalf}^{n,+})\right ) \\ 
  & \qquad= -\ih_{2, j_{0}^{n}-\ahalf}^{n,p} \cdot \left (u_{j_{0}^{n}-\ahalf}^{n,p,+}-u_{j_{0}^{n}-\ahalf}^{n,+}\right ) 
       -\ih_{1, j_{0}^{n}+\ahalf}^{n,p} \cdot \left (u_{j_{0}^{n}+\ahalf}^{n,p,-}-u_{j_{0}^{n}+\ahalf}^{n,-}\right ) 
 \end{array}
 \end{equation*}
 where 
\begin{equation} \label{eq:ih}
\begin{array}{ll}
\ih_{1, j_{0}^{n}+\ahalf}^{n,p}&:= \frac{h(u_{j_{0}^{n}+\ahalf}^{n,p,-},u_{j_{0}^{n}-\ahalf}^{n,+})-h(u_{j_{0}^{n}+\ahalf}^{n,-},u_{j_{0}^{n}-\ahalf}^{n,+})}{u_{j_{0}^{n}+\ahalf}^{n,p,-}-u_{j_{0}^{n}+\ahalf}^{n,-}} \geq 0, \vspace{3pt}\\ 
-\ih_{2, j_{0}^{n}-\ahalf}^{n,p}&:= -\frac{h(u_{j_{0}^{n}+\ahalf}^{n,p,-},u_{j_{0}^{n}-\ahalf}^{n,p,+})-h(u_{j_{0}^{n}+\ahalf}^{n,p,-},u_{j_{0}^{n}-\ahalf}^{n,+})}{u_{j_{0}^{n}-\ahalf}^{n,p,+}-u_{j_{0}^{n}-\ahalf}^{n,+}} \geq 0.
\end{array}
\end{equation}
Moreover, using the second order predictor \eqref{eq:2ndpredictor}, we deduce that
\begin{equation*}\small
u_{j_{0}^{n}-\ahalf}^{n,p,+}-u_{j_{0}^{n}-\ahalf}^{n,+}=\frac{p}{M}\left (u_{j_{0}^{n}-\ahalf}^{n,(1),+}-u_{j_{0}^{n}-\ahalf}^{n,+}\right ),
 \; u_{j_{0}^{n}+\ahalf}^{n,p,-}-u_{j_{0}^{n}+\ahalf}^{n,-}=\frac{p}{M}\left (u_{j_{0}^{n}+\ahalf}^{n,(1),-}-u_{j_{0}^{n}+\ahalf}^{n,-}\right ).
\end{equation*}
Therefore, we can rewrite the correction \eqref{eq:interfaceLTS2} as follows:
\begin{equation} \label{eq:IFcorr1}
\begin{array}{ll}
\widehat{\baru}_{j_{0}^{n}}^{n,(1)}&\hspace{-0.2cm}=\baru_{j_{0}^{n}}^{n} +\frac{\Delta t}{M\Delta x_{j_{0}^{n}}} \sum_{p=0}^{M-1} \bigg (C_{j_{0}^{n}+\ahalf}^{n,p}\Delta_{+} \baru_{j_{0}^{n}}^{n,p}-D_{j_{0}^{n}-\ahalf}^{n} \Delta_{-} \baru_{j_{0}^{n}}^{n}\vspace{3pt}\\ 
&+\frac{p}{M}( -\ih_{2, j_{0}^{n}-\ahalf}^{n,p})\left (u_{j_{0}^{n}-\ahalf}^{n,(1),+}-u_{j_{0}^{n}-\ahalf}^{n,+}\right )-\frac{p}{M} \ih_{1, j_{0}^{n}+\ahalf}^{n,p}\left (u_{j_{0}^{n}+\ahalf}^{n,(1),+}-u_{j_{0}^{n}+\ahalf}^{n,+}\right )\bigg ).
\end{array}
\end{equation}
Similarly, 
\begin{equation} \label{eq:IFcorr2}
\begin{array}{ll}
\baru_{j_{0}^{n}}^{n+1}&\hspace{-0.2cm}=\frac{1}{2} \baru_{j_{0}^{n}}^{n} + \frac{1}{2} \bigg [\widehat{\baru}_{j_{0}^{n}}^{n,(1)}+\frac{\Delta t}{M\Delta x_{j_{0}^{n}}} \sum_{p=0}^{M-1} \bigg (C_{j_{0}^{n}+\ahalf}^{n,p,(1)}\Delta_{+} \baru_{j_{0}^{n}}^{n,p,(1)}-D_{j_{0}^{n}-\ahalf}^{n,(1)} \Delta_{-} \baru_{j_{0}^{n}}^{n,(1)} \vspace{3pt}\\
&\hspace{-0.6cm}+\frac{p+1}{M}( -\ih_{2, j_{0}^{n}-\ahalf}^{n,p,(1)})\left (u_{j_{0}^{n}-\ahalf}^{n,(1),+}-u_{j_{0}^{n}-\ahalf}^{n,+}\right )-\frac{p+1}{M}\ih_{1, j_{0}^{n}+\ahalf}^{n,p,(1)}\left (u_{j_{0}^{n}+\ahalf}^{n,(1),+}-u_{j_{0}^{n}+\ahalf}^{n,+}\right )\bigg )\bigg ],
\end{array} \vspace{0.1cm}
\end{equation}
in which $\ih_{1, j_{0}^{n}+\ahalf}^{n,p,(1)}$ and $\ih_{2, j_{0}^{n}-\ahalf}^{n,p,(1)}$ are defined in a similar way as in \eqref{eq:ih} but with the solutions of the first stage $u_{j_{0}^{n}\pm \ahalf}^{n,p,(1), \pm}$ and $u_{j_{0}^{n}\pm \ahalf}^{n,(1), \pm}$.
The TVBM property of the second order LTS scheme is guaranteed by the following theorem.  \vspace{0.1cm}

\begin{theorem}[TVBM]
Assume that there exists some $\theta >0$ such that 
\begin{equation} \label{eq:TVDlimiter}
\begin{array}{llll}
-\theta \leq \frac{\Delta_{+} \dtildeu_{j}^{n,(i)}}{\Delta_{+} \baru_{j}^{n,(i)}} \leq 1, &  \forall\, j < j_{0}^{n}, \quad &
-\theta \leq \frac{\Delta_{+} \dtildeu_{j}^{n,p, (i)}}{\Delta_{+} \baru_{j}^{n,p, (i)}} \leq 1, & \forall\, j \geq j_{0}^{n}, \vspace{3pt}\\
 -\theta \leq -\frac{\Delta_{+} \widetilde{u}_{j}^{n,(i)}}{\Delta_{+} \baru_{j}^{n,(i)}} \leq 1, & \forall\, j < j_{0}^{n},\quad & -\theta \leq -\frac{\Delta_{+} \widetilde{u}_{j}^{n,p, (i)}}{\Delta_{+} \baru_{j}^{n,p, (i)}} \leq 1, & \forall\, j\geq j_{0}^{n},
\end{array}
\end{equation}
for $n=0, \ldots, N-1$, $p=0, \ldots, M-1$ and $i=0,1.$ In addition, if a local CFL condition is satisfied:
\begin{equation} \label{eq:localCFL}
\lambda_{j}^{n,p} (h_{1}-h_{2}) \leq \frac{1}{1+\theta},
\end{equation}
where $h_{1}$ and $-h_{2}$ are the Lipschitz coefficients of $h(\cdot, \cdot)$ with respect to the first and second arguments respectively, and $\lambda_{j}^{n,p}$ is defined by
\begin{equation*}
\lambda_{j}^{n,p} = \left \{ \begin{array}{ll} \frac{\Delta t}{\Delta x_{j}}, & \text{if} \; j\leq j_{0}^{n}+1, \vspace{3pt} \\
\frac{\Delta t}{M \Delta x_{j}}, & \text{if} \; j>j_{0}^{n}+1,
\end{array}\right . \qquad \begin{array}{l} \text{for $n=0, \ldots, N-1,$} \vspace{3pt}\\ \text{  {and} $p=0, \ldots, M-1$.}
\end{array}
\end{equation*}
Then the second order LTS scheme is TVBM. 
\end{theorem}
\begin{proof} Following the techniques in \cite{Osher83}, we first introduce some important facts that will be used later in the proof. From the monotonicity of $h(\cdot, \cdot)$ and \eqref{eq:TVDlimiter}, we deduce that, in \eqref{eq:cd}: \vspace{-0.1cm}
\begin{equation} \label{eq:posCD}
C_{j+\ahalf}^{n,(i)}, D_{j+\ahalf}^{n,(i)} \geq 0, \quad \forall j < j_{0}, \quad \text{and} \quad C_{j+\ahalf}^{n,p,(i)},  D_{j+\ahalf}^{n,p,(i)} \geq 0, \quad \forall j \geq j_{0}.
\end{equation}
We may omit the superscripts for the ease of presentation. Given any two nonnegative numbers $\alpha, \beta$ and suppose $\lambda_{j}^{n,p}=\max(\alpha, \beta)$ that satisfies \eqref{eq:localCFL}, we have 
\begin{equation*}
\alpha C_{j+\ahalf}+ \beta D_{j+\ahalf} \leq 1,
\end{equation*}
and consequently,
\begin{align*}
&\left \vert \baru_{j+1}-\baru_{j} - \alpha D_{j+\ahalf} \Delta_{-} \baru_{j+1} -\beta C_{j+\ahalf} \Delta_{+} \baru_{j} \right \vert  = \left  \vert \Delta_{+} \baru_{j} \right \vert \, \left \vert 1- \alpha D_{j+\ahalf} - \beta C_{j+\ahalf} \right \vert   \vspace{3pt}\\
&= \vert \baru_{j+1}-\baru_{j} \vert - \alpha  D_{j+\ahalf} \, \vert \Delta_{-} \baru_{j+1}  \vert - \beta C_{j+\ahalf} \, \vert \Delta_{+} \baru_{j}  \vert , 
\end{align*}
in which the functions must be evaluated at the same time level. Then, together with \eqref{eq:posCD}, we deduce that
\begin{align} 
&\left \vert \baru_{j+1}-\baru_{j} + \alpha \left (C_{j+\sfrac{3}{2}} \Delta_{+} \baru_{j+1} - D_{j+\ahalf} \Delta_{-} \baru_{j+1} \right )  - \beta \left (C_{j+\ahalf} \Delta_{+} \baru_{j} -D_{j-\ahalf} \Delta_{-} \baru_{j}\right ) \right \vert    \notag \vspace{3pt} \\ 
&\leq  \left  \vert \baru_{j+1}-\baru_{j} - \alpha D_{j+\ahalf} \Delta_{-} \baru_{j+1} -\beta C_{j+\ahalf} \Delta_{+} \baru_{j} \right  \vert + \alpha C_{j+\sfrac{3}{2}} \, \vert \Delta_{+} \baru_{j+1} \vert + \beta D_{j-\ahalf} \, \vert \Delta_{-} \baru_{j} \vert  \notag \vspace{3pt}\\ 
&\leq \vert \baru_{j+1}-\baru_{j} \vert + \alpha \left ( C_{j+\sfrac{3}{2}} \, \vert \Delta_{+} \baru_{j+1} \vert  - D_{j+\ahalf} \, \vert \Delta_{-} \baru_{j+1} \vert \right )     \notag  \vspace{2pt}\\ 
&\hspace{2.2cm} - \beta \left (C_{j+\ahalf} \, \vert \Delta_{+} \baru_{j} \vert - D_{j-\ahalf} \, \vert \Delta_{-} \baru_{j} \vert \right ). \label{eq:keyes}
\end{align}
In the following, we compute the variation $\vert u_{j+1} - u_{j} \vert$ for all $j$. Particularly, we consider four cases: \vspace{0.2cm}\\
%
%
\noindent {\em i) If $j<(j_{0}^{n}-1)$}: from \eqref{eq:coarseLTS2}, we find that
\begin{equation}\label{eq:C2}
\begin{array}{ll}
\baru_{j+1}^{n,(1)} - \baru_{j}^{n, (1)}
= &\hspace{-0.2cm}\left (\baru_{j+1}^{n} - \baru_{j}^{n} \right ) +\frac{\Delta t}{\Delta x_{j+1}} \left (C_{j+\sfrac{3}{2}}^{n} \Delta_{+} \baru_{j+1}^{n} - D_{j+\ahalf}^{n} \Delta_{-} \baru_{j+1}^{n} \right) \\ &- \frac{\Delta t}{\Delta x_{j}} \left (C_{j+\ahalf}^{n} \Delta_{+} \baru_{j}^{n} -D_{j-\ahalf}^{n} \Delta_{-} \baru_{j}^{n} \right ) \bigg ].
\end{array} \hspace{-1cm}
\end{equation}
Applying \eqref{eq:keyes} with $\alpha=\frac{\Delta t}{\Delta x_{j+1}} $ and $\beta = \frac{\Delta t}{\Delta x_{j}}$, we deduce from \eqref{eq:C2} that \vspace{-0.1cm}
\begin{equation*} \label{eq:C3}
\begin{array}{ll}
 \left \vert \baru_{j+1}^{n,(1)} - \baru_{j}^{n, (1)}\right \vert \leq  \left \vert \baru_{j+1}^{n} - \baru_{j}^{n}\right \vert  + \Delta_{+} \left [ \frac{\Delta t}{\Delta x_{j}} \left (C_{j+\ahalf}^{n} \left \vert \Delta_{+} \baru_{j}^{n} \right \vert - D_{j-\ahalf}^{n} \left \vert \Delta_{-} \baru_{j}^{n}\right \vert \right )\right ].
\end{array}  \vspace{-0.1cm}
\end{equation*}
From this we obtain \vspace{-0.1cm}
\begin{align}  
&\left \vert \baru_{j+1}^{n+1} - \baru_{j}^{n+1}\right \vert \leq \frac{1}{2} \left \vert \baru_{j+1}^{n} - \baru_{j}^{n} \right \vert + \frac{1}{2} \left \vert  \baru_{j+1}^{n, (1)} - \baru_{j}^{n, (1)} \right \vert  \notag \vspace{3pt}\\
& \hspace{1cm}+ \frac{1}{2} \Delta_{+} \left [ \frac{\Delta t}{\Delta x_{j}} \left (C_{j+\ahalf}^{n,(1)} \left \vert \Delta_{+} \baru_{j}^{n,(1)} \right \vert - D_{j-\ahalf}^{n,(1)} \left \vert \Delta_{-} \baru_{j}^{n,(1)}\right \vert \right )\right ] \notag \vspace{3pt}\\
& \leq \left \vert \baru_{j+1}^{n} - \baru_{j}^{n} \right \vert + \frac{1}{2} \sum_{\nu=0}^{1} \Delta_{+} \left [ \frac{\Delta t}{\Delta x_{j}} \left (C_{j+\ahalf}^{n,(\nu)} \left \vert \Delta_{+} \baru_{j}^{n,(\nu)} \right \vert - D_{j-\ahalf}^{n,(\nu)} \left \vert \Delta_{-} \baru_{j}^{n,(\nu)}\right \vert \right ) \right ].  \notag 
\end{align} \vspace{-0.1cm}
or equivalently \vspace{-0.1cm}
\begin{align}  
&\left \vert \baru_{j+1}^{n+1} - \baru_{j}^{n+1}\right \vert \leq \left \vert \baru_{j+1}^{n} - \baru_{j}^{n} \right \vert + \frac{1}{2} \sum_{p=0}^{M-1}\sum_{\nu=0}^{1} \Delta_{+} \bigg [ \frac{\Delta t}{M\Delta x_{j}} \big (C_{j+\ahalf}^{n,(\nu)} \left \vert \Delta_{+} \baru_{j}^{n,(\nu)} \right \vert  \notag \vspace{3pt}\\
& \hspace{5cm} - D_{j-\ahalf}^{n,(\nu)} \left \vert \Delta_{-} \baru_{j}^{n,(\nu)}\right \vert \big ) \bigg ], \quad \forall\, j < j_{0}^{n}-1. \label{eq:varianceC} \vspace{-0.1cm}
\end{align}
%
%
\noindent {\em ii) If $j>j_{0}^{n}$}: By the same argument applied to \eqref{eq:fineLTS2} with a fine time step, we find that \vspace{-0.1cm}
\begin{equation*}
\begin{array}{ll} 
&\left \vert \baru_{j+1}^{n+1} - \baru_{j}^{n+1}\right \vert \leq \left \vert \baru_{j+1}^{n,M-1} - \baru_{j}^{n, M-1} \right \vert \vspace{3pt}\\
\hspace{0.4cm} &+ \frac{1}{2} \sum_{\nu=0}^{1} \Delta_{+} \left [ \frac{\Delta t}{M\Delta x_{j}} \left (C_{j+\ahalf}^{n,M-1, (\nu)} \left \vert \Delta_{+} \baru_{j}^{n, M-1, (\nu)} \right \vert - D_{j-\ahalf}^{n,M-1, (\nu)} \left \vert \Delta_{-} \baru_{j}^{n,M-1, (\nu)}\right \vert \right ) \right ]. 
\end{array}  \vspace{-0.1cm}
\end{equation*}
Repeating this argument inductively, we obtain a similar bound as \eqref{eq:varianceC}:
\begin{equation} \label{eq:variance}
\begin{array}{ll} 
\left \vert \baru_{j+1}^{n+1} - \baru_{j}^{n+1}\right \vert & \leq \left \vert \baru_{j+1}^{n} - \baru_{j}^{n} \right \vert + \frac{1}{2} \sum_{p=0}^{M-1}\sum_{\nu=0}^{1} \Delta_{+} \bigg [ \frac{\Delta t}{M\Delta x_{j}} \big (C_{j+\ahalf}^{n,p, (\nu)} \left \vert \Delta_{+} \baru_{j}^{n, p, (\nu)} \right \vert \vspace{3pt}\\
& \hspace{2cm} - D_{j-\ahalf}^{n,p, (\nu)} \left \vert \Delta_{-} \baru_{j}^{n,p, (\nu)}\right \vert \big ) \bigg ], \quad \forall\, j > j_{0}^{n}. 
\end{array} \vspace{-0.1cm}
\end{equation}
%
%
\noindent {\em iii)  If $j=j_{0}^{n}$ (the interface cell)}  \vspace{0.1cm}\\
We aim to show that \eqref{eq:variance} again holds for $j=j_{0}^{n}$, which is the main part of the proof. Using the formulation for the corrector \eqref{eq:IFcorr1}-\eqref{eq:IFcorr2}, as well as the time-stepping scheme in the fine region, we obtain
\begin{align}
&\hspace{-0.3cm} \baru_{j_{0}^{n}+1}^{n+1} -\baru_{j_{0}^{n}}^{n+1}  
= \baru_{j_{0}^{n}+1}^{n} -\baru_{j_{0}^{n}}^{n} \notag \vspace{3pt}\\
& + \frac{1}{2} \sum_{p=0}^{M-1} \sum_{\nu=0}^{1} \bigg [ \frac{\Delta t}{M\Delta x_{j_{0}^{n}+1}} \left (C_{j_{0}^{n}+\sfrac{3}{2}}^{n,p, (\nu)} \,  \Delta_{+} \baru_{j_{0}^{n}+1}^{n, p, (\nu)}  - D_{j_{0}^{n}+\ahalf}^{n,p, (\nu)} \,  \Delta_{-} \baru_{j_{0}^{n}+1}^{n,p, (\nu)}\right ) \notag \vspace{3pt}\\
&- \frac{\Delta t}{M\Delta x_{j_{0}^{n}}} \left ( C_{j_{0}^{n}+\ahalf}^{n,p,(\nu)} \Delta_{+} \baru_{j_{0}^{n}}^{n,p,(\nu)} - D_{j_{0}^{n}-\ahalf}^{n,(\nu)} \Delta_{-} \baru_{j_{0}^{n}}^{n,(\nu)}\right ) \bigg] \notag \vspace{3pt}\\
&-\frac{1}{2} \sum_{p=0}^{M-1}  \frac{\Delta t}{M\Delta x_{j_{0}^{n}}}  \bigg [ \left (\frac{p}{M}( -\ih_{2, j_{0}^{n}-\ahalf}^{n,p})+\frac{p+1}{M}( -\ih_{2, j_{0}^{n}-\ahalf}^{n,p,(1)}) \right ) \left (u_{j_{0}^{n}-\ahalf}^{n,(1),+}-u_{j_{0}^{n}-\ahalf}^{n,+}\right )\notag \vspace{3pt}\\
& -\left (\frac{p}{M} \ih_{1, j_{0}^{n}+\ahalf}^{n,p}+\frac{p+1}{M} \ih_{1, j_{0}^{n}+\ahalf}^{n,p,(1)}\right )\left (u_{j_{0}^{n}+\ahalf}^{n,(1),-}-u_{j_{0}^{n}+\ahalf}^{n,-}\right ) \bigg) \bigg ]. \notag
\end{align}
Since $u_{j_{0}^{n}-\ahalf}^{n,(1), \pm}-u_{j_{0}^{n}-\ahalf}^{n,\pm} = O(\Delta t)$, and by the CFL condition, we can bound
\begin{align}
&\hspace{-0.4cm} \left \vert \baru_{j_{0}^{n}+1}^{n+1} -\baru_{j_{0}^{n}}^{n+1}  \right  \vert \leq \bigg \vert  \baru_{j_{0}^{n}+1}^{n} -\baru_{j_{0}^{n}}^{n}  + \frac{1}{2} \sum_{p=0}^{M-1} \sum_{\nu=0}^{1} \bigg[ \frac{\Delta t}{M\Delta x_{j_{0}^{n}+1}} \big (C_{j_{0}^{n}+\sfrac{3}{2}}^{n,p, (\nu)} \,  \Delta_{+} \baru_{j_{0}^{n}+1}^{n, p, (\nu)} \notag \vspace{3pt}\\
& - D_{j_{0}^{n}+\ahalf}^{n,p, (\nu)} \,  \Delta_{-} \baru_{j_{0}^{n}+1}^{n,p, (\nu)}\big ) - \frac{\Delta t}{M\Delta x_{j_{0}^{n}}} \left ( C_{j_{0}^{n}+\ahalf}^{n,p,(\nu)} \Delta_{+} \baru_{j_{0}^{n}}^{n,p,(\nu)} - D_{j_{0}^{n}-\ahalf}^{n,(\nu)} \Delta_{-} \baru_{j_{0}^{n}}^{n,(\nu)}\right ) \bigg] \bigg \vert \notag \vspace{3pt}\\
&  + \frac{M}{1+\theta} O(\Delta t). \label{eq:IF1}
\end{align}
We have
\begin{align}
&\baru_{j_{0}^{n}+1}^{n} -\baru_{j_{0}^{n}}^{n} =\frac{1}{2M} \sum_{p=0}^{M-1} \bigg \{  \left [\left (\baru_{j_{0}^{n}+1}^{n} -\baru_{j_{0}^{n}}^{n}\right ) - \left (\baru_{j_{0}^{n}+1}^{n,p} -\baru_{j_{0}^{n}}^{n,p}\right ) \right ] \notag \vspace{3pt}\\
& \;+   \left [\left (\baru_{j_{0}^{n}+1}^{n} -\baru_{j_{0}^{n}}^{n}\right ) - \left (\baru_{j_{0}^{n}+1}^{n,p,(1)} -\baru_{j_{0}^{n}}^{n,p,(1)}\right )\right ]  + \left (\baru_{j_{0}^{n}+1}^{n,p} -\baru_{j_{0}^{n}}^{n,p}\right ) +\left (\baru_{j_{0}^{n}+1}^{n,p, (1)} -\baru_{j_{0}^{n}}^{n,p,(1)}\right ) \bigg \}. \label{eq:IF2}
\end{align}
Regarding the first two terms, let us write 
\begin{align}
& \frac{1}{2} \left [\left (\baru_{j_{0}^{n}+1}^{n} -\baru_{j_{0}^{n}}^{n}\right ) - \left (\baru_{j_{0}^{n}+1}^{n,p} -\baru_{j_{0}^{n}}^{n,p}\right ) \right ] +  \frac{1}{2} \left [\left (\baru_{j_{0}^{n}+1}^{n} -\baru_{j_{0}^{n}}^{n}\right ) - \left (\baru_{j_{0}^{n}+1}^{n,p,(1)} -\baru_{j_{0}^{n}}^{n,p,(1)}\right )\right ]  \notag \vspace{3pt}\\
&\hspace{-0.1cm}= \left (\baru_{j_{0}^{n}+1}^{n} - \baru_{j_{0}^{n}+1}^{n,p}\right ) + \frac{1}{2} \left ( \baru_{j_{0}^{n}+1}^{n,p}-u_{j_{0}^{n}+1}^{n,p,(1)}\right ) +\frac{1}{2} \left [\left (\baru_{j_{0}^{n}}^{n,p}  - \baru_{j_{0}^{n}}^{n}\right )+\left (\baru_{j_{0}^{n}}^{n,p,(1)}-\baru_{j_{0}^{n}}^{n}\right )\right ]. \label{eq:IF3} 
\end{align}
By definition of the second order predictor \eqref{eq:2ndpredictor}, the last term in \eqref{eq:IF3} is given by
\begin{equation} \label{eq:IF5}
\frac{1}{2} \left [\left (\baru_{j_{0}^{n}}^{n,p}  - \baru_{j_{0}^{n}}^{n}\right )+\left (\baru_{j_{0}^{n}}^{n,p,(1)}-\baru_{j_{0}^{n}}^{n}\right )\right ]=\frac{2p+1}{2M}\left (\baru_{j_{0}^{n}}^{n,(1)}-\baru_{j_{0}^{n}}^{n}\right ) = O(\Delta t).
\end{equation}
On the other hand, the first and second terms in \eqref{eq:IF3} can be computed by using the time-stepping in the fine region \eqref{eq:fineLTS2}:
\begin{equation}  \label{eq:IF4c}
\begin{array}{ll}
 &\left (\baru_{j_{0}^{n}+1}^{n}  - \baru_{j_{0}^{n}+1}^{n,p}\right )+\frac{1}{2}\left ( \baru_{j_{0}^{n}+1}^{n,p}-\baru_{j_{0}^{n}+1}^{n,p,(1)}\right ) \vspace{3pt}\\
 &\quad=-\frac{1}{2} \sum_{q=0}^{p} \frac{\Delta t}{M \Delta x_{j_{0}^{n}+1}} \left (C_{j_{0}^{n}+\sfrac{3}{2}}^{n,q} \,  \Delta_{+} \baru_{j_{0}^{n}+1}^{n, q}  - D_{j_{0}^{n}+\ahalf}^{n,q} \,  \Delta_{-} \baru_{j_{0}^{n}+1}^{n,q}\right ) \vspace{3pt}\\
 &\quad\quad-\frac{1}{2} \sum_{q=0}^{p-1}\frac{\Delta t}{M \Delta x_{j_{0}^{n}+1}} \left (C_{j_{0}^{n}+\sfrac{3}{2}}^{n,q, (1)} \,  \Delta_{+} \baru_{j_{0}^{n}+1}^{n, q, (1)}  - D_{j_{0}^{n}+\ahalf}^{n,q, (1)} \,  \Delta_{-} \baru_{j_{0}^{n}+1}^{n,q, (1)}\right ).
 \end{array} \vspace{-0.1cm}
\end{equation}
Summing \eqref{eq:IF4c} over $p=0, \ldots, M-1$ yields
\begin{align}
& \sum_{p=0}^{M-1} \left \{ \left (\baru_{j_{0}^{n}+1}^{n}  - \baru_{j_{0}^{n}+1}^{n,p}\right ) +\frac{1}{2}\left ( \baru_{j_{0}^{n}+1}^{n,p}-u_{j_{0}^{n}+1}^{n,p,(1)}\right ) \right \} \notag \vspace{2pt}\\
&\quad = -\frac{1}{2}  \sum_{p=0}^{M-1} \left (1-\frac{p}{M}\right )  \frac{\Delta t}{\Delta x_{j_{0}^{n}+1}} \left (C_{j_{0}^{n}+\sfrac{3}{2}}^{n,p} \,  \Delta_{+} \baru_{j_{0}^{n}+1}^{n, p}  - D_{j_{0}^{n}+\ahalf}^{n,p} \,  \Delta_{-} \baru_{j_{0}^{n}+1}^{n,p}\right ) \notag \vspace{2pt}\\
&\quad\quad -\frac{1}{2}  \sum_{p=0}^{M-1}\left (1-\frac{p+1}{M}\right )  \frac{\Delta t}{\Delta x_{j_{0}^{n}+1}} \left (C_{j_{0}^{n}+\sfrac{3}{2}}^{n,p, (1)} \,  \Delta_{+} \baru_{j_{0}^{n}+1}^{n, p, (1)}  - D_{j_{0}^{n}+\ahalf}^{n,p, (1)} \,  \Delta_{-} \baru_{j_{0}^{n}+1}^{n,p, (1)}\right ).  \label{eq:IF6} \vspace{-0.1cm}
\end{align}
Substituting \eqref{eq:IF5} and \eqref{eq:IF6}  into \eqref{eq:IF2} and then \eqref{eq:IF1}, and using \eqref{eq:keyes} with $\textstyle\alpha= \frac{p}{M} \left(\frac{\Delta t}{\Delta x_{j_{0}^{n}+1}}\right)$ or $\textstyle\alpha= \frac{p+1}{M}\left(\frac{\Delta t}{\Delta x_{j_{0}^{n}+1}}\right)$ and $\textstyle\beta =\frac{\Delta t}{\Delta x_{j_{0}^{n}}}$, we obtain: 
\begin{equation} \label{eq:IF7}
\begin{array}{ll}
 & \left \vert \baru_{j_{0}^{n}+1}^{n+1} -\baru_{j_{0}^{n}}^{n+1} \right \vert \leq  \frac{1}{2M} \sum_{p=0}^{M-1} \bigg \{\left  \vert \baru_{j_{0}^{n}+1}^{n,p} -\baru_{j_{0}^{n}}^{n,p}\right  \vert  +\left  \vert \baru_{j_{0}^{n}+1}^{n,p, (1)} -\baru_{j_{0}^{n}}^{n,p,(1)}\right  \vert  \vspace{3pt}\\ 
&\qquad + \frac{p}{M} \frac{\Delta t}{\Delta x_{j_{0}^{n}+1}} \left (C_{j_{0}^{n}+\sfrac{3}{2}}^{n,p} \, \left \vert \Delta_{+} \baru_{j_{0}^{n}+1}^{n, p} \right \vert - D_{j_{0}^{n}+\ahalf}^{n,p} \, \left \vert \Delta_{-} \baru_{j_{0}^{n}+1}^{n,p}  \right \vert \right )\vspace{3pt}\\
 &\qquad + \frac{p+1}{M}\frac{\Delta t}{\Delta x_{j_{0}^{n}+1}} \left (C_{j_{0}^{n}+\sfrac{3}{2}}^{n,p, (1)} \,  \left \vert \Delta_{+} \baru_{j_{0}^{n}+1}^{n, p, (1)}\right \vert  - D_{j_{0}^{n}+\ahalf}^{n,p, (1)} \,  \left \vert \Delta_{-} \baru_{j_{0}^{n}+1}^{n,p, (1)}  \right \vert \right ) \vspace{3pt}\\
 &\qquad - \frac{\Delta t}{\Delta x_{j_{0}^{n}}} \sum_{\nu=0}^{1} \left (C_{j_{0}^{n}+\ahalf}^{n,p, (\nu)} \,  \left \vert \Delta_{+} \baru_{j_{0}^{n}}^{n, p, (\nu)}  \right \vert - D_{j_{0}^{n}-\ahalf}^{n,(\nu)} \,  \left \vert \Delta_{-} \baru_{j_{0}^{n}}^{n,(\nu)}  \right \vert\right ) \bigg\}+ O(\Delta t).
\end{array} 
\end{equation}
Furthermore, by the definition of the second order predictor, we have 
$\textstyle\baru_{j_{0}^{n}}^{n,p,(1)}~=~\baru_{j_{0}^{n}}^{n,p}+\frac{1}{M}\left (\baru_{j_{0}^{n}}^{n,(1)}-\baru_{j_{0}^{n}}^{n}\right ) = \baru_{j_{0}^{n}}^{n,p}+ O(\Delta t).$
This together with using \eqref{eq:keyes} for $\textstyle\alpha~=~\frac{\Delta t}{M \Delta x_{j_{0}^{n}+1}} $  and $\beta =0$, we have that 
\begin{equation} \label{eq:IF9}
\begin{array}{ll}
&\hspace{-0.3cm} \left \vert \baru_{j_{0}^{n}+1}^{n,p,(1)} -\baru_{j_{0}^{n}}^{n,p,(1)} \right \vert = \bigg \vert \baru_{j_{0}^{n}+1}^{n, p}-\baru_{j_{0}^{n}}^{n, p} \vspace{3pt} \\
& + \frac{\Delta t}{M \Delta x_{j_{0}^{n}+1}} \left (C_{j_{0}^{n}+\sfrac{3}{2}}^{n,p} \,  \Delta_{+} \baru_{j_{0}^{n}+1}^{n, p}  - D_{j_{0}^{n}+\ahalf}^{n,p} \,  \Delta_{-} \baru_{j_{0}^{n}+1}^{n,p}\right )  \bigg \vert    + O(\Delta t)
\vspace{3pt} \\
& \leq \left \vert \baru_{j_{0}^{n}+1}^{n,p} -\baru_{j_{0}^{n}}^{n,p} \right \vert +  \frac{\Delta t}{M \Delta x_{j_{0}^{n}+1}}\left (C_{j_{0}^{n}+\sfrac{3}{2}}^{n,p} \left \vert \Delta_{+} \baru_{j_{0}^{n}+1}^{n, p} \right \vert  - D_{j_{0}^{n}+\ahalf}^{n,p} \left \vert \Delta_{-} \baru_{j_{0}^{n}+1}^{n,p} \right \vert \right )  + O(\Delta t).
\end{array}  
\end{equation}
Plugging this into \eqref{eq:IF7} yields: \vspace{-0.1cm}
\begin{equation} \label{eq:IFess_b}
\begin{array}{ll}
& \left \vert \baru_{j_{0}^{n}+1}^{n+1} -\baru_{j_{0}^{n}}^{n+1} \right \vert \leq \frac{1}{M} \sum_{p=0}^{M-1} \bigg \{ \left \vert \baru_{j_{0}^{n}+1}^{n,p} -\baru_{j_{0}^{n}}^{n,p} \right \vert  \vspace{2pt}\\
&\qquad + \frac{p+1}{2M} \frac{\Delta t}{\Delta x_{j_{0}^{n}+1}} \left (C_{j_{0}^{n}+\sfrac{3}{2}}^{n,p} \left \vert \Delta_{+} \baru_{j_{0}^{n}+1}^{n, p} \right \vert  - D_{j_{0}^{n}+\ahalf}^{n,p} \left \vert \Delta_{-} \baru_{j_{0}^{n}+1}^{n,p} \right \vert \right ) \vspace{3pt}\\
&\qquad + \frac{p+1}{2M} \frac{\Delta t}{\Delta x_{j_{0}^{n}+1}} \left (C_{j_{0}^{n}+\sfrac{3}{2}}^{n,p,(1)} \left \vert \Delta_{+} \baru_{j_{0}^{n}+1}^{n, p,(1)} \right \vert  - D_{j_{0}^{n}+\ahalf}^{n,p,(1)} \left \vert \Delta_{-} \baru_{j_{0}^{n}+1}^{n,p,(1)} \right \vert \right ) \vspace{3pt}\\
&\qquad - \frac{1}{2}\sum_{\nu=0}^{1}\frac{\Delta t}{\Delta x_{j_{0}^{n}}} \left (C_{j_{0}^{n}+\ahalf}^{n,p,(\nu)} \left \vert \Delta_{+} \baru_{j_{0}^{n}}^{n, p,(\nu)} \right \vert  - D_{j_{0}^{n}-\ahalf}^{n,(\nu)} \left \vert  \Delta_{-} \baru_{j_{0}^{n}}^{n,(\nu)} \right \vert \right ) \bigg \} + O(\Delta t).
\end{array} \hspace{-0.8cm}
\end{equation}
On the other hand, by the SSP-RK$(2,2)$ time-stepping in the fine cell $(j_{0}^{n}+1)$ and using \eqref{eq:IF9}, we deduce that \vspace{-0.1cm}
\begin{equation*} 
\begin{array}{ll}
&\hspace{-0.4cm} \left \vert \baru_{j_{0}^{n}+1}^{n,p} -\baru_{j_{0}^{n}}^{n,p} \right \vert = \bigg \vert \frac{1}{2} \left (\baru_{j_{0}^{n}+1}^{n, p-1}-\baru_{j_{0}^{n}}^{n, p-1}\right ) + \frac{1}{2} \left (\baru_{j_{0}^{n}+1}^{n, p-1,(1)}-\baru_{j_{0}^{n}}^{n, p-1,(1)}\right )  \vspace{3pt} \\
& + \frac{1}{2} \frac{\Delta t}{M \Delta x_{j_{0}^{n}+1}}\left (C_{j_{0}^{n}+\sfrac{3}{2}}^{n,p-1,(1)} \,  \Delta_{+} \baru_{j_{0}^{n}+1}^{n, p-1,(1)}  - D_{j_{0}^{n}+\ahalf}^{n,p-1,(1)} \,  \Delta_{-} \baru_{j_{0}^{n}+1}^{n,p-1,(1)}\right )   \vspace{3pt} \\
&  + \frac{1}{2}\left (\baru_{j_{0}^{n}}^{n, p-1}+\baru_{j_{0}^{n}}^{n, p-1,(1)}-2\baru_{j_{0}^{n}}^{n,p}\right ) \bigg \vert 
 \leq \left \vert \baru_{j_{0}^{n}+1}^{n,p-1} -\baru_{j_{0}^{n}}^{n,p-1} \right \vert \vspace{3pt}\\
&+\frac{1}{2} \frac{\Delta t}{M \Delta x_{j_{0}^{n}+1}} \sum_{\nu=0}^{1}\left (C_{j_{0}^{n}+\sfrac{3}{2}}^{n,p-1,(\nu)} \left \vert \Delta_{+} \baru_{j_{0}^{n}+1}^{n, p-1,(\nu)} \right \vert  - D_{j_{0}^{n}+\ahalf}^{n,p-1,(\nu)} \left \vert \Delta_{-} \baru_{j_{0}^{n}+1}^{n,p-1,(\nu)} \right \vert \right ) + O(\Delta t),
\end{array}
\end{equation*}
in which we have used definition of the second order predictor to obtain \vspace{-0.1cm}
\begin{align*}
\frac{1}{2}\left (\baru_{j_{0}^{n}}^{n, p-1}+\baru_{j_{0}^{n}}^{n, p-1,(1)}-2\baru_{j_{0}^{n}}^{n,p}\right ) &= \frac{1}{2}\left (\baru_{j_{0}^{n}}^{n, p-1}-\baru_{j_{0}^{n}}^{n,p}\right ) =\frac{1}{2M}\left (\baru_{j_{0}^{n}}^{n}-\baru_{j_{0}^{n}}^{n,(1)}\right ) = O(\Delta t). \vspace{-0.1cm}
\end{align*}
Repeat the above argument inductively, we arrive at \vspace{-0.1cm}
\begin{equation*} \label{eq:IF10}
\begin{array}{ll}
& \left \vert \baru_{j_{0}^{n}+1}^{n,p} -\baru_{j_{0}^{n}}^{n,p} \right \vert \leq \left \vert \baru_{j_{0}^{n}+1}^{n} -\baru_{j_{0}^{n}}^{n} \right \vert \vspace{3pt}\\
&+\frac{1}{2} \frac{\Delta t}{M \Delta x_{j_{0}^{n}+1}} \sum_{q=0}^{p-1}\sum_{\nu=0}^{1}\left (C_{j_{0}^{n}+\sfrac{3}{2}}^{n,q,(\nu)} \left \vert \Delta_{+} \baru_{j_{0}^{n}+1}^{n, q,(\nu)} \right \vert  - D_{j_{0}^{n}+\ahalf}^{n,q,(\nu)} \left \vert \Delta_{-} \baru_{j_{0}^{n}+1}^{n,q,(\nu)} \right \vert \right ) + O(\Delta t).
\end{array} \vspace{-0.1cm}
\end{equation*}
Consequently, \vspace{-0.1cm}
\begin{equation*} \label{eq:IF11}
\begin{array}{ll}
&\hspace{-0.4cm} \sum_{p=0}^{M-1}\left \vert \baru_{j_{0}^{n}+1}^{n,p} -\baru_{j_{0}^{n}}^{n,p} \right \vert \leq M\left \vert \baru_{j_{0}^{n}+1}^{n} -\baru_{j_{0}^{n}}^{n} \right \vert \vspace{3pt}\\
&+\frac{1}{2} \frac{\Delta t}{M \Delta x_{j_{0}^{n}+1}} \sum_{p=0}^{M-1} \sum_{q=0}^{p-1}\sum_{\nu=0}^{1}\left (C_{j_{0}^{n}+\sfrac{3}{2}}^{n,q,(\nu)} \left \vert \Delta_{+} \baru_{j_{0}^{n}+1}^{n, q,(\nu)} \right \vert  - D_{j_{0}^{n}+\ahalf}^{n,q,(\nu)} \left \vert \Delta_{-} \baru_{j_{0}^{n}+1}^{n,q,(\nu)} \right \vert \right ) + O(\Delta t)\vspace{3pt}\\
& 
\leq M\left \vert \baru_{j_{0}^{n}+1}^{n} -\baru_{j_{0}^{n}}^{n} \right \vert +\frac{1}{2} \frac{\Delta t}{M \Delta x_{j_{0}^{n}+1}} \sum_{p=0}^{M-1}\sum_{\nu=0}^{1} \left (1-\frac{p+1}{M}\right )\bigg (C_{j_{0}^{n}+\sfrac{3}{2}}^{n,p,(\nu)} \left \vert \Delta_{+} \baru_{j_{0}^{n}+1}^{n, p,(\nu)} \right \vert \vspace{3pt}\\
& \hspace{5cm} - D_{j_{0}^{n}+\ahalf}^{n,p,(\nu)} \left \vert \Delta_{-} \baru_{j_{0}^{n}+1}^{n,p,(\nu)} \right \vert \bigg ) + O(\Delta t),
\end{array} \vspace{-0.1cm}
\end{equation*}
where the last inequality is obtained by reversing the order of summation. Plug this into \eqref{eq:IFess_b}, we find that
\begin{equation} \label{eq:IFess_bb} 
\begin{array}{ll}
 & \left \vert \baru_{j_{0}^{n}+1}^{n+1} -\baru_{j_{0}^{n}}^{n+1} \right \vert \leq \left \vert \baru_{j_{0}^{n}+1}^{n} - \baru_{j_{0}^{n}}^{n} \right \vert \vspace{3pt}\\
 & + \frac{1}{2} \sum_{p=0}^{M-1} \sum_{\nu=0}^{1} \bigg [ \frac{\Delta t}{M\Delta x_{j_{0}^{n}+1}} \left (C_{j_{0}^{n}+\sfrac{3}{2}}^{n,p, (\nu)} \left \vert \Delta_{+} \baru_{j_{0}^{n}+1}^{n, p, (\nu)} \right \vert - D_{j_{0}^{n}+\ahalf}^{n,p, (\nu)} \left \vert \Delta_{-} \baru_{j_{0}^{n}+1}^{n,p, (\nu)}\right \vert \right ) \vspace{3pt}\\
 & + \frac{\Delta t}{M\Delta x_{j_{0}^{n}}} \left (C_{j_{0}^{n}+\ahalf}^{n,p, (\nu)} \left \vert \Delta_{+} \baru_{j_{0}^{n}}^{n, p, (\nu)} \right \vert - D_{j_{0}^{n}-\ahalf}^{n,(\nu)} \left \vert \Delta_{-} \baru_{j_{0}^{n}}^{n,(\nu)}\right \vert \right ) \bigg] + O(\Delta t).
\end{array} \hspace{-0.3cm} 
\end{equation}
%
%
\noindent {\em iv) If $j=j_{0}^{n}-1$}: It remains to investigate the case of the interface cell and its neighbor in $\Omega_c^{n}$. Similarly to \eqref{eq:IF1}, we have 
\begin{align*}
&\hspace{-0.4cm} \left \vert \baru_{j_{0}^{n}}^{n+1} -\baru_{j_{0}^{n}-1}^{n+1}  \right  \vert \leq \bigg \vert  \baru_{j_{0}^{n}}^{n} -\baru_{j_{0}^{n}-1}^{n}  + \frac{1}{2} \sum_{p=0}^{M-1} \sum_{\nu=0}^{1}  \frac{\Delta t}{M\Delta x_{j_{0}^{n}}} \big (C_{j_{0}^{n}+\sfrac{1}{2}}^{n,p, (\nu)} \,  \Delta_{+} \baru_{j_{0}^{n}}^{n, p, (\nu)} \notag \vspace{3pt}\\
& - D_{j_{0}^{n}-\ahalf}^{n,(\nu)} \,  \Delta_{-} \baru_{j_{0}^{n}}^{n,(\nu)}\big ) - \frac{\Delta t}{M\Delta x_{j_{0}^{n}-1}} \left ( C_{j_{0}^{n}-\ahalf}^{n,(\nu)} \Delta_{+} \baru_{j_{0}^{n}-1}^{n,(\nu)} - D_{j_{0}^{n}-\sfrac{3}{2}}^{n,(\nu)} \Delta_{-} \baru_{j_{0}^{n}-1}^{n,(\nu)}\right ) \bigg \vert \notag \vspace{3pt}\\
&  + \frac{M}{1+\theta} O(\Delta t). \label{eq:IF_c1} 
\end{align*}
By performing similar manipulations as for the case $j=j_{0}^{n}$, one arrives at
\begin{equation} \label{eq:IFess_c1}
\begin{array}{ll}
 &\hspace{-0.4cm} \left \vert \baru_{j_{0}^{n}}^{n+1} -\baru_{j_{0}^{n}-1}^{n+1} \right \vert \leq \left \vert \baru_{j_{0}^{n}}^{n} - \baru_{j_{0}^{n}-1}^{n} \right \vert \vspace{3pt}\\
 & + \frac{1}{2} \sum_{p=0}^{M-1} \sum_{\nu=0}^{1} \bigg [ \frac{\Delta t}{M\Delta x_{j_{0}^{n}}} \left (C_{j_{0}^{n}+\ahalf}^{n,p, (\nu)} \left \vert \Delta_{+} \baru_{j_{0}^{n}}^{n, p, (\nu)} \right \vert - D_{j_{0}^{n}-\ahalf}^{n, (\nu)} \left \vert \Delta_{-} \baru_{j_{0}^{n}}^{n,(\nu)}\right \vert \right ) \vspace{3pt}\\
 & + \frac{\Delta t}{M\Delta x_{j_{0}^{n}-1}} \left (C_{j_{0}^{n}-\ahalf}^{n,(\nu)} \left \vert \Delta_{+} \baru_{j_{0}^{n}-1}^{n,(\nu)} \right \vert - D_{j_{0}^{n}-\sfrac{3}{2}}^{n,(\nu)} \left \vert \Delta_{-} \baru_{j_{0}^{n}-1}^{n,(\nu)}\right \vert \right ) \bigg] + O(\Delta t).
\end{array} 
\end{equation}
Finally, we combine \eqref{eq:varianceC}, \eqref{eq:variance}, \eqref{eq:IFess_bb} and \eqref{eq:IFess_c1} and obtain: 
$$ TV(\baru^{n+1}) \leq TV(\baru^{n}) + O(\Delta t), \quad \text{or} \quad  TV(\baru^{n}) \leq TV(\baru^{0}) + CT.
$$
Hence, the second order LTS scheme is TVBM. 
\end{proof} 
The condition~\eqref{eq:TVDlimiter} is fulfilled if the solution is limited by the {\em minmod} function $m$ defined in \eqref{eq:minmod} (see \cite{CSII}). The scheme remains TVB when the modified {\em minmod} function $\widetilde{m}$ is used, which is achieved by Theorem~2.2 in \cite{Shu87a} (see also \cite[Lemma 2.3]{CSII}). Finally, the TVB property of the means $\baru_{j}$ can be passed to whole solution $u_{h}$ in the same manner as the RK-DG method \cite[Propositition 2.11]{CSII}. 
We remark that it is assumed that the solution near the time-dependent LTS interface is sufficiently smooth so that the condition~\eqref{eq:TVDlimiter} is satisfied in the region of the LTS interface without limiting. In practice, local time-stepping should be coupled with adaptive spatial meshing to achieve computational efficiency and accuracy when dealing with hyperbolic conservation laws. 

\section{Numerical experiments} \label{sec:NumRe}
We consider several standard test cases of one dimensional scalar conservation laws (Subsection~\ref{subsec:scalarNumRe}) and system of conservation laws (Subsection~\ref{subsec:sysNumRe}). We aim to verify the accuracy, mass conservation and stability of the LTS schemes as predicted theoretically and compare with those by the global time-stepping (GTS) schemes. 
As a first step towards study the behavior of proposed schemes, we use a fixed LTS interface (i.e., $j_{0}^{n} = j_{0}$ for all $n$ ) in all the tests, instead of a time-varying LTS interface as discussed in Section~\ref{sec:LTS}, and leave the investigation on parallel performance of the proposed methods with space-time adaptive multiresolution meshes in two and three dimensions to future work. Note that for the prediction step, if the solution is discontinuous at the fixed LTS interface, then it is necessary to limit the solution on the interface with a coarse time step $\bu_{j_{0}}^{n,(i) \text{(mod)}}$ for $i=1, \ldots, s-1$ before calculating the predicted interface values~\eqref{eq:predUedge} at intermediate time levels. 

\subsection{Scalar conservation laws} \label{subsec:scalarNumRe} 
We first consider two model problems that obey the scalar conservation laws: the linear advection equation and Burgers' equation.
For problems with a smooth solution, we confirm the convergence order in time of our LTS algorithms. 
The effectiveness of LTS algorithms is demonstrated by comparing with the GTS schemes in terms of accuracy and CFL conditions.
\subsubsection*{Example 1: Linear problem}
We solve the linear advection problem with a smooth initial condition \vspace{-0.1cm}
\begin{equation} \label{eq:adv}
u_{t} + u_{x} = 0, \quad u(x,0) = \sin \pi x,   \vspace{-0.1cm}
\end{equation}
in $-1 \leq x \leq 1$ with periodic boundary conditions. The exact solution is given by $u(x,t)=\sin \pi (x-t)$. The spatial domain is divided into two subdomains, $\Omega_{1}=[-1,0]$ and $\Omega_{2}=[0,1]$. 
The mesh size and time step size in $\Omega_i$ are respectively $\Delta x_i$ and $\Delta t_i$, which are fine when $i=1$ and are coarse  when $i=2$: 
$$ \Delta x_{1} = \frac{\Delta x_{\text{coarse}}}{M}, \; \Delta x_{2}=\Delta x_{\text{coarse}}, \;\; \Delta t_{1}= \frac{\iC}{2k+1} \Delta x_{1}=\frac{\Delta t_{\text{coarse}}}{M}, \; \Delta t_{2}=\Delta t_{\text{coarse}},
$$
for $M=1,2,4,8$, and for $k=1,2,3$ corresponding to second, third and fourth order LTS methods. The $L^{1}$ relative errors at $T=2$ of the three LTS algorithms are listed in Table~\ref{tab:Advsmooth}. We observe that for all schemes, the errors decrease as $M$ increases; and the LTS schemes (with $M= 2, 4, 8$) preserve the order of convergence as in the GTS case ($M=1$), regardless of how large $M$ is.  

\begin{table}[!ht]\footnotesize
\setlength{\extrarowheight}{3pt}
\centering
\begin{tabular}{|c|c|c c|c c| c c|}
\hline 
\multicolumn{8}{|c|}{Linear advection} \\ \hline
$\Delta x_{\text{coarse}}$ & $M$ & \multicolumn{2}{c|}{RK-DG2}& \multicolumn{2}{c|}{RK-DG3} & \multicolumn{2}{c|}{RK-DG4}\\  \cline{3-8}
& & Rel. $L^{1}$ error & [CR] & Rel. $L^{1}$ error & [CR] & Rel. $L^{1}$ error & [CR]   \\ \hline 
\multirow{4}{*}{$1/5$} & $1$  &  5.70e-02  & -- & 1.39e-03  &  -- & 4.43e-05  & --  \\ \cline{2-8}
& $2$   & 3.59e-02   & -- & 8.09e-04   &  -- & 2.40e-05  & --  \\ \cline{2-8}
& $4$  & 3.13e-02  & -- & 7.65e-04   &  -- & 2.32e-05 & --  \\ \cline{2-8}
& $8$  & 3.04e-02  & -- & 7.62e-04   &  -- & 2.33e-05 & --  \\ \hline
\multirow{4}{*}{$1/10$} & $1$  &  1.36e-02   & [2.07] & 1.66e-04   &  [3.07] & 2.73e-06   & [4.02] \\ \cline{2-8}
& $2$   & 8.71e-03  & [2.04]&  9.76e-05  &  [3.05] & 1.46e-06  & [4.04]  \\ \cline{2-8}
& $4$  & 7.60e-03   & [2.04] &  9.20e-05  &  [3.06] & 1.39e-06  & [4.06]  \\ \cline{2-8}
& $8$  & 3.91e-03   & [2.04] & 9.17e-05   &  [3.06] & 1.41e-06 & [4.05]  \\ \hline
\multirow{4}{*}{$1/20$} & $1$ & 3.30e-03  & [2.04] & 2.04e-05  & [3.03] & 1.71e-07  & [4.00]\\ \cline{2-8}
& $2$  & 2.14e-03  & [2.03]& 1.20e-05  & [3.02] & 9.07e-08  & [4.01]\\ \cline{2-8}
& $4$ & 1.87e-03 & [2.02] & 1.14e-05  & [3.01] &  8.66e-08 & [4.01]\\ \cline{2-8}
& $8$ & 1.81e-03  & [2.03] &  1.13e-05  & [3.02] & 8.83e-08  & [4.00]\\ \hline
\multirow{4}{*}{$1/40$} & $1$ & 8.11e-04  & [2.03] &  2.54e-06 & [3.01] &  1.07e-08 & [4.00]\\ \cline{2-8}
& $2$  &  5.31e-04 & [2.01] &  1.49e-06 & [3.01] & 5.67e-09  & [4.00]\\ \cline{2-8}
& $4$ &  4.64e-04 & [2.01] & 1.41e-06  & [3.02] & 5.42e-09  & [4.00]\\ \cline{2-8}
& $8$ & 4.50e-04  & [2.01] & 1.41e-06   & [3.00] & 5.53e-09   & [4.00]\\  \hline 
\end{tabular}
\caption{[Linear advection with a smooth initial condition] $L^{1}$ relative errors at $T=2$ for different $M$. The rates of convergence ``CR''  for fixed $M$ are shown in square brackets.} \label{tab:Advsmooth} \vspace{-0.2cm}
\end{table}
Now to show that the LTS algorithms are stable with a local CFL condition, we still consider the linear problem \eqref{eq:adv} but with a discontinuous initial condition   \vspace{-0.1cm}
$$ u(x,0) = \left \{ \begin{array}{ll} 2, & x \leq -1, \\
-1, & x>-1.
\end{array} \right .  \vspace{-0.1cm}
$$
We are interested in the behavior of the approximate solution near the discontinuity $x=-0.5$ at $T=0.5$ and $x=0$ (the LTS interface) at $T=1$.  Again, the fine region is $\Omega_{1}=[-1,0]$ and the coarse one is $\Omega_{2}=[0,1]$. We use second order RK-DG method and consider three schemes as follows:
\begin{enumerate}[leftmargin=*]
\item Coarse GTS scheme with a coarse global time step $\Delta t = \Delta t_{\text{coarse}}$. \vspace{3pt}
\item Fine GTS scheme with a fine global time step $\Delta t = \sfrac{\Delta t_{\text{coarse}}}{M}$. \vspace{3pt}
\item LTS scheme with spatial variable time step $\Delta t_{1}=\sfrac{\Delta t_{\text{coarse}}}{M}$ and $\Delta t_{2}=\Delta t_{\text{coarse}}$. \vspace{3pt}
\end{enumerate}
Note that the spatial mesh is refined in $\Omega_{1}$ by a factor of $M$. Table~\ref{tab:Advdis} shows the $L^{1}$ relative errors of the three schemes at $T=0.5$ and $T=1$ respectively. It is seen that the coarse GTS becomes unstable as the spatial mesh is refined due to the violation of the CFL condition, while the LTS scheme, with a valid local CFL condition, gives stable solution with the same accuracy as the fine GTS scheme.

\begin{table}[!http]\footnotesize
\setlength{\extrarowheight}{3pt}
\centering
\begin{tabular}{|c|c | c| c | c |  c| c | c |}
\hline 
\multicolumn{8}{|c|}{Linear advection, RK-DG2} \\ \hline
\multirow{2}{*}{$\Delta x_{\text{coarse}}$} & \multirow{2}{*}{$M$}   & \multicolumn{3}{c}{At $T=0.5$} & \multicolumn{3}{|c|}{At $T=1$} \\    \cline{3-8}
&  & coarse GTS & fine GTS & LTS & coarse GTS & fine GTS & LTS  \\ \hline
\multirow{4}{*}{1/40} & 1  & \multicolumn{3}{c|}{2.58e-02} & \multicolumn{3}{c|}{2.53e-02} \\ \cline{2-8}
& 2 & 3.46e-02 & \multicolumn{2}{c|}{1.52e-02} & 3.46e-02 & \multicolumn{2}{c|}{1.50e-02} \\  \cline{2-8}
& 4 & -- &  \multicolumn{2}{c|}{9.00e-03}  & -- &  \multicolumn{2}{c|}{8.99e-03} \\  \cline{2-8}
& 8 & --  & \multicolumn{2}{c|}{5.40e-03} & --  & \multicolumn{2}{c|}{5.46e-03} \\  \hline 
\end{tabular}
\caption{[Linear advection with a discontinuous initial condition] $L^{1}$ relative errors at $T=0.5$ and $T=1$ of RK-DG2 global time-stepping (GTS) and local time-stepping (LTS) schemes.} \label{tab:Advdis} \vspace{-0.2cm}
\end{table}
%
%
%
%
\subsubsection*{Example 2: Burger's equation}
Next, we test the proposed algorithms on the Burgers' equation with a smooth initial condition:  \vspace{-0.1cm}
\begin{equation} \label{eq:BurgersBCs}
u_{t} + \left (\frac{u^{2}}{2}\right )_{x}  =0, \quad  u(x,0)  = \frac{1}{4} + \frac{1}{2} \sin \pi x,  \vspace{-0.1cm}
\end{equation}
in $-1 \leq x \leq 1$. The exact solution of the problem is given by \cite{HJCP87}:  \vspace{-0.1cm}
\begin{equation}
w(x,t)=\sfrac{1}{4} + \sfrac{1}{2} \, v(x-t, \sfrac{t}{2}),  \vspace{-0.1cm}
\end{equation}
in which $v(x,t)$ is the solution of the Burgers' equation with $v(x,0)=\sin \pi x$. We compute $v$ by Newton iterations to solve the characteristic relation:  \vspace{-0.1cm}
$$ v=\sin (\pi x- vt), \quad  0 \leq x <1.  \vspace{-0.1cm}
$$
The solution $v$ in $(-1,0)$ is computed from $v$ in $(0,1)$ via: $v(-x,t) = -v(x,t)$. The solution of \eqref{eq:BurgersBCs} is smooth up to $t=\sfrac{2}{\pi}$ then it develops a moving shock. For details, see \cite{HJCP87}.

We divide the spatial domain into two zones and use the same discretization in space and in time as in Example~1. In Figure~\ref{fig:BurgersSol}, we show the exact solution and the approximate solution by the fourth order LTS algorithm with $\Delta x_{\text{coarse}}=1/40$ and $M=4$. We see that LTS scheme clearly captures the shock with local refinement in space and in time. In Figure~\ref{fig:MassBurgers}, mass evolution as a function of time of different LTS schemes with $\Delta x_{\text{coarse}}=1/40$ and $M=4$ is displayed. The LTS schemes conserve the mass in the region of the LTS interface, and thus in the whole domain. The relative $L^{1}$ errors at $T=0.3$ when the solution is still smooth are shown in Table~\ref{tab:BurgersCR}. Again, the LTS schemes converge at the same order as the associated GTS schemes and the errors are improved as $M$ increases. At $T=1.1,$ the errors in the smooth regions ($0.1$ away from the shock) are as the same magnitude as in the smooth case as displayed in Table~\ref{tab:BurgersShock}. 

\begin{figure}[!http]
\centering
    \begin{subfigure}[b]{0.32\textwidth}
    \centering
        \includegraphics[width=\textwidth]{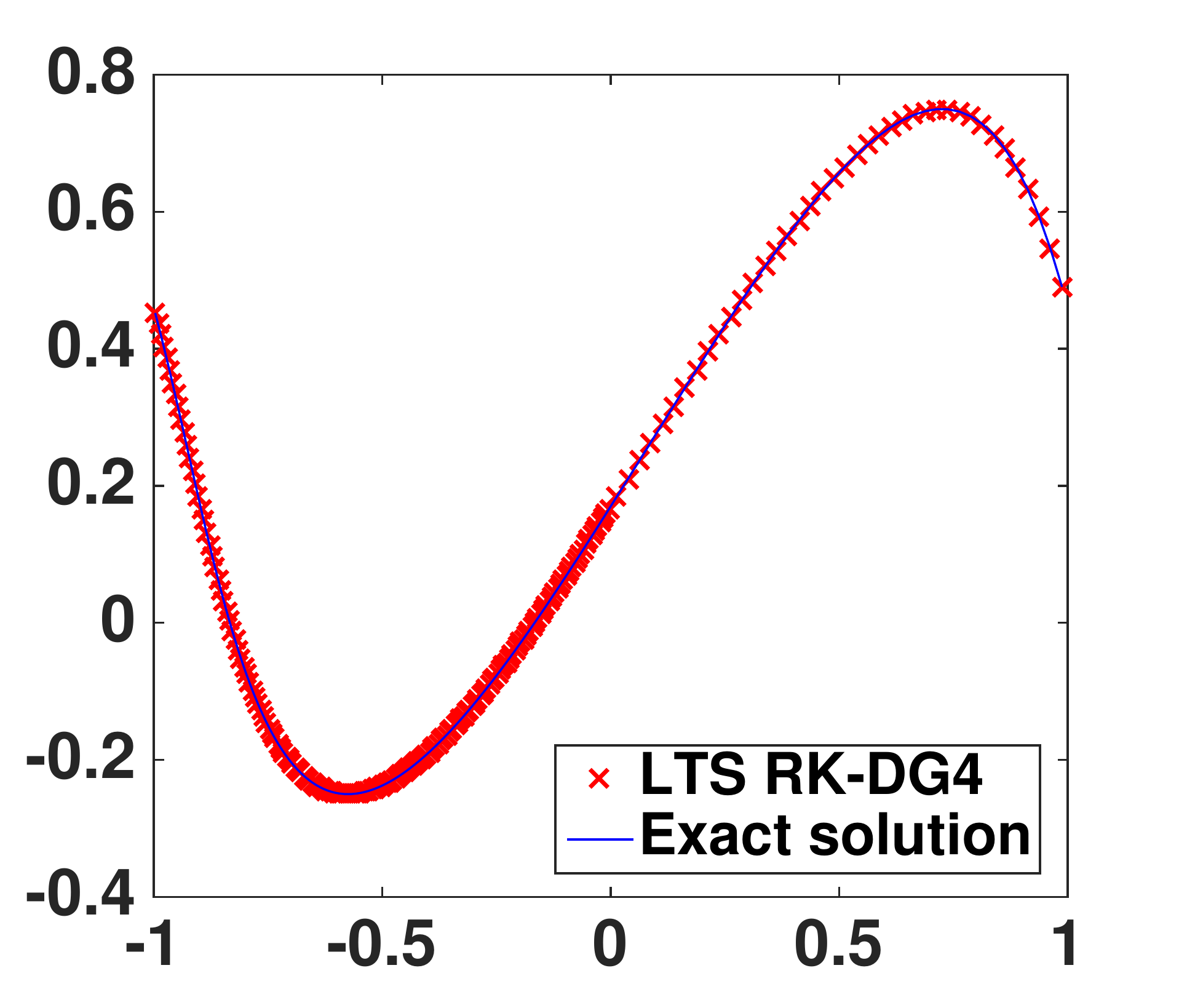}
        \caption{At $T=0.3$}
        \label{fig:Burgers03}
    \end{subfigure}
    \begin{subfigure}[b]{0.32\textwidth}
     \centering
        \includegraphics[width=\textwidth]{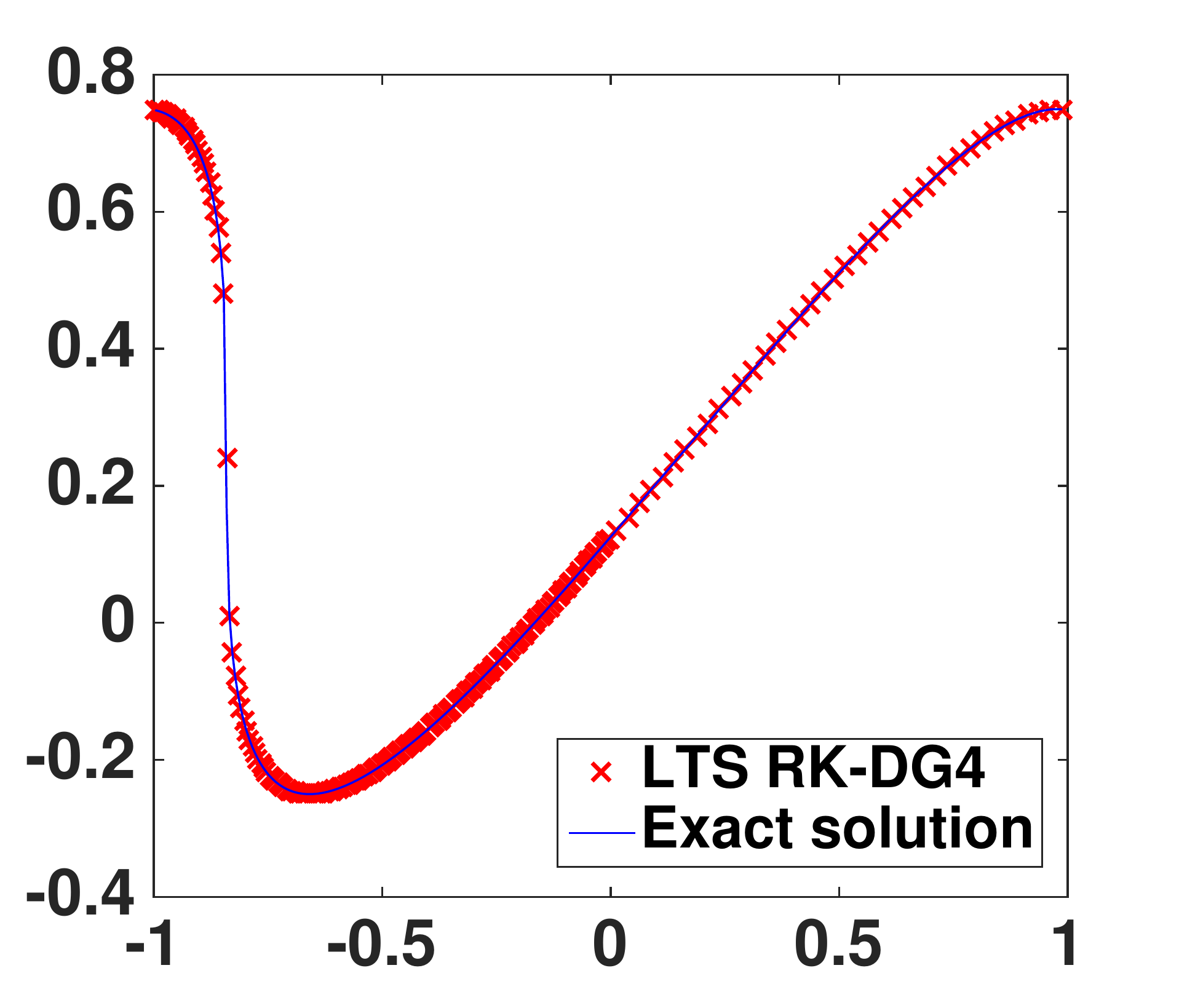}
        \caption{At $T=2/\pi$}
        \label{fig:Burgers2pi}
    \end{subfigure}
    \begin{subfigure}[b]{0.32\textwidth}
     \centering
        \includegraphics[width=\textwidth]{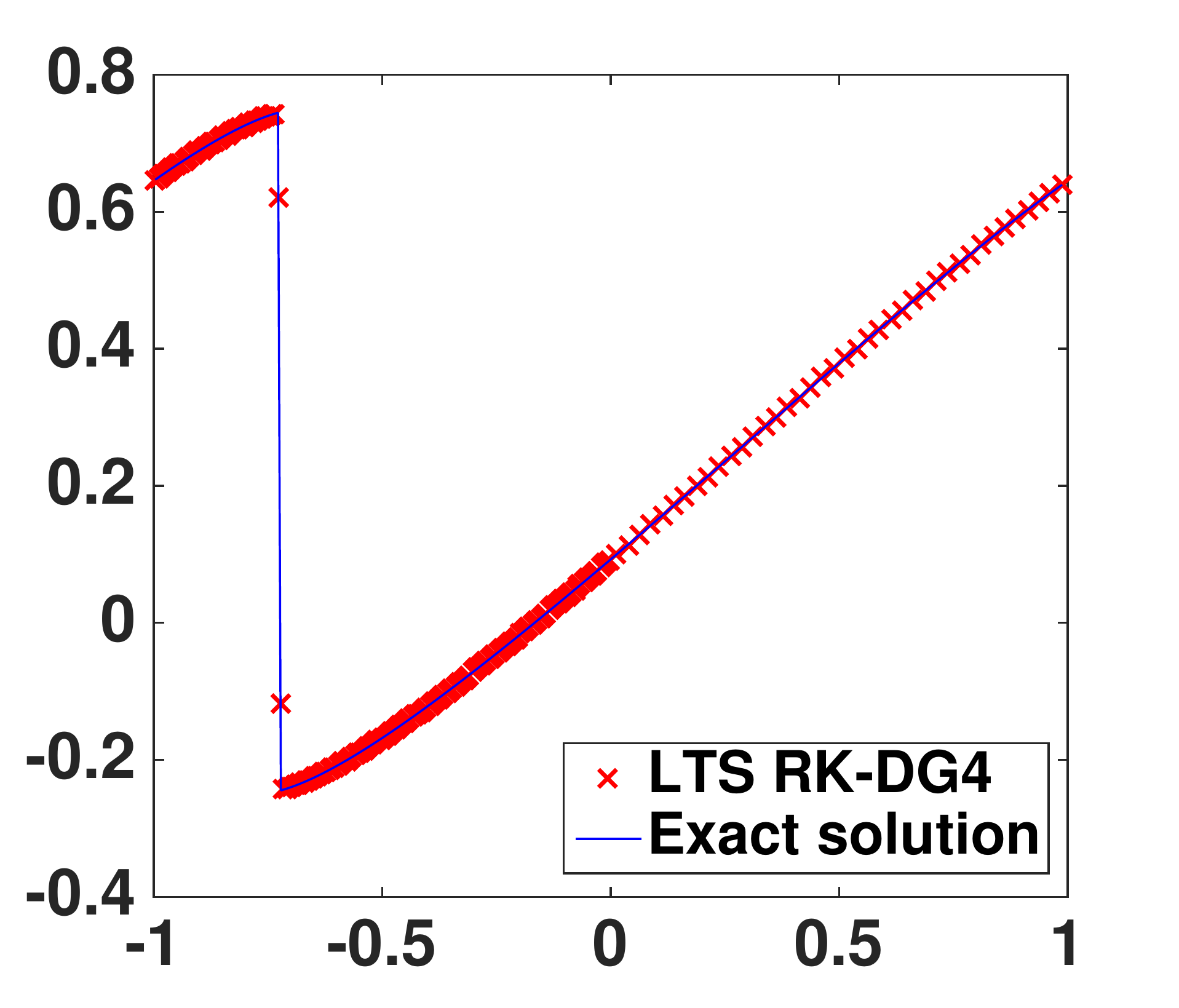}
        \caption{At $T=1.1$}
        \label{fig:Burgers11}
    \end{subfigure} 
    \caption{[Burger's equation] Snapshots of the solution by the fourth order LTS scheme with 2 subdomains, $\Delta x_{1} = \Delta x_{2}/4$ and $\Delta x_{2} =1/40$, }\label{fig:BurgersSol} \vspace{-0.2cm}
\end{figure}

\begin{figure}[!http]
\centering
\includegraphics[scale=0.2]{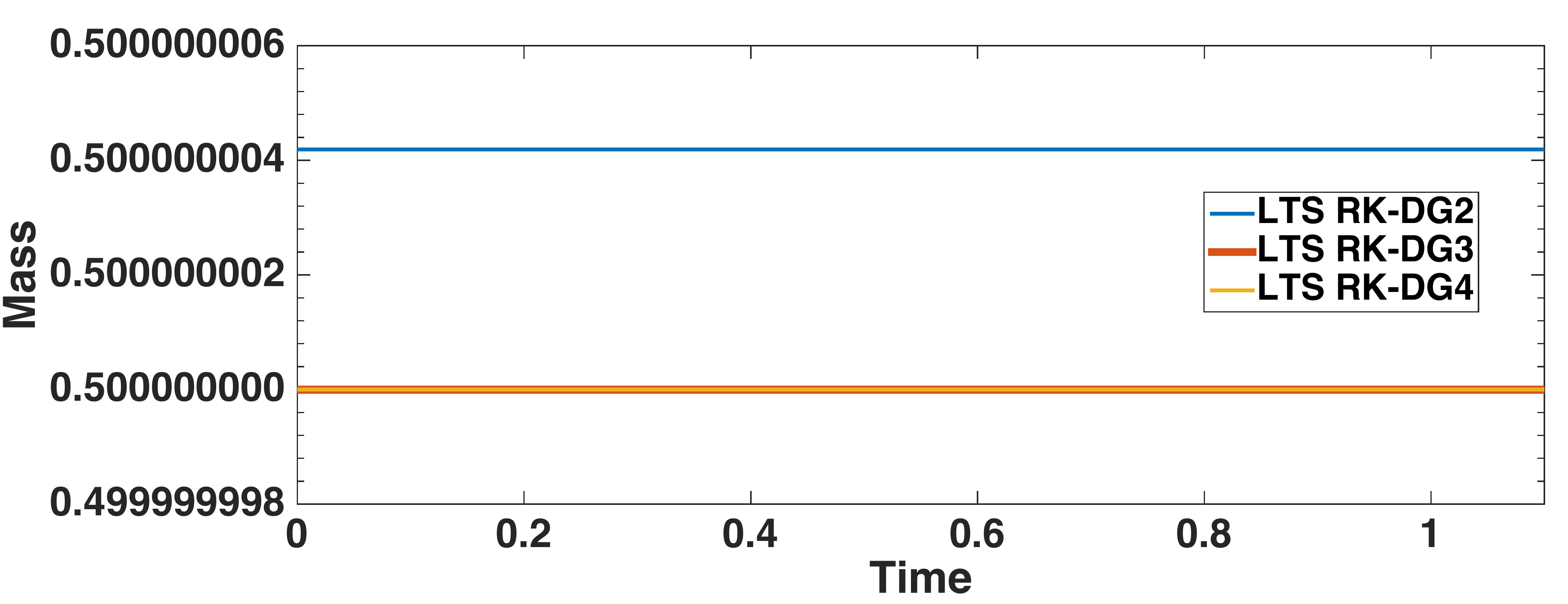}
\caption{[Burger's equation] Time evolution of mass for second, third and fourth order LTS schemes at $\Delta x_{\text{coarse}}=1/40$ and $M=4$. Note that the third and fourth order approximations coincide with each other. } \vspace{-0.2cm}
\label{fig:MassBurgers}
\end{figure}

\begin{table}[!http]\footnotesize
\setlength{\extrarowheight}{3pt}
\centering
\begin{tabular}{|c|c|c c|c c| c c|}
\hline 
$\Delta x_{\text{coarse}}$ & $M$ &  \multicolumn{2}{c}{RK-DG2}& \multicolumn{2}{c|}{RK-DG3} & \multicolumn{2}{c|}{RK-DG4}\\  \cline{3-8}
& & Rel. $L^{1}$ error & [CR] & Rel. $L^{1}$ error & [CR] & Rel. $L^{1}$ error & [CR]   \\ \hline 
\multirow{4}{*}{$1/10$} & $1$  &  6.49e-03   & [1.97] & 3.27e-04   & [3.18] & 2.26e-05  & -- \\ \cline{2-8}
& $2$ &  3.97e-03   & [2.07] & 1.10e-04   &  [3.12] & 6.75e-06   & --  \\ \cline{2-8}
& $4$ &  3.37e-03  & [2.09] &  8.97e-05  &  [3.26] & 5.97e-06  & --  \\ \cline{2-8}
& $8$  &  3.22e-03  & [2.09]  &  8.73e-05  &  [3.28] & 5.93e-06  & --  \\ \hline
\multirow{4}{*}{$1/20$} & $1$  &  1.62e-03   & [2.00]  &  4.10e-05  &  [3.00] & 1.31e-06  & [4.11] \\ \cline{2-8}
& $2$ &  9.63e-04   & [2.04]  & 1.45e-05    &  [2.92] & 4.19e-07 & [4.01]  \\ \cline{2-8}
& $4$ &  8.11e-04  & [2.06]  & 1.17e-05   &  [2.94] &  3.60e-07 & [4.05]  \\ \cline{2-8}
& $8$  &  7.74e-04  & [2.06]  & 1.14e-05   &  [2.94] &  3.57e-07 & [4.05]  \\ \hline
\multirow{4}{*}{$1/40$} & $1$  &  4.03e-04   & [2.01]  & 5.01e-06   &  [3.03] & 8.68e-08  & [3.92] \\ \cline{2-8}
& $2$ &  2.38e-04   & [2.02]  &  1.81e-06  &  [3.00] & 2.63e-08  & [3.99]  \\ \cline{2-8}
& $4$ &  1.99e-04   & [2.03]  &  1.43e-06  &  [3.03] & 2.30e-08   & [3.97]  \\ \cline{2-8}
& $8$  & 1.89e-04   & [2.03]  &  1.40e-06  &  [3.03] & 2.28e-08  & [3.97]  \\ \hline
\multirow{4}{*}{$1/80$} & $1$  & 1.00e-04    & [2.01]  &  6.18e-07   &  [3.02] & 5.44e-09   & [4.00] \\ \cline{2-8}
& $2$ & 5.95e-05    & [2.00]  &  2.25e-07    &  [3.01] & 1.75e-09    & [3.91] \\ \cline{2-8}
& $4$ &  4.95e-05   & [2.01]  & 1.78e-07    &  [3.01] & 1.56e-09   & [3.88] \\ \cline{2-8}
& $8$ &  4.71e-05   & [2.01]  & 1.74e-07    &  [3.01] &  1.55e-09  & [3.88] \\ \hline
\end{tabular}
\caption{[Burger's equation] $L^{1}$ relative errors at $T=0.3$ for different $M$. The rates of convergence ``CR''  for fixed $M$ are shown in square brackets.} \label{tab:BurgersCR} \vspace{-0.2cm}
\end{table}

\begin{table}[!http]\footnotesize
\setlength{\extrarowheight}{3pt}
\centering
\begin{tabular}{|c|c| c|c | c|}
\hline 
\multicolumn{5}{|c|}{At $T=1.1$, errors in smooth region $\| x - \text{shock} \| >=0.1$} \\ \hline
$\Delta x_{\text{coarse}}$ & $M$ & 2nd order LTS & 3rd order LTS & 4th order LTS \\   \hline
\multirow{4}{*}{1/40} & 1 & 6.16e-05  & 4.25e-07   & 4.31e-09 \\ \cline{2-5}
& 2 &  4.67e-05 & 1.73e-07   & 1.71e-09 \\  \cline{2-5}
& 4 &  4.37e-05  & 1.53e-07  & 1.59e-09 \\  \cline{2-5}
& 8 &  4.35e-05   & 1.52e-07   & 1.60e-09 \\  \hline
%
%
%
%
\end{tabular}
\caption{[Burger's equation] $L^{1}$ relative errors in smooth regions at $T=1.1$ of different local time-stepping schemes.} \label{tab:BurgersShock} \vspace{-0.2cm}
\end{table}

\subsection{Euler equations of gas dynamics} \label{subsec:sysNumRe}
We next apply the proposed LTS algorithms to solve a system of one dimensional conservation laws. For the spatial discretization, we employ the DG methods for systems of equations presented in \cite{CSIII} with the local projection limiting in the characteristic fields. The time-stepping is still SSP-RK and thus it is straightforward to apply the proposed LTS algorithms for such a system. We consider the Euler equations of gas dynamics for a polytropic gas: 
\begin{equation}
\bu_{t} + \bff(\bu)_{x} = \pmb{0}, \quad \bu=(\rho, m, E)^{T}, \quad \bff(\bu)=q\bu + (0,P, qP)^{T},
\end{equation}
with $P=(\gamma-1)\left (E-\sfrac{1}{2}\rho q^{2}\right )$. Here $\rho, q, P$ and $E$ are the density, velocity, pressure and total energy, respectively; $m=\rho \, q$ is the momentum and $\gamma$ is the ratio of specific heats. In the following computation, we use $\gamma=1.4$ and present numerical results of applying the second order LTS algorithm to solve Riemann problems of Euler equations and the problem of interaction of blast waves. Note that for these test cases, there is no advantage of using higher order schemes as investigated in \cite{CSIII}. 

\subsubsection*{Example 3: Shock tube problem}
Consider the Riemann problem
\begin{equation}
\bu(x,0)=\left \{ \begin{array}{ll} \bu_{L}, & x<0, \\
\bu_{R}, & x>0,
\end{array} \right .
\end{equation}
with two sets of initial conditions: \vspace{3pt}
\begin{itemize}
\item[a)] The Sod problem \cite{Sod78}: $(\rho_{L}, q_{L}, P_{L})=(1,0,1)$ and $(\rho_{R}, q_{R}, P_{R})=(0.125, 0, 0.10)$;
\item[b)] The Lax problem \cite{Lax54}: $(\rho_{L}, q_{L}, P_{L})=(0.445,0.698,3.528)$ and $(\rho_{R}, q_{R}, P_{R})=(0.5, 0, 0.571)$. 
\end{itemize}

The Sod problem has become a standard test problem of Euler equations with a monotone decreasing density profile. For this problem, we consider two settings of the decomposition into fine and coarse regions: 
\begin{enumerate}
\item two subdomains with the coarse region $[-4.9, -0.5)$ and fine region $[-0.5, 5.1)$;
\item three subdomains with the coarse regions $[-4.9, -2.9)$ and $[4, 5.1]$, and the fine region $[-2.9, 4)$.
\end{enumerate}
The exact solution and approximation solution by the second order LTS algorithm at $T=2.0$ with $\Delta x_{\text{coarse}}=1/5$ and $M=4$ are shown in Figure~\ref{fig:Sod} for the case of two subdomains  and in Figure~\ref{fig:Sod3subs} for the case of three subdomains. The $L^{1}$ relative errors are displayed in Table~\ref{tab:Sod3subs} in which we observe that for both settings, the errors decrease as $M$ increases, especially for the three subdomain case, and the order of convergence is first order as the solution is discontinuous. The three subdomain setting, as expected, has a better performance since the fine region includes the contact discontinuities and the corners of rarefaction waves. 

\begin{figure}[!ht]
    \centering
    \begin{subfigure}[b]{0.32\textwidth}
        \includegraphics[width=\textwidth]{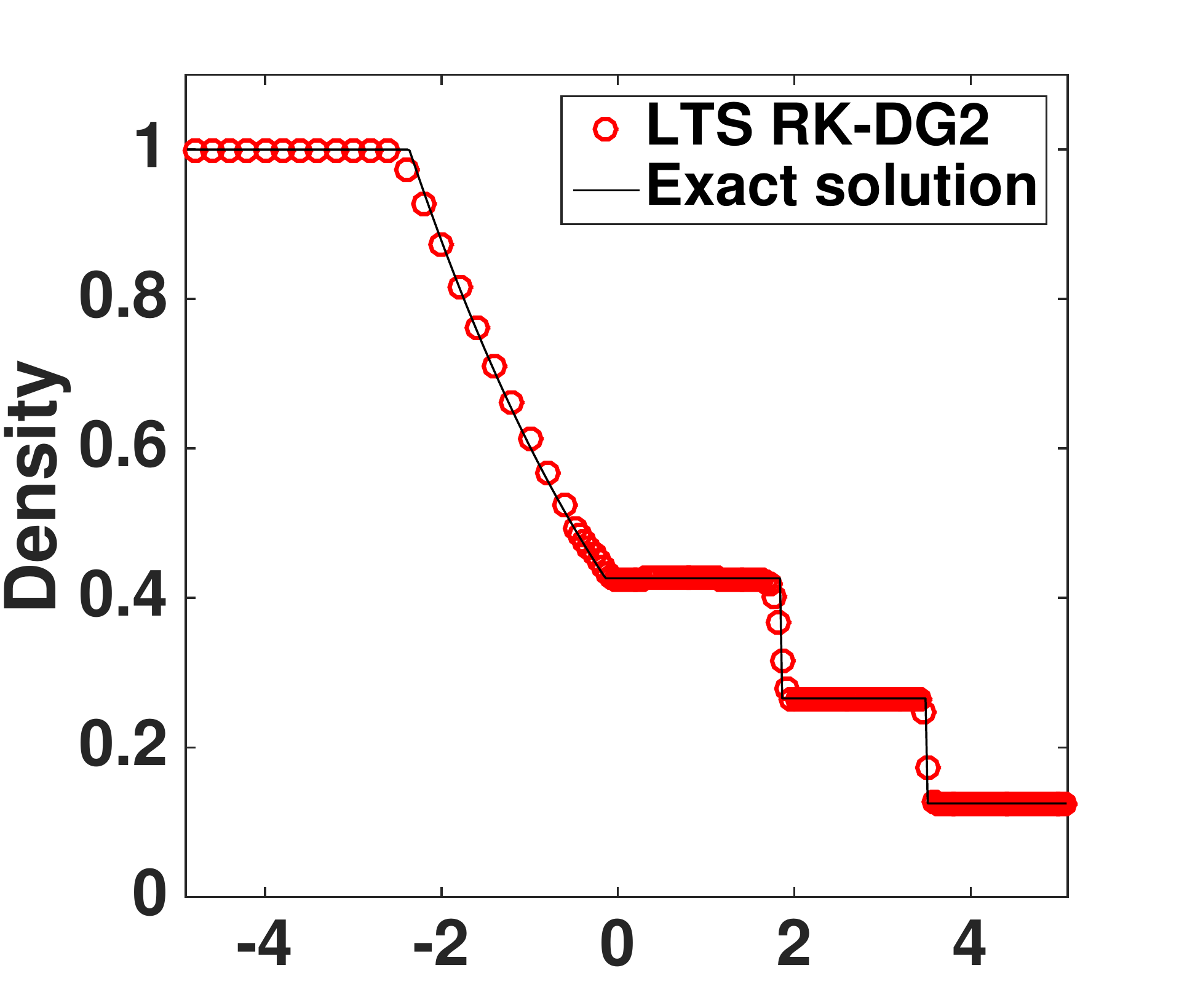}
    \end{subfigure}
    \begin{subfigure}[b]{0.32\textwidth}
        \includegraphics[width=\textwidth]{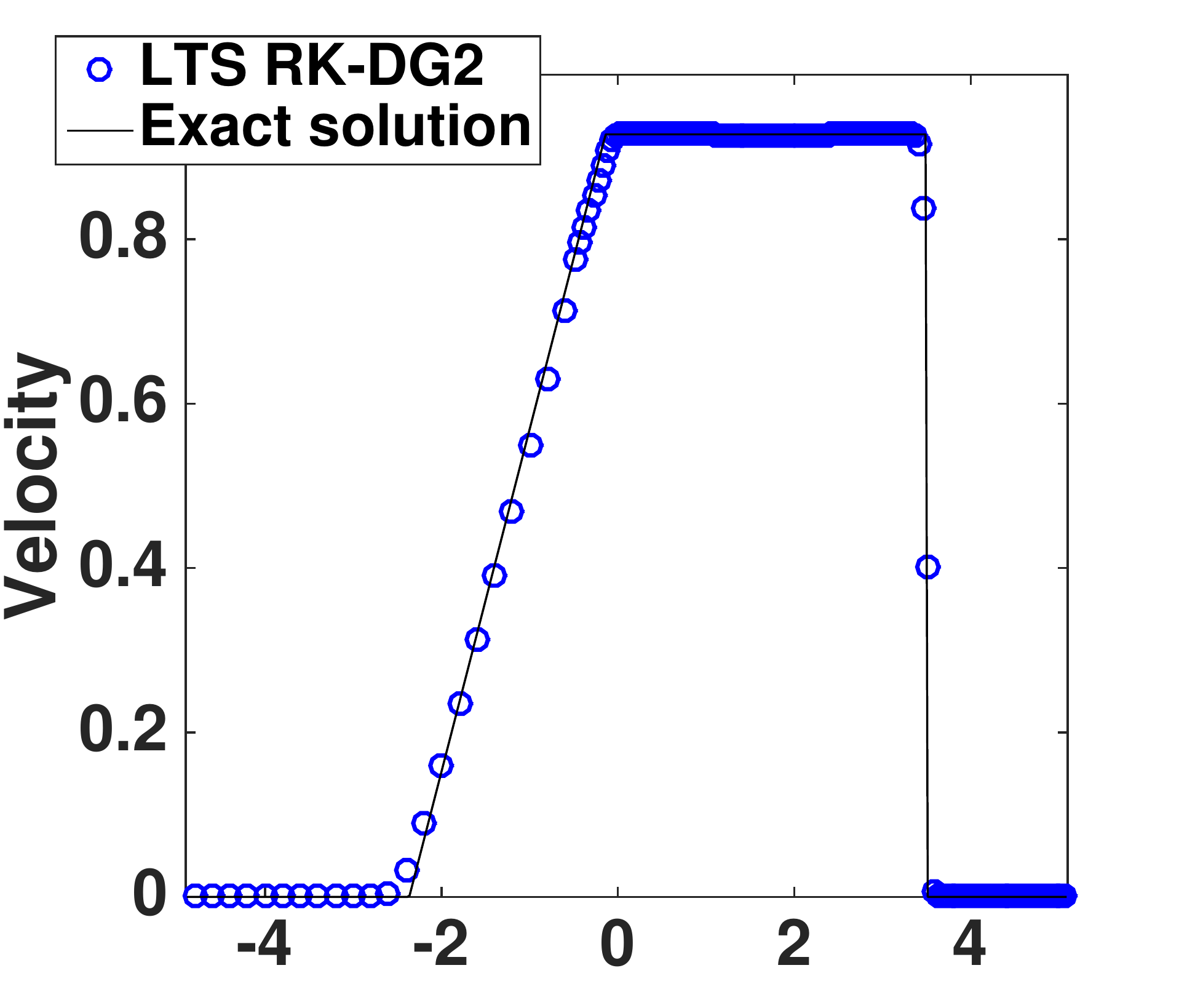}
    \end{subfigure}
    \begin{subfigure}[b]{0.32\textwidth}
        \includegraphics[width=\textwidth]{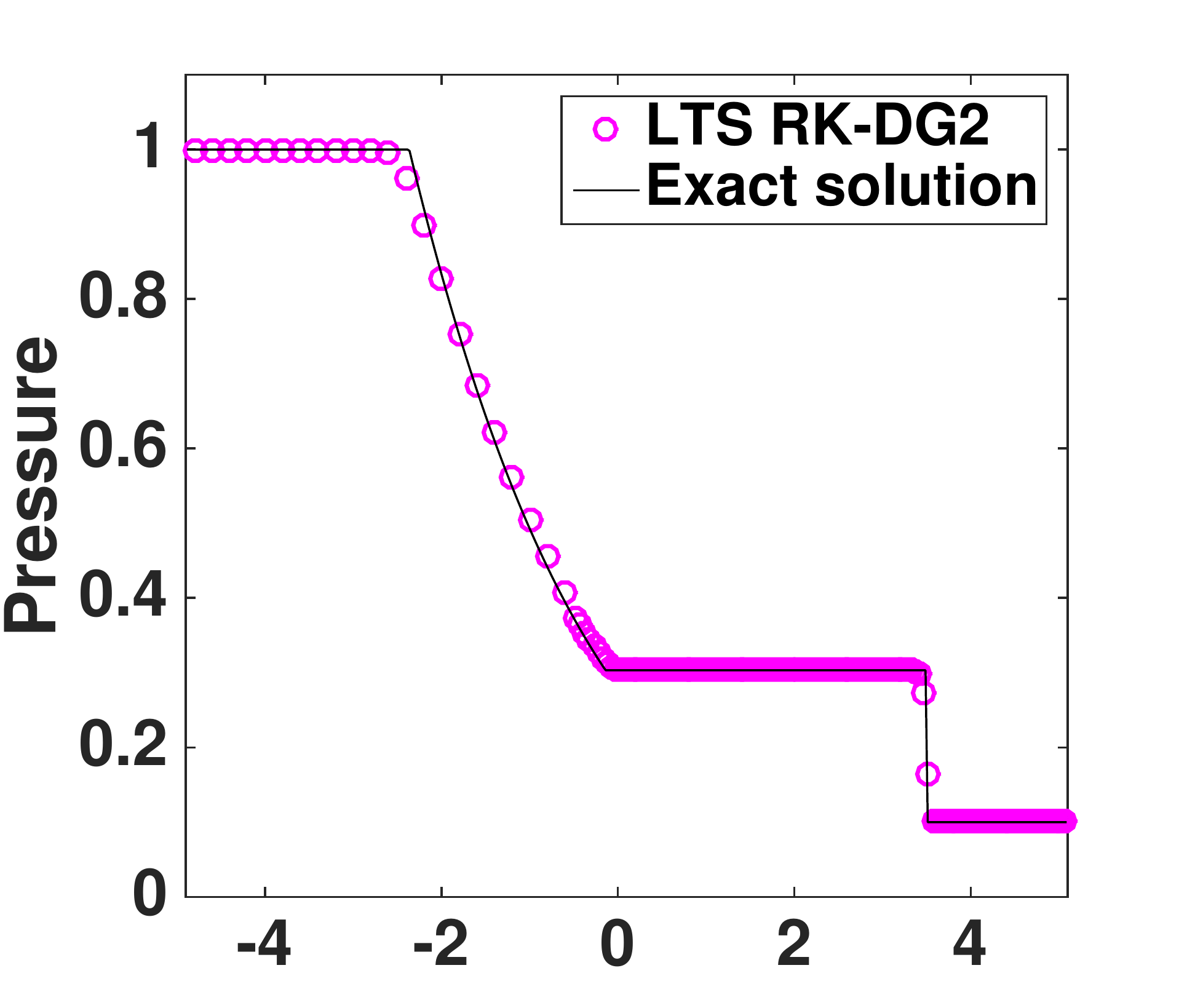}
    \end{subfigure}
    \caption{[Sod shock tube problem] Snapshots of the density, velocity and pressure at $T=2.0$ by the second order LTS scheme with 2 subdomains, $\Delta x_{\text{coarse}}=1/5$ and $M=4$.}\label{fig:Sod} \vspace{-0.2cm}
\end{figure}

\begin{figure}[!ht]
    \centering
    \begin{subfigure}[b]{0.32\textwidth}
        \includegraphics[width=\textwidth]{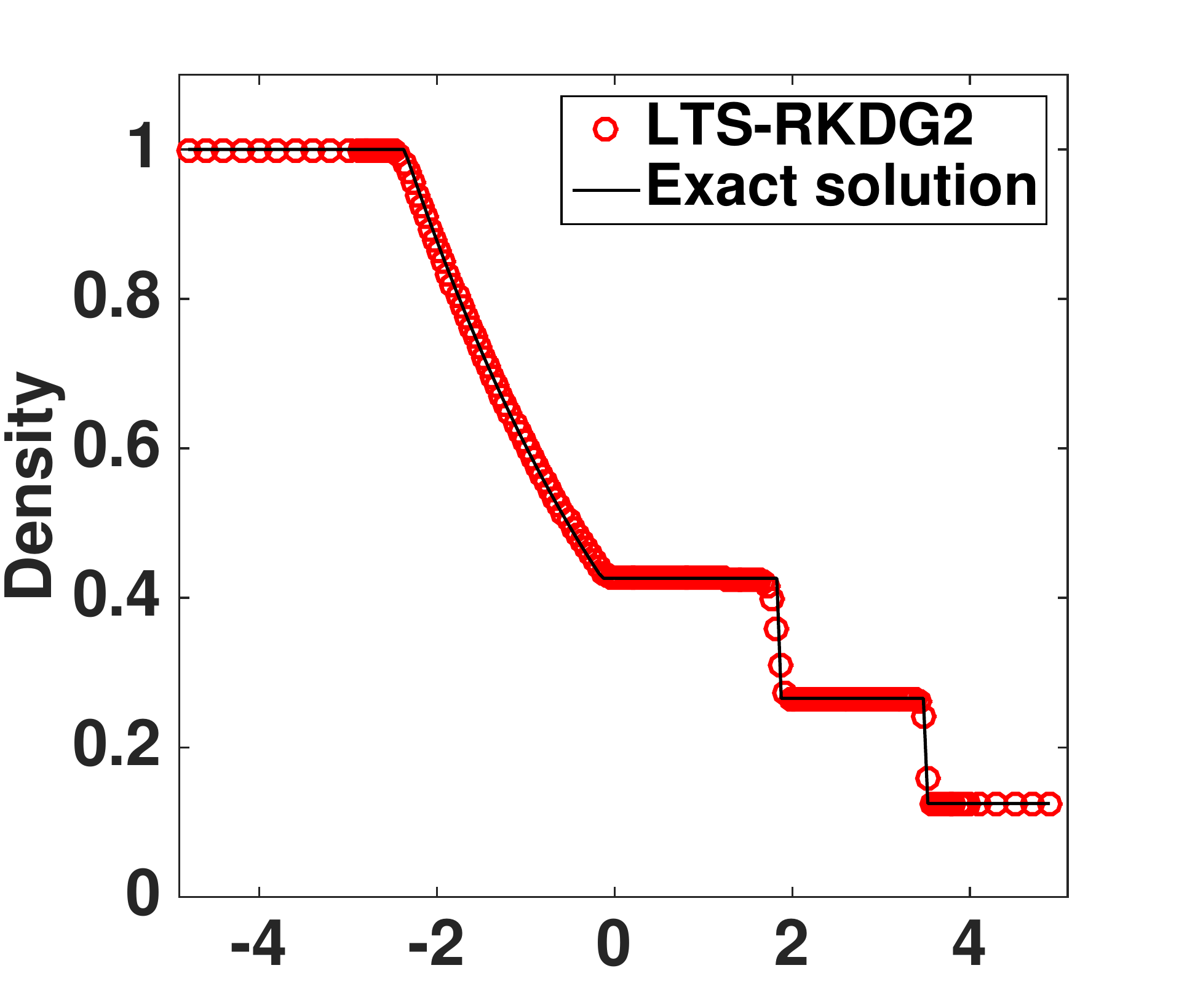}
    \end{subfigure}
    \begin{subfigure}[b]{0.32\textwidth}
        \includegraphics[width=\textwidth]{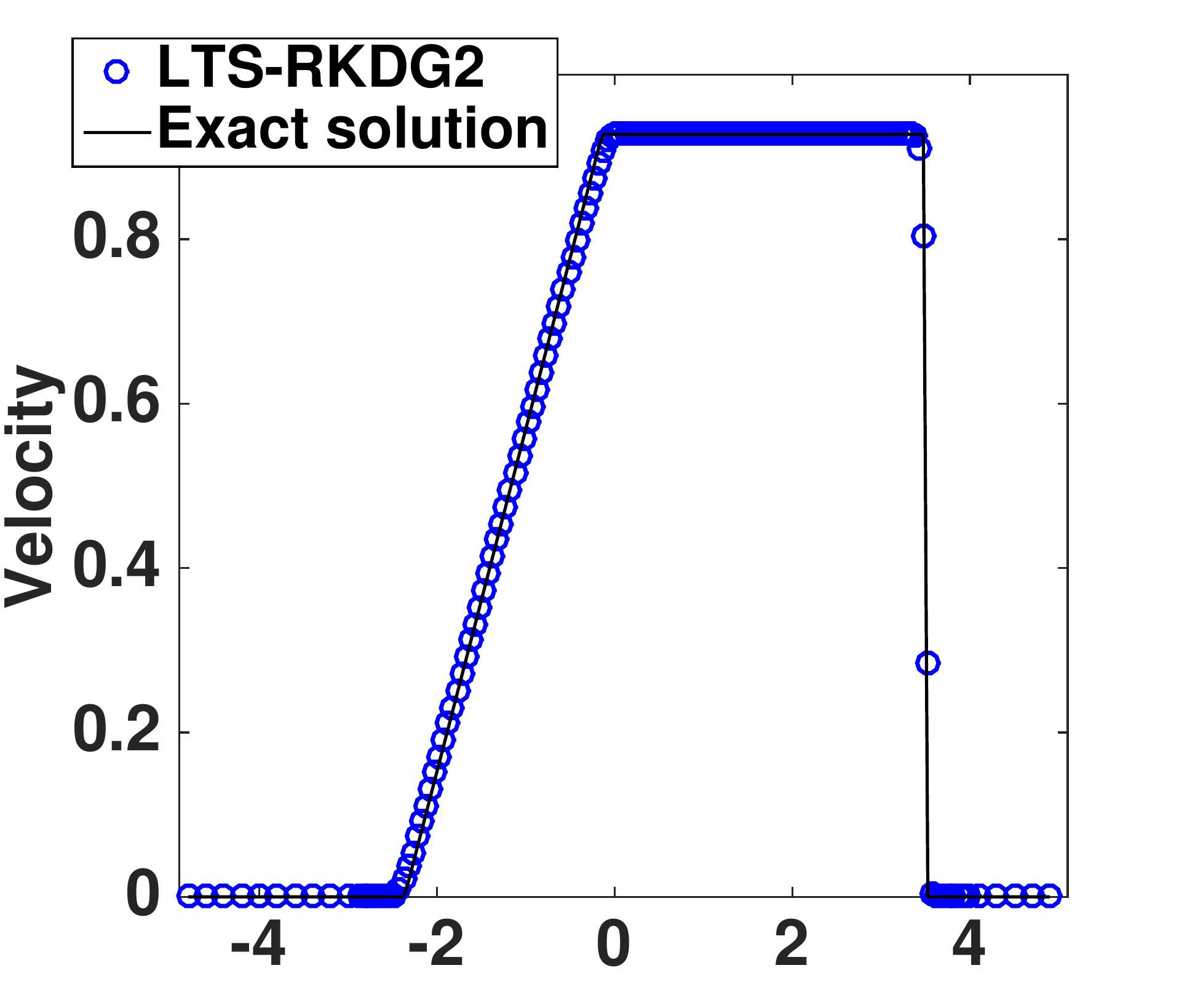}
    \end{subfigure}
    \begin{subfigure}[b]{0.32\textwidth}
        \includegraphics[width=\textwidth]{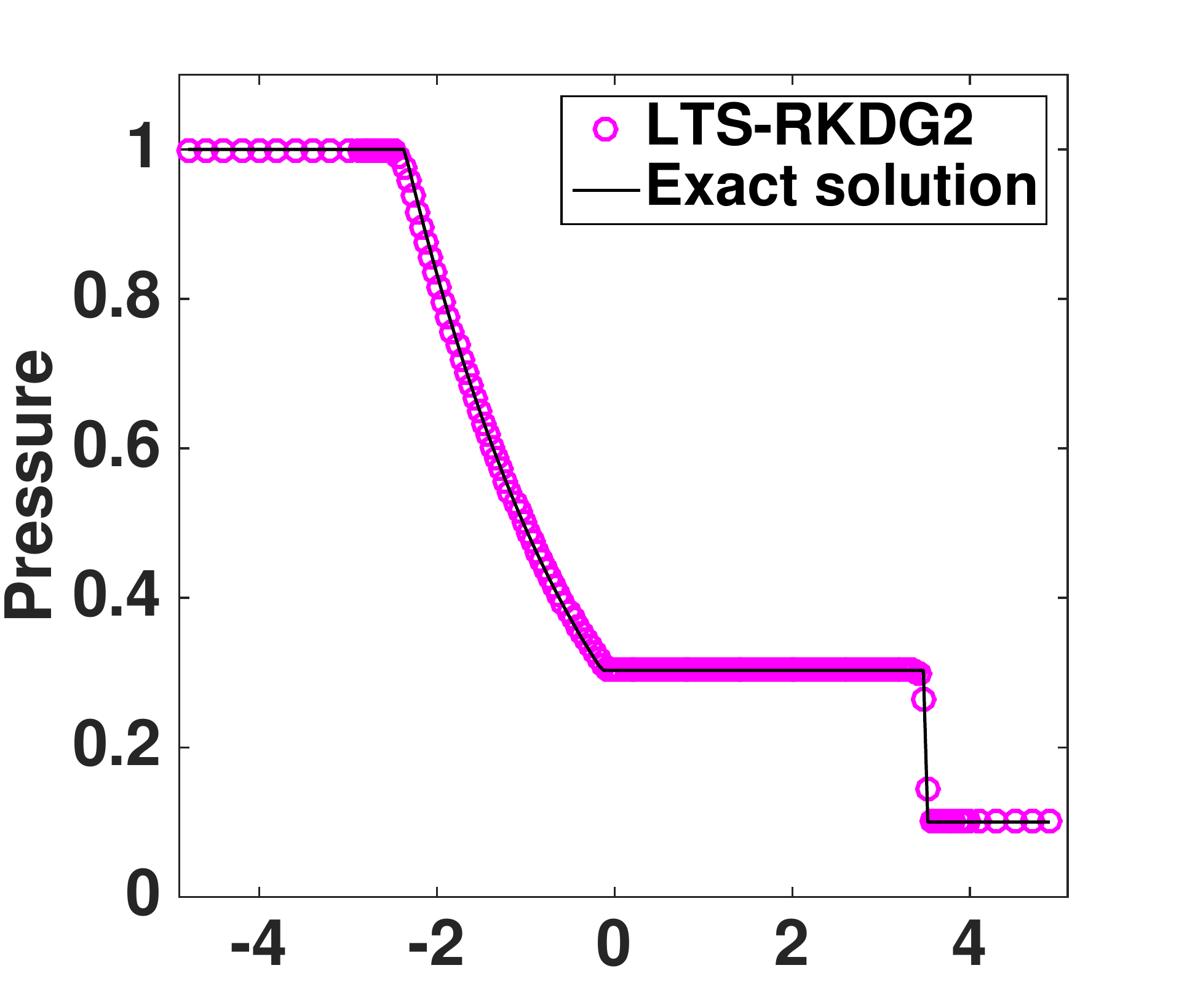}
    \end{subfigure}
    \caption{[Sod shock tube problem] Snapshots of the density, velocity and pressure at $T=2.0$ by the second order LTS scheme with 3 subdomains, $\Delta x_{\text{coarse}}=1/5$ and $M=4$.}\label{fig:Sod3subs} \vspace{-0.2cm}
\end{figure}

\begin{table}[!ht]\footnotesize
\setlength{\extrarowheight}{3pt}
\centering
\begin{tabular}{|c|c|c|c|c|c|c|c|}
\hline 
\multirow{2}{*}{$\Delta x_{\text{coarse}}$} & \multirow{2}{*}{M} &  \multicolumn{2}{c|}{Density} &  \multicolumn{2}{c|}{Velocity} &  \multicolumn{2}{c|}{Pressure} \\ \cline{3-8}
 & & 2 domains & 3 domains & 2 domains & 3 domains & 2 domains & 3 domains \\ \hline
\hline
\multirow{4}{*}{1/5} & 1 &  \multicolumn{2}{c|}{1.79e-02}  & \multicolumn{2}{c|}{3.88e-02} & \multicolumn{2}{c|}{1.76e-02}  \\ \cline{2-8} 
& 2 & 8.96e-03  & 8.28e-03  & 2.03e-02  &1.75e-02  & 8.46e-03   &7.22e-03 \\ \cline{2-8} 
& 4 & 6.35e-03   & 4.25e-03 & 1.39e-02   & 8.82e-03 & 6.82e-03  & 3.62e-03 \\ \cline{2-8} 
& 8 & 6.06e-03 & 2.22e-03 & 1.27e-02  & 4.46e-03 & 7.34e-03  & 1.81e-03 \\ \hline 
\multirow{4}{*}{1/10} & 1 & \multicolumn{2}{c|}{8.28e-03}  & \multicolumn{2}{c|}{1.78e-02}  & \multicolumn{2}{c|}{7.24e-03}   \\ \cline{2-8} 
& 2 & 4.41e-03   & 4.23e-03  &  9.79e-03  & 8.83e-03  & 3.89e-03    & 3.61e-03 \\ \cline{2-8} 
& 4 & 2.88e-03  & 2.21e-03  & 6.11e-03   & 4.46e-03  & 2.76e-03  & 1.80e-03 \\ \cline{2-8} 
& 8 &  2.50e-03 & 1.16e-03 & 5.19e-03  &  2.30e-03 & 2.80e-03   &  9.03e-04 \\ \hline 
\end{tabular}
\caption{[Sod shock tube problem] $L^{1}$ relative errors at $T=2.0$ of the second order LTS algorithm for the Sod shock tube problem.}
\label{tab:Sod3subs} \vspace{-0.2cm}
\end{table}

For the Lax problem, the density profile has a ``built-up'' intermediate state, thus we divide the domain into two subdomains with the coarse region $[-4.9, 0)$ and the fine region $[0, 5.1)$. The second order LTS approximate solution at $T=1.3$ with $\Delta x_{\text{coarse}}=1/5$ and $M=4$ is shown in Figure~\ref{fig:Lax} together with the exact solution. We observe that the LTS scheme captures very well the ``built-up'' density profile with local refinement in space and in time. The $L^{1}$ relative errors are presented in Table~\ref{tab:Lax} that match our expectation. 

\begin{figure}[!ht]
    \centering
    \begin{subfigure}[b]{0.32\textwidth}
        \includegraphics[width=\textwidth]{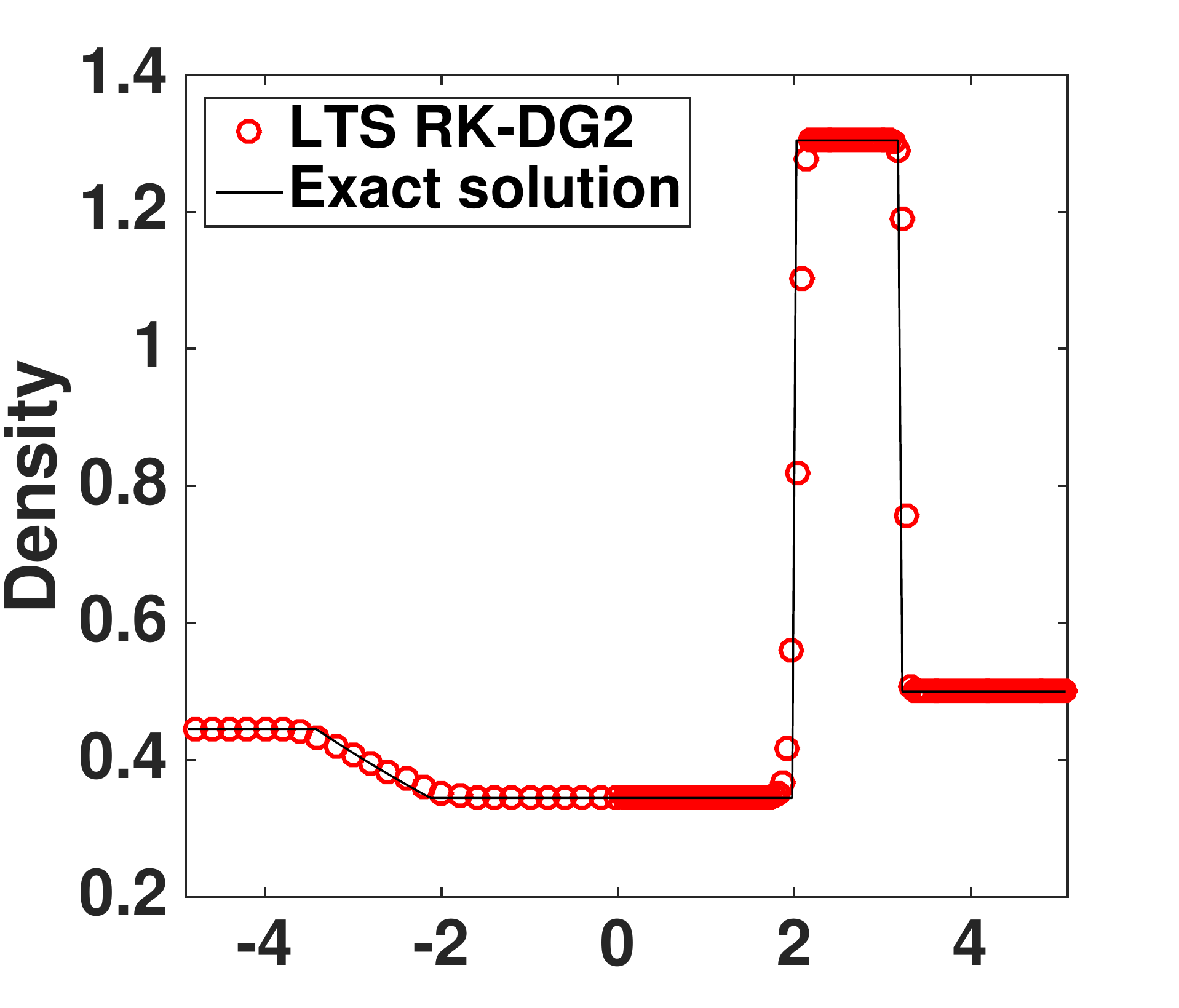}
    \end{subfigure}
    \begin{subfigure}[b]{0.32\textwidth}
        \includegraphics[width=\textwidth]{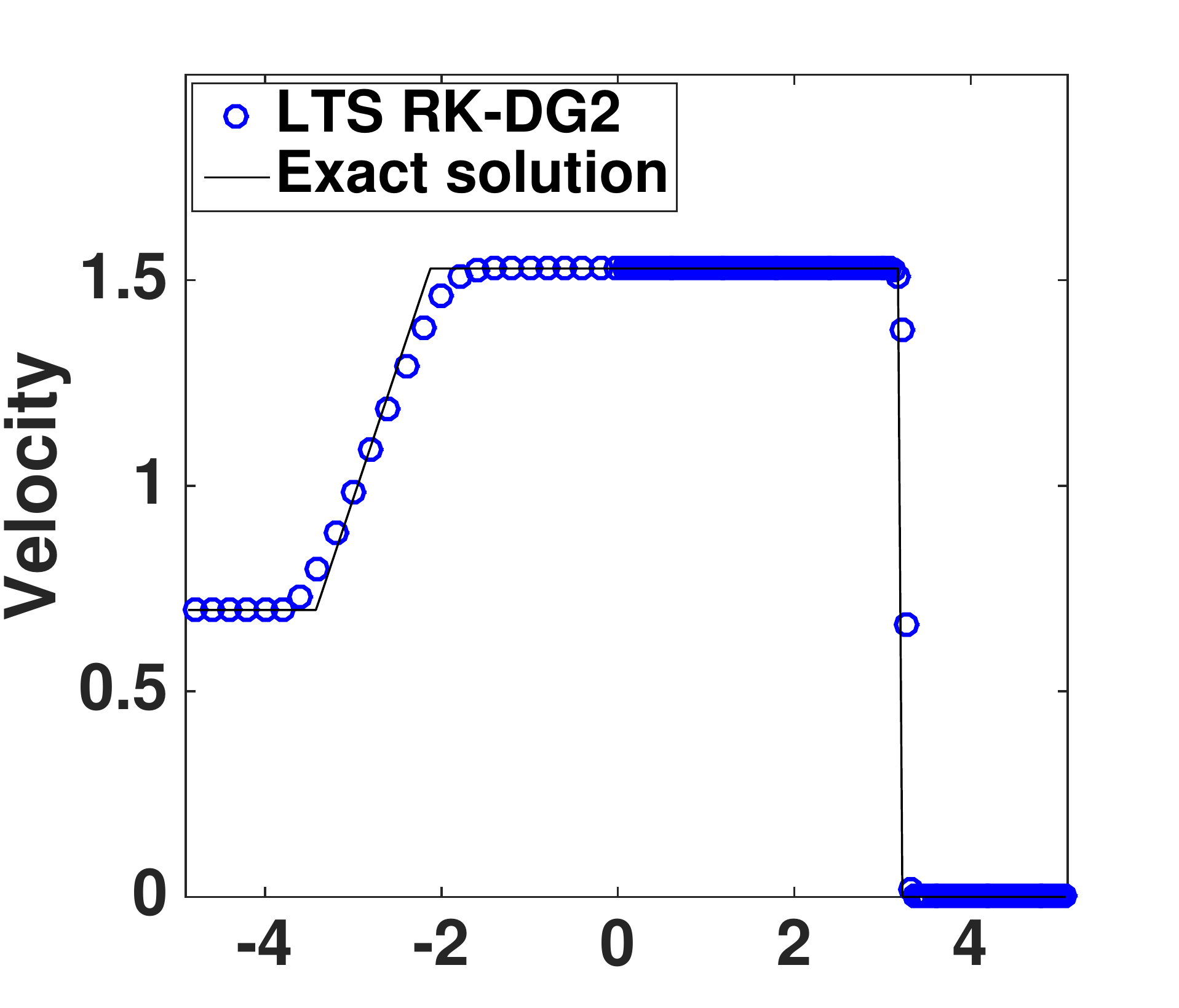}
    \end{subfigure}
    \begin{subfigure}[b]{0.32\textwidth}
        \includegraphics[width=\textwidth]{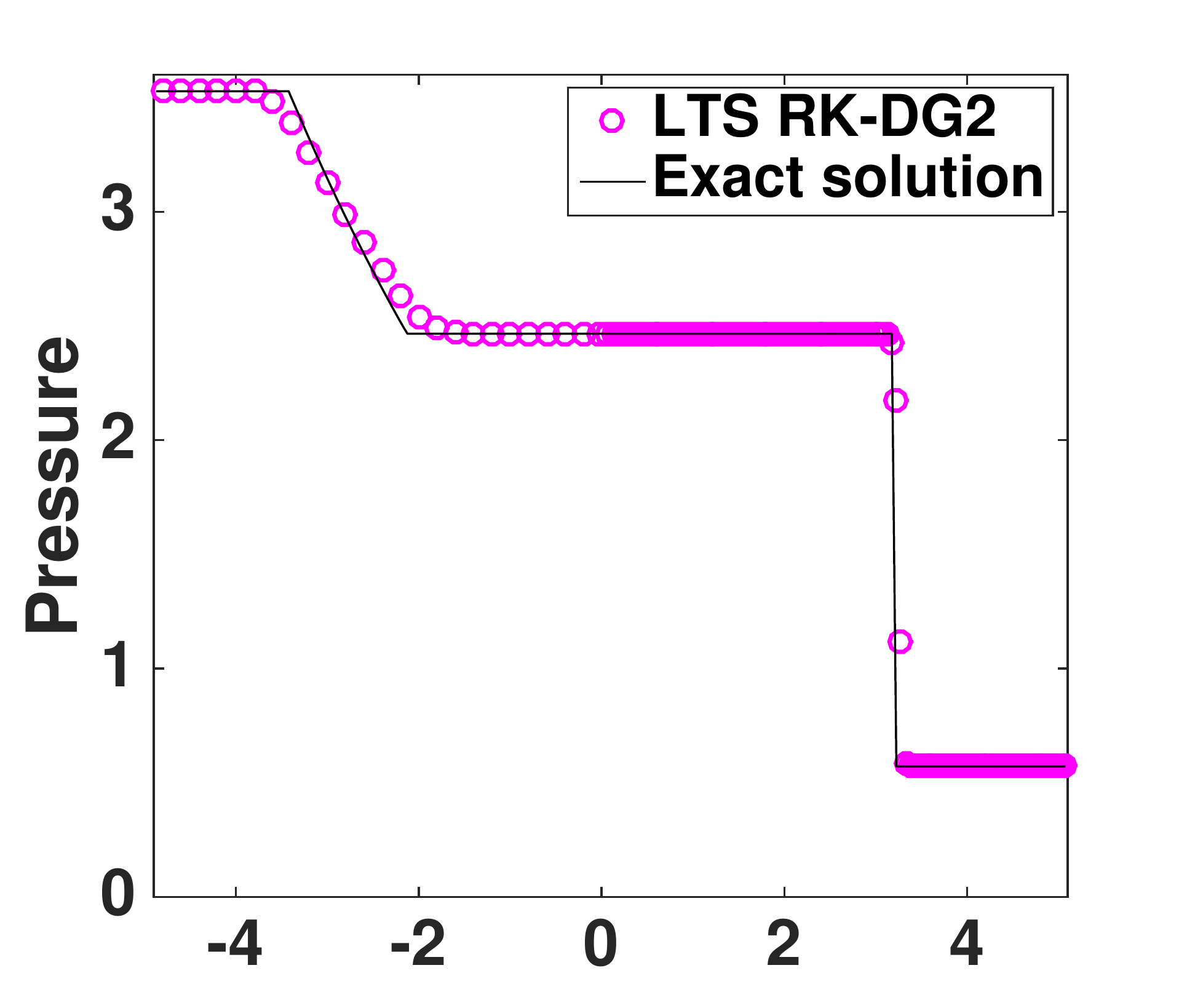}
    \end{subfigure}
    \caption{[Lax shock tube problem] Snapshots of the density, velocity and pressure at $T=1.3$ by the second order LTS scheme with two subdomains, $\Delta x_{\text{coarse}}=1/5$ and $M=4$.}\label{fig:Lax} \vspace{-0.2cm}
\end{figure}

\begin{table}[!ht]\footnotesize
\setlength{\extrarowheight}{3pt}
\centering
\begin{tabular}{|c|c|c|c|c|}
\hline 
$\Delta x_{\text{coarse}}$& M & Density & Velocity & Pressure \\ \hline \hline
\multirow{4}{*}{1/5} & 1 & 4.74e-02  &  2.15e-02 & 1.44e-02 \\ \cline{2-5} 
& 2 & 2.94e-02  & 1.69e-02  & 1.07e-02 \\ \cline{2-5} 
& 4 & 2.11e-02  & 1.73e-02  & 9.12e-03 \\ \cline{2-5} 
& 8 & 1.18e-02   & 1.37e-02  & 7.63e-03 \\ \hline 
\multirow{4}{*}{1/10} & 1 & 2.52e-02 & 1.20e-02  & 6.55e-03  \\ \cline{2-5} 
& 2 & 1.51e-02 & 9.07e-03 & 4.74e-03 \\ \cline{2-5} 
& 4 & 8.88e-03 & 7.13e-03  & 3.91e-03\\ \cline{2-5} 
& 8 & 5.70e-03  & 7.41e-03   & 4.19e-03 \\ \hline 
\end{tabular}
\caption{[Lax shock tube problem] $L^{1}$ relative errors at $T=1.3$ of the second order LTS algorithm.}
\label{tab:Lax} \vspace{-0.2cm}
\end{table}

\subsubsection*{Example 4: Interaction of blast waves}
We finally consider the problem of two interacting blast waves: 
\begin{equation}
\bu(x,0)=\left \{ \begin{array}{ll} \bu_{L}, & 0\leq x<0.1, \\
\bu_{M}, & 0.1 \leq x < 0.9, \\
\bu_{R}, & 0.9 \leq x<1,
\end{array} \right .
\end{equation}
with $(\rho_{L}, q_{L}, P_{L})=(1,0,10^{3})$, $(\rho_{M}, q_{M}, P_{M})=(1,0, 10^{-2})$ and $(\rho_{R}, q_{R}, P_{R})=(1, 0, 100)$. Reflection boundary conditions are applied at $x=0$ and $x=1$.  For details, see \cite{WC84, HJCP87}. 

We divide $\Omega=[0,1]$ into 3 subdomains $\Omega_{1}=[0, 0.2)$, $\Omega_{2}=[0.2, 0.9)$ and $\Omega_{3}=[0.9, 1]$. The mesh and time step sizes in $\Omega_{2}$ are $M$ times smaller than those in $\Omega_{1}$ and $\Omega_{3}$.  The solutions at $T=0.038$ with the second order global time-stepping ($M=1$) and local time-stepping ($M=2$) are shown in Figure~\ref{fig:blastwave}. We see that the LTS scheme with a local refinement in space and time gives a much better resolution, especially for the density profile.

\begin{figure}[!ht]
    \centering
    \begin{subfigure}[b]{0.32\textwidth}
        \includegraphics[width=\textwidth]{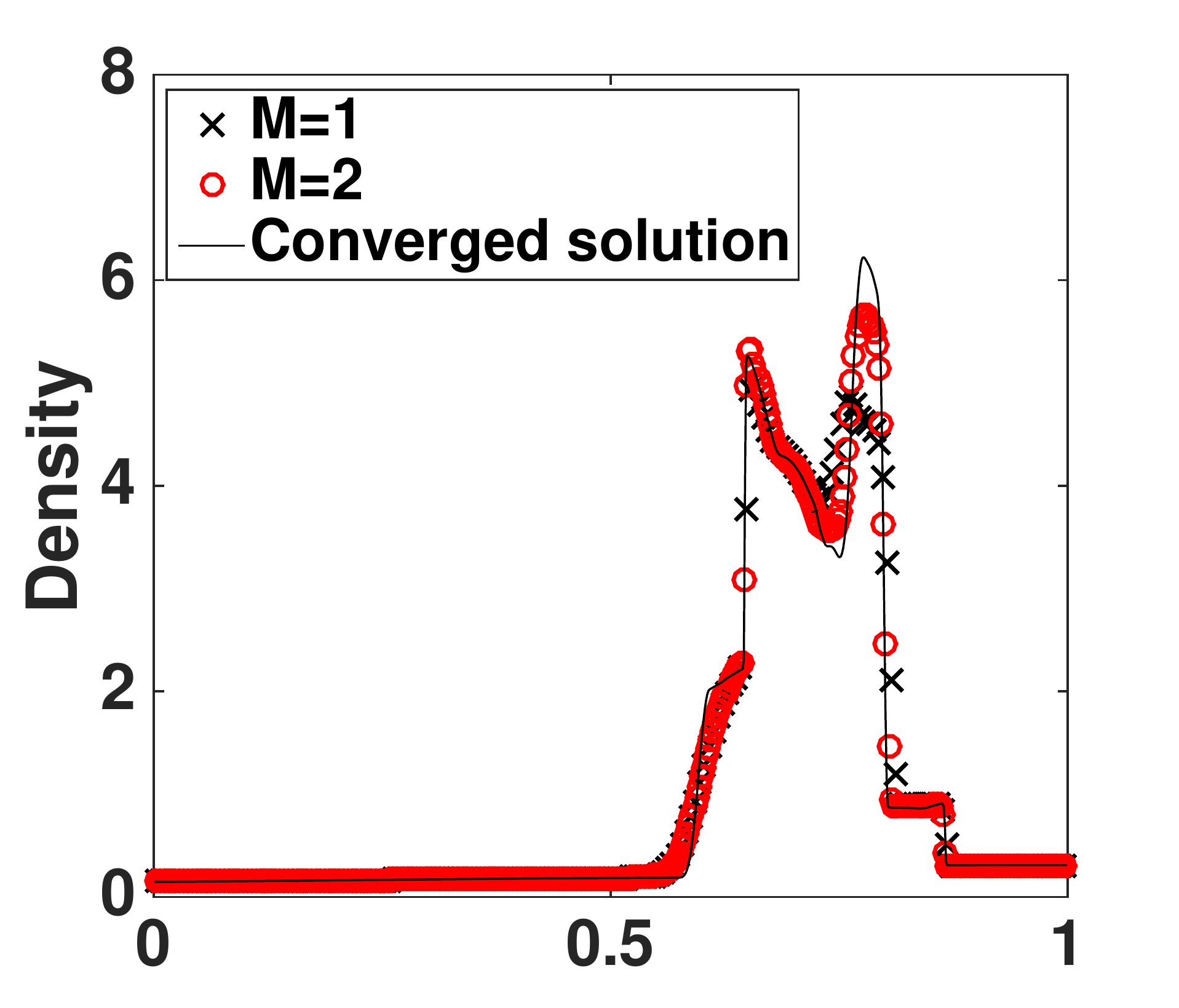}
    \end{subfigure}
    \begin{subfigure}[b]{0.32\textwidth}
        \includegraphics[width=\textwidth]{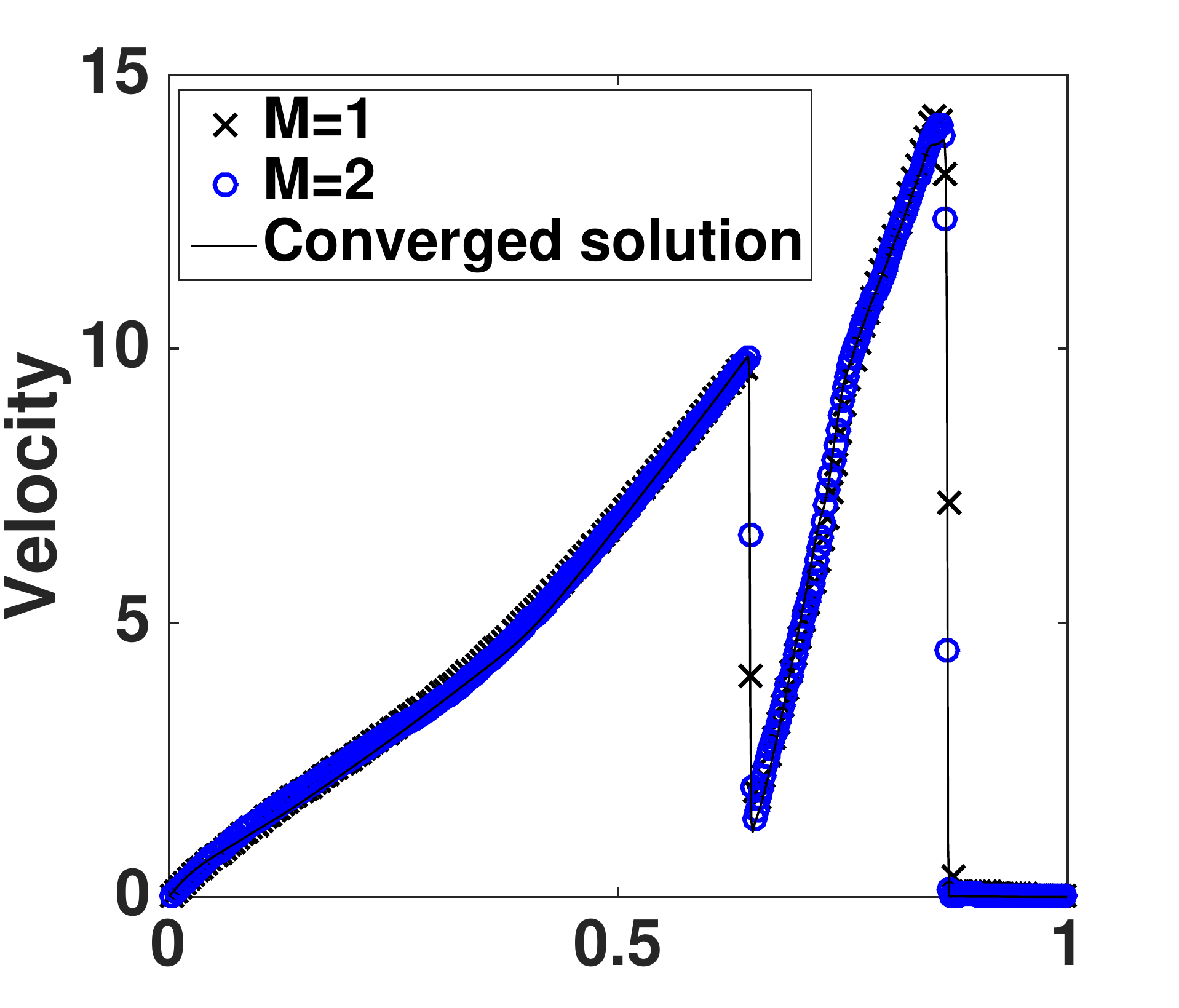}
    \end{subfigure}
    \begin{subfigure}[b]{0.32\textwidth}
        \includegraphics[width=\textwidth]{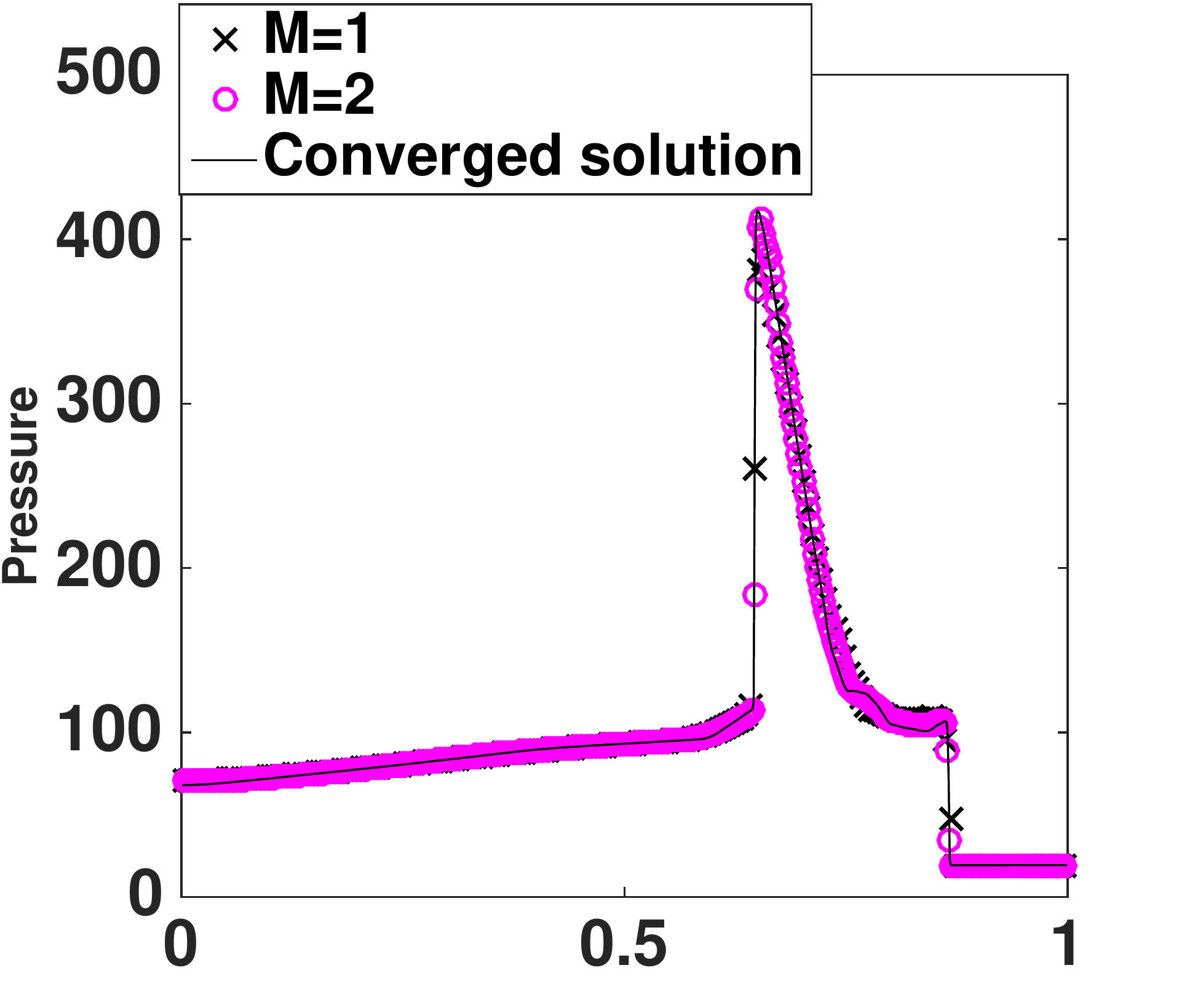}
    \end{subfigure}
    \caption{[Blast waves' interaction problem] Snapshots of the density, velocity and pressure at $T=0.038$ by the second order GTS ($M=1$) and LTS ($M=2$) schemes with $\Delta x_{\text{coarse}}=1/200$. }\label{fig:blastwave} \vspace{-0.2cm}
\end{figure}

\section{Conclusion}
In this work, high order explicit local time-stepping algorithms have been proposed and analyzed for hyperbolic conservation laws. The approaches are of predictor-corrector type, and algorithms of up to fourth order accuracy are constructed in a general setting of Runge-Kutta discontinuous Galerkin methods with the modified {\em minmod} limiter. With our LTS schemes, different time-step sizes can be used based on a local CFL condition instead of the more restrictive global CFL condition. Thus, they outperform the global time-stepping for simulations on multi-resolution meshes or of multiple scales. In addition, we rigorously prove the conservation property and nonlinear stability of these schemes. Numerical results confirm their accuracy and efficiency. Future work includes the coupling of adaptive multi-resolution meshes with our local time-stepping to carry out simulations in parallel and further investigations on the numerical performance of the proposed approaches for large scale simulations on modern supercomputer systems. 

\bibliographystyle{siam}

\section*{Appendices}
\begin{appendices}
\renewcommand{\theequation}{A.\arabic{equation}}
\setcounter{equation}{0}  

\section{SSP-RK time stepping schemes}\label{SSPRKappend}
We first present the SSP-RK(2,2) and SSP-RK(3,3) schemes for solution of the system \eqref{eq:auto}, which are \emph{optimal} in the sense that the number of stages equals the order of accuracy and the coefficients $\beta_{i,\nu}$ are nonnegative~\cite{GS98}. In both cases, the schemes possess the SSP coefficient $\iC=1$.

\emph{Second order SSP-RK(2,2)}: $\alpha_{10}=\beta_{10}=1$, $\alpha_{20}=\alpha_{21}=\beta_{21}=\sfrac{1}{2}$ and $\beta_{20}=0$, which is equivalent to the Heun's method:
\begin{equation} \label{eq:SSPRK2}
\begin{array}{ll}
\bU_{h}^{n,(1)} &\hspace{-0.2cm}=\bU_{h}^{n} +\Delta t \bL_{h}(\bU_{h}^{n}), \vspace{0.1cm}\\
\bU_{h}^{n+1}&\hspace{-0.2cm}=\sfrac{1}{2}\, \bU_{h}^{n} + \sfrac{1}{2} \left (\bU_{h}^{n,(1)} +\Delta t \bL_{h}(\bU_{h}^{n,(1)})\right). 
\end{array}
\end{equation}

\emph{Third order SSP-RK(3,3)}: $\alpha_{10}=\beta_{10}=1$, $\alpha_{20}=\sfrac{3}{4}, \, \alpha_{21}=\beta_{21}=\sfrac{1}{4}$, $\alpha_{30}=\sfrac{1}{3}, \, \alpha_{32}=\beta_{32}=\sfrac{2}{3}$ and $\beta_{20}=\beta_{30}=\alpha_{31}=\beta_{31}=0$, or explicitly:
\begin{equation}
\begin{array}{ll}
\bU_{h}^{n,(1)}   &\hspace{-0.2cm}=\bU_{h}^{n} +\Delta t \bL_{h}(\bU_{h}^{n}), \vspace{0.1cm}\\
\bU_{h}^{n,(2)}&\hspace{-0.2cm}=\sfrac{3}{4}\, \bU_{h}^{n} + \sfrac{1}{4} \left (\bU_{h}^{n,(1)} +\Delta t \bL_{h}(\bU_{h}^{n,(1)})\right), \vspace{0.1cm}\\
\bU_{h}^{n+1}&\hspace{-0.2cm}=\sfrac{1}{3}\, \bU_{h}^{n} + \sfrac{2}{3} \left (\bU_{h}^{n,(2)} +\Delta t \bL_{h}(\bU_{h}^{n,(2)})\right). 
\end{array}
\end{equation}


%
For higher order schemes $r\geq 4$, we can not avoid negative $\beta_{i\nu}$ without using additional stages. We shall use the SSP-RK(5,4) proposed in \cite{KYK14} with the SSP coefficient $\iC=1.652$, and the coefficients of the scheme are listed in Table~\ref{tab:RK4}.
\begin{table}[!ht]\scriptsize
\setlength{\extrarowheight}{3pt}
\begin{tabular}{c  c  c  c  c} 
\hline 
$\alpha_{i\nu}$ &  &  &  & \\
1 & 0  & 0 & 0 & 0\\
0.261216512493821   & 0.738783487506179 & 0 & 0 & 0\\
0.623613752757655 & 0 & 0.376386247242345 & 0  & 0\\
0.444745181201454 & 0.120932584902288 & 0 & 0.434322233896258 & 0 \\
0.213357715199957 & 0.209928473023448 & 0.063353148180384 & 0 & 0.513360663596212 \\
$\beta_{i\nu}$ &  &  &  &  \\
0.605491839566400 & 0  & 0 & 0 & 0\\
0 & 0.447327372891397 & 0 & 0 & 0\\
0.000000844149769 & 0 & 0.227898801230261 & 0 & 0  \\
0.002856233144485 & 0.073223693296006 & 0 & 0.262978568366434 & 0\\
0.002362549760441 & 0.127109977308333 & 0.038359814234063 & 0  & 0.310835692561898 \\ \hline
\end{tabular}
\caption{\cite{KYK14} Coefficients of the SSP-RK (5,4) scheme with $\iC=1.652$.} \label{tab:RK4} \vspace{-0.4cm}
\end{table}

\renewcommand{\theequation}{B.\arabic{equation}}
\setcounter{equation}{0}  

\section{Derivation of the predictors} \label{PREDappend}
We derive the predictors up to fourth order accuracy $(k=1,2,3)$ in the DG-RK setting. To simplify the notation, in the following, we consider an arbitrary equation in the system \eqref{eq:auto} and write it in the form:
\begin{equation} \label{eq:autoj}
\partial_{t} w_{j} = L_{j}(\bw), 
\end{equation}
where $\bw =\left (w_{j}(t)\right )_{\forall\, j}$ ($w_{j}$ represents $u_{j}^{(l)}$ and we have dropped the superscript $(l)$ for simplicity) and $L_{j}$ is a multivariable, real-valued function. 

Given the time partition with coarse and fine time steps as defined in Section~\ref{sec:LTS} and assume that the solution $\bw^{n}$ at $t^{n}$ is known, we shall construct the approximation of $\bw$ at the interface $x_{j+\ahalf}$ at the intermediate time levels $t^{n,p}$ for $p=1, \ldots, M-1$. Performing Taylor expansion of $w_{j}$ at $t^{n}$ yields:
\begin{equation} \label{aTaylor}
w_{j}(t) = w_{j}(t^{n}) + (t-t^{n}) \frac{dw_{j}}{dt}(t^{n}) + \ldots + \frac{1}{k!} (t-t^{n})^{k} \frac{d^{(k)}w_{j}}{dt}(t^{n}) + O\left ((\Delta t)^{k+1}\right ).
\end{equation}
The time derivatives of $w_{j}$ up to order $k$ are approximated by the SSP-RK solution of the first $(s-1)$ stages with a coarse time step, $w_{j}^{n, (i)}$ for $i=1, \ldots, s-1$. These approximations are detailed in the following for second, third and fourth order schemes respectively. 
\subsection{Predictor for the SSP-RK(2,2) scheme}
We obtain the approximation for $w_{j}^{n,p}=w_{j}^{n,p, (0)}$ by truncating \eqref{aTaylor} to the second term:
\begin{equation} \label{eq:pRK2-s0}
w_{j}^{n,p, (0)} = w_{j}^{n, (0)} + \frac{p\Delta t}{M} \partial_{t} w_{j}^{n, (0)} + O(\Delta t_{p}^{2}), \quad p=0, 1, \ldots, M-1,
\end{equation}
where $\Delta t_{p}=p\sfrac{\Delta t}{M}$. To compute the first time derivative, we first notice that $\partial_{t} w_{j}^{n, (0)} = L_{j}(\bw^{n, (0)})$
according to \eqref{eq:autoj}. Thus, by using the solution at stage 1 of SSP-RK(2,2) with a coarse time step
$$ w_{j}^{n, (1)}= w_{j}^{n, (0)} + \Delta t  L_{j}(\bw^{n, (0)}),
$$ 
we deduce that
\begin{equation} \label{eq:Ut}
\partial_{t} w_{j}^{n, (0)} = L_{j}(\bw^{n, (0)}) = \frac{w_{j}^{n, (1)} - w_{j}^{n, (0)}}{\Delta t} + O\left (\Delta t\right ).
\end{equation}
Substituting this into \eqref{eq:pRK2-s0} yields:
\begin{equation} \label{eq:pRK2-s0b}
w_{j}^{n,p, (0)} = w_{j}^{n, (0)} + \frac{p}{M} \left (w_{j}^{n, (1)} - w_{j}^{n, (0)} \right ) + O(\Delta t^{2}). 
\end{equation}
We also need to predict $w_{j}^{n,p, (1)}$, the solution at stage 1 at intermediate time levels, for $p=0, \ldots, M-1$. By definition of the SSP-RK(2,2) scheme, we have:
\begin{equation*} 
w_{j}^{n,p, (1)}= w_{j}^{n,p, (0)} + \frac{\Delta t}{M}  L_{j}(\bw^{n,p, (0)}),
\end{equation*}
As $L_{j}(\bw^{n,p, (0)}) = L_{j}(\bw^{n,(0)}) + O(\Delta t) $, we deduce from \eqref{eq:pRK2-s1} that
\begin{equation*}
w_{j}^{n,p, (1)}= w_{j}^{n,p, (0)} + \frac{\Delta t}{M} L_{j}(\bw^{n,(0)}) + O(\Delta t^{2}),
\end{equation*}
or equivalently via \eqref{eq:pRK2-s0b}
\begin{equation} \label{eq:pRK2-s1}
w_{j}^{n,p, (1)}= w_{j}^{n,p, (0)} + \frac{p+1}{M} \left (w_{j}^{n, (1)} - w_{j}^{n, (0)} \right ) + O(\Delta t^{2}). 
\end{equation}
It is clear from \eqref{eq:pRK2-s0b} and \eqref{eq:pRK2-s1} that the proposed predictor gives second order accurate approximations of the solutions at intermediate time levels $t^{n,p}$, for $p=0, \ldots, M-1$. 
\subsection{Predictor for the SSP-RK(3,3) scheme}
As in the second order case, we approximate $w_{j}^{n,p, (0)}$ by truncating \eqref{aTaylor}, but now to the third term:
\begin{equation} \label{eq:pRK3-s0}
w_{j}^{n,p, (0)} = w_{j}^{n, (0)} + \frac{p\Delta t}{M} \partial_{t} w_{j}^{n, (0)} + \frac{1}{2}\left (\frac{p\Delta t}{M}\right )^{2} \partial_{tt} w_{j}^{n, (0)}  + O(\Delta t_{p}^{3}), \quad p=0, 1, \ldots, M-1.
\end{equation}
The time derivatives are computed from the solutions at stage 1 and stage 2 of SSP-RK(3,3) with a coarse time step:
\begin{align}
w_{j}^{n, (1)}&= w_{j}^{n,(0)} + \Delta t L_{j}(\bw^{n,(0)}), \vspace{3pt}\\
w_{j}^{n, (2)}&= \frac{3}{4}w_{j}^{n,(0)}  + \frac{1}{4}w_{j}^{n,(1)} + \frac{1}{4}\Delta t L_{j}(\bw^{n,(1)}). \label{eq:RK3-s2}
\end{align}
The first time derivative can be obtained as in \eqref{eq:Ut}, for the second time derivative, by the chain rule we deduce from \eqref{eq:autoj} that
\begin{equation*}
 \partial_{tt} w_{j} = \nabla L_{j}(\bw) \cdot \partial_{t} \bw,
 \end{equation*} 
and we will compute the right-hand side by using \eqref{eq:RK3-s2}. In particular, by performing first order Taylor expansion, we obtain:
\begin{equation*}
L_{j}(\bw^{n,(1)})  =  L_{j}(\bw^{n,(0)})+ \nabla L_{j}(\bw^{n,(0)}) \cdot (\bw^{n,(1)}-\bw^{n,(0)}) + O(\Delta t^{2}).
\end{equation*}
Substituting this into \eqref{eq:RK3-s2} yields
\begin{equation*}
w_{j}^{n, (2)}= \frac{3}{4} w_{j}^{n,(0)}  + \frac{1}{4} w_{j} ^{n,(1)} + \frac{1}{4} \Delta t L_{j}(\bw^{n,(0)}) + \frac{1}{4} \Delta t^{2} \, \nabla L_{j}(\bw^{n,(0)}) \cdot \partial_{t} (\bw^{n,(0)})  + O(\Delta t^{3}),
\end{equation*}
from which we deduce that
\begin{equation*}
\begin{array}{ll}
\Delta t^{2} \nabla L_{j}(\bw^{n,(0)}) \cdot \partial_{t} (\bw^{n,(0)}) & = 4w_{j}^{n, (2)} - 3w_{j}^{n,(0)}  -w_{j}^{n,(1)} -\Delta t L_{j}(\bw^{n,(0)}) + O(\Delta t^{3}) \\
& = 4w_{j}^{n, (2)} - 2w_{j}^{n,(0)}  -2w_{j}^{n,(1)} + O(\Delta t^{3}),
\end{array}
\end{equation*}
where the last equality is obtained by \eqref{eq:Ut}. Inserting this and \eqref{eq:Ut} into \eqref{eq:pRK3-s0}, we arrive at
\begin{equation*}
w_{j}^{n,p, (0)} = w_{j}^{n, (0)} + \frac{p}{M} (w_{j}^{n,(1)}-w_{j}^{n,(0)} ) + \frac{p^{2}}{M^{2}}\left (2w_{j}^{n, (2)} - w_{j}^{n,(0)}  -w_{j}^{n,(1)})\right )+ O(\Delta t^{3}),
\end{equation*}
for $p=0, 1, \ldots, M-1$. Similarly, we can approximate $w_{j}^{n,p, (1)}$ and $w_{j}^{n,p, (2)}$, the solutions at stage 1 and stage 2 of SSP-RK(3,3) as follows:
\begin{align*}
& w_{j}^{n,p, (1)}= w_{j}^{n,p,(0)} + \frac{\Delta t}{M} L_{j}(\bw_{j}^{n,p,(0)}) \\
&= w_{j}^{n,p,(0)} + \frac{\Delta t}{M} \left (L_{j}(\bw^{n,(0)}) + \nabla L_{j}(\bw^{n,(0)}) \cdot (\bw^{n,p,(0)}-\bw^{n,(0)})+ O(\Delta t^{2})\right )\\
& = w_{j}^{n,p,(0)} + \frac{\Delta t}{M} L_{j}(\bw_{j}^{n,(0)}) + \frac{\Delta t}{M} \nabla L_{j}(\bw^{n,(0)}) \cdot  \frac{p\Delta t}{M} \partial_{t}\bw^{n,(0)} + O(\Delta t^{3}) \\
& = w_{j}^{n,(0)} + \frac{p+1}{M} (w_{j}^{n,(1)}-w_{j}^{n,(0)} ) + \frac{p(p+2)}{M^{2}} \left (2w_{j}^{n, (2)} - w_{j}^{n,(0)}  -w_{j}^{n,(1)})\right ) + O(\Delta t^{3}),
\end{align*}
and 
\begin{align*}
& w_{j}^{n,p, (2)}= \frac{3}{4}w_{j}^{n,p,(0)} + \frac{1}{4}w_{j}^{n,p, (1)} +\frac{1}{4} \frac{\Delta t}{M} L_{j}(\bw^{n,p,(1)}) \\
&= \frac{3}{4}w_{j}^{n,p,(0)} + \frac{1}{4}w_{j}^{n,p, (1)} +\frac{1}{4}\frac{\Delta t}{M} \left (L_{j}(\bw^{n,(0)}) + \nabla L_{j}(\bw^{n,(0)}) \cdot  (\bw^{n,p,(1)}-\bw^{n,(0)})+ O(\Delta t^{2})\right ) \\
& = w_{j}^{n,p,(0)} + \frac{\Delta t}{M} L_{j}(\bw^{n,(0)}) + \frac{\Delta t}{M} \nabla L_{j}(\bw^{n,(0)}) \cdot  \frac{(p+1)\Delta t}{M} \, \partial_{t}\bw^{n,(0)} + O(\Delta t^{3}) \\
& = w_{j}^{n,(0)} + \frac{2p+1}{2M} (w_{j}^{n,(1)}-w_{j}^{n,(0)} ) + \frac{2p^{2}+2p+1}{2M^{2}} \left (2w_{j}^{n, (2)} - w_{j}^{n,(0)}  -w_{j}^{n,(1)})\right ) + O(\Delta t^{3}).
\end{align*}

\subsection{Predictor for the SSP-RK(5,4) scheme} \label{subsec:pRK4} Again, we approximate $w_{j}^{n,p, (0)}$ by truncating \eqref{aTaylor}, now to the fourth term:
\begin{equation} \label{eq:pRK4-s0}
w_{j}^{n,p, (0)} = w_{j}^{n, (0)} + \frac{p\Delta t}{M} \partial_{t} w_{j}^{n, (0)} + \frac{1}{2}\left (\frac{p\Delta t}{M}\right )^{2} \partial_{tt} w_{j}^{n, (0)}  + \frac{1}{6}\left (\frac{p\Delta t}{M}\right )^{3} \partial_{ttt} w_{j}^{n, (0)}+ O(\Delta t_{p}^{4}),
\end{equation}
for $p=0, 1, \ldots, M-1$. As for the second and third order cases, we approximate the time derivatives 
\begin{align}
\partial_{t} w_{j}^{n, (0)}&=L_{j}(\bw^{n,(0)}), \quad \partial_{tt} w_{j}^{n,(0)}=\nabla L_{j} (\bw^{n,(0)}) \cdot \partial_{t} \bw^{n,(0)}, \vspace{4pt} \label{eq:deri2t}\\
\partial_{ttt} w_{j}^{n, (0)} &= (\partial_{t} \bw^{n,(0)})^{T} \, \bH_{L_{j}}(\bw^{n,(0)}) \,  \partial_{t} \bw^{n,(0)}  + \nabla L_{j}(\bw^{n,(0)} ) \cdot \partial_{tt} \bw^{n,(0)},  \label{eq:deri3t}
\end{align}
by using the solution the first four stages of SSP-RK(5,4) with a coarse time step: 
\begin{align}
w_{j}^{n, (1)}&= \alpha_{10} w_{j}^{n,(0)} + \beta_{10} \Delta t L_{j}(\bw^{n,(0)}), \label{eq:RK4-s1} \vspace{3pt}\\
w_{j}^{n, (2)}&= \alpha_{20} w_{j}^{n,(0)} + \alpha_{21} w_{j}^{n,(1)} + \beta_{21} \Delta t L_{j}(\bw^{n,(1)}),\label{eq:RK4-s2} \vspace{3pt}\\
w_{j}^{n, (3)}&= \alpha_{30} w_{j}^{n,(0)} + \beta_{30} \Delta t L_{j}(\bw^{n,(0)}) + \alpha_{32} w_{j}^{n,(2)} + \beta_{32} \Delta t L_{j}(\bw^{n,(2)}), \label{eq:RK4-s3} \vspace{3pt}\\
w_{j}^{n, (4)}&= \alpha_{40} w_{j}^{n,(0)} + \beta_{40} \Delta t L_{j}(\bw^{n,(0)}) + \alpha_{41} w_{j}^{n,(1)} + \beta_{41} \Delta t L_{j}(\bw^{n,(1)}) \notag \\ 
& \hspace{1cm} + \alpha_{43} w_{j}^{n,(3)} + \beta_{43} \Delta t L_{j}(\bw^{n,(3)}). \label{eq:RK4-s4}  \vspace{-0.2cm}
\end{align}
Denote by  \vspace{-0.2cm}
\begin{equation} \label{eq:rk4tstages}
\Delta t^{n,(i)}= \gamma^{(i)} \Delta t, \; i=0,1,2,3, \vspace{-0.2cm}
\end{equation}
with  \vspace{-0.2cm}
\begin{equation} \label{eq:gamma}
\gamma^{(0)}=0, \quad \gamma^{(1)}=\beta_{10}, \quad \gamma^{(2)}=\alpha_{21}\gamma^{(1)}+\beta_{21}, \quad \gamma^{(3)}=\alpha_{32}\gamma^{(2)} + \beta_{32}+\beta_{30}. \vspace{-0.2cm}
\end{equation}
From \eqref{eq:RK4-s1} we have
\begin{equation*}
\Delta t \, \partial_{t} w_{j}^{n,(0)} = \Delta t \, L_{j}(\bw^{n,(0)}) =\frac{1}{\beta_{10}} \left (w_{j}^{n, (1)}-\alpha_{10} w_{j}^{n,(0)}\right ).
\end{equation*}
Next, we approximate the flux by Taylor expansion with $O(\Delta t^{3})$ truncated error:
\begin{align} 
L_{j}(\bw^{n,(1)}) &=L_{j}(\bw^{n,(0)}) + \nabla L_{j} (\bw^{n,(0)}) \cdot (\bw^{n,(1)}-\bw^{n,(0)}) \notag \\
& \hspace{0.4cm} + \frac{1}{2} (\bw^{n,(1)}-\bw^{n,(0)}) \, \bH_{L_{j}}(\bw^{n,(0)}) \,  (\bw^{n,(1)}-\bw^{n,(0)}) + O(\Delta t^{3}) \nonumber\\
& = L_{j}(\bw^{n,(0)}) +  \nabla L_{j} (\bw^{n,(0)}) \cdot \left (\Delta t^{n,(1)}  \, \partial_{t} \bw^{n,(0)} + \frac{(\Delta t^{n,(1)})^2}{2} \, \partial_{tt} \bw^{n,(0)}\right ) \nonumber\\
& \hspace{0.2cm}+ \frac{(\Delta t^{n,(1)})^2}{2} \, (\partial_{t} \bw^{n,(0)})^{T} \, \bH_{L_{j}}(\bw^{n,(0)}) \,  \partial_{t} \bw^{n,(0)}  + O(\Delta t^{3}), \nonumber \\
& =  L_{j}(\bw^{n,(0)}) +  \Delta t^{n,(1)}  \, \partial_{tt} w_{j}^{n,(0)}  + \frac{(\Delta t^{n,(1)})^2}{2} \,   \partial_{ttt} w_{j}^{n,(0)}  + O(\Delta t^{3}), \label{eq:Lappr1}
\end{align}
in which $\Delta t^{n,(1)}$ is defined in \eqref{eq:rk4tstages} and the last equality is obtained by substituting the derivatives in time \eqref{eq:deri2t}-\eqref{eq:deri3t}. Similarly, \vspace{-0.1cm}
\begin{align}
L_{j}(\bw^{n,(2)}) &=L_{j}(\bw^{n,(0)}) + \Delta t^{n,(2)} \, \partial_{tt} w_{j}^{n,(0)}+ \frac{\left (\Delta t^{n,(2)}\right )^{2}}{2} \, \partial_{ttt} w_{j}^{n,(0)}+ O(\Delta t^{3}),  \label{eq:Lappr2}
\end{align}
and
\begin{align}
L_{j}(\bw^{n,(3)}) &=L_{j}(\bw^{n,(0)}) + \Delta t^{n,(3)} \, \partial_{tt} w_{j}^{n,(0)}+ \frac{\left (\Delta t^{n,(3)}\right )^{2}}{2} \, \partial_{ttt} w_{j}^{n,(0)}+ O(\Delta t^{3}).  \label{eq:Lappr3}
\end{align}
For the fourth order SSP-RK scheme, the number of stages is larger than the order of the scheme. Consequently, we can compute different approximations of the time derivatives $\partial_{tt} w_{j}^{n,(0)}$ and $\partial_{ttt} w_{j}^{n,(0)}$ using either equations \eqref{eq:Lappr1}-\eqref{eq:Lappr2} or \eqref{eq:Lappr2}-\eqref{eq:Lappr3}.  In particular, if we substitute the equations \eqref{eq:Lappr1}-\eqref{eq:Lappr2} into \eqref{eq:RK4-s2}-\eqref{eq:RK4-s3}, we obtain the following system for $\partial_{tt} w_{j}^{n,(0)} $ and $\partial_{ttt} w_{j}^{n,(0)}$:
\begin{align}
w_{j}^{n,(2)} &= A_{j} + \beta_{21}\beta_{10} \Delta t^{2} \partial_{tt} w_{j}^{n,(0)} + \beta_{21}\beta_{10}^{2} \frac{\Delta t^{3}}{2} \partial_{ttt} w_{j}^{n,(0)}+ O(\Delta t^{4}), \label{eq:lin1}\vspace{4pt}\\
w_{j}^{n,(3)} & = B_{j} + \beta_{32} (\alpha_{21}\beta_{10} + \beta_{21}) \Delta t^{2}  \partial_{tt} w_{j}^{n,(0)}+ \beta_{32}(\alpha_{21}\beta_{10} + \beta_{21}) ^{2} \frac{\Delta t^{3}  }{2} \partial_{ttt} w_{j}^{n,(0)} \nonumber \\
& \hspace{8cm}  + O(\Delta t^{4}), \label{eq:lin2} \vspace{-0.2cm}
\end{align}
where \vspace{-0.2cm}
\begin{align*}
A_{j}&=\alpha_{20} w_{j}^{n,(0)} + \alpha_{21} w_{j}^{n,(1)} + \frac{\beta_{21}}{\beta_{10}} \left (w_{j}^{n, (1)}-\alpha_{10} w_{j}^{n,(0)}\right ), \\
B_{j}&= \alpha_{30} w_{j}^{n,(0)} + \alpha_{32} w_{j}^{n,(2)} + \frac{(\beta_{30}+\beta_{32}) }{\beta_{10}} \left (w_{j}^{n, (1)}-\alpha_{10} w_{j}^{n,(0)}\right ).
\end{align*}
By solving \eqref{eq:lin1}-\eqref{eq:lin2}, we can compute fourth order approximations of $\partial_{tt} w_{j}^{n,(0)}$ and $\partial_{ttt} w_{j}^{n,(0)}$, denoted by $\partial_{tt} \overline{w}_{j}^{n,(0)}$ and $\partial_{ttt} \overline{w}_{j}^{n,(0)}$, as linear combinations of the solutions at different stages with a coarse time step $w_{j}^{n,(0)}, \, w_{j}^{n,(1)}, \, w_{j}^{n,(2)}$ and $w_{j}^{n,(3)}$: 
\begin{equation} \label{eq:4thapprox1}
\begin{array}{ll}
\partial_{tt} \overline{w}_{j}^{n,(0)}&=\partial_{tt} \overline{w}_{j}^{n,(0)}(w_{j}^{n,(0)}, \, w_{j}^{n,(1)}, \, w_{j}^{n,(2)}, \, w_{j}^{n,(3)}), \vspace{3pt}\\
\partial_{ttt} \overline{w}_{j}^{n,(0)}&=\partial_{ttt} \overline{w}_{j}^{n,(0)}(w_{j}^{n,(0)}, \, w_{j}^{n,(1)}, \, w_{j}^{n,(2)}, \, w_{j}^{n,(3)}).
\end{array}
\end{equation}
Similarly, we can substitute the equations \eqref{eq:Lappr2}-\eqref{eq:Lappr3} into \eqref{eq:RK4-s3}-\eqref{eq:RK4-s4} to obtain alternative approximations of $\partial_{tt} w_{j}^{n,(0)}$ and $\partial_{ttt} w_{j}^{n,(0)}$, denoted by $\partial_{tt} \overline{\overline{w}}_{j}^{n,(0)}$ and $\partial_{ttt} \overline{\overline{w}}_{j}^{n,(0)}$, as linear combinations of the solutions $w_{j}^{n,(0)}, \, w_{j}^{n,(1)}, \, w_{j}^{n,(3)}$ and $w_{j}^{n,(4)}$: 
\begin{equation} \label{eq:4thapprox2}
\begin{array}{ll}
\partial_{tt} \overline{\overline{w}}_{j}^{n,(0)}&=\partial_{tt} \overline{\overline{w}}_{j}^{n,(0)}(w_{j}^{n,(0)}, \, w_{j}^{n,(1)}, \, w_{j}^{n,(3)}, \, w_{j}^{n,(4)}), \vspace{3pt}\\
\partial_{ttt} \overline{\overline{w}}_{j}^{n,(0)}&=\partial_{ttt} \overline{\overline{w}}_{j}^{n,(0)}(w_{j}^{n,(0)}, \, w_{j}^{n,(1)}, \, w_{j}^{n,(3)}, \, w_{j}^{n,(4)}).
\end{array}
\end{equation}
To take into account values at all four stages, we choose the average of \eqref{eq:4thapprox1} and \eqref{eq:4thapprox2} as approximations of $\partial_{tt} w_{j}^{n,(0)}$ and $\partial_{ttt} w_{j}^{n,(0)}$ respectively
and insert them into \eqref{eq:pRK4-s0} to obtain a fourth order approximation of $w_{j}^{n,p,(0)}$. 

Next, we approximate $w_{j}^{n,p, (i)}, \, i=1,2,3, 4$ based on the definition of the SSP-RK(5,4) scheme:
\begin{equation*}
 w_{j}^{n,p, (i)}=\sum_{\nu=0}^{i-1} \alpha_{i\nu} w_{j}^{n,p, (\nu)} + \beta_{i\nu} \frac{\Delta t}{M} L_{j}(\bw^{n,p, (\nu)}), \quad \forall\, i=1,2,3,4. \vspace{-0.2cm}
\end{equation*}
Denote by \vspace{-0.2cm}
\begin{equation*}
\Delta t^{n,p,(i)}= (p+ \gamma^{(i)})\frac{\Delta t}{M}, \; i=0,1,2,3,
\end{equation*}
with $\gamma^{(i)}$ defined in \eqref{eq:gamma}. We approximate $L_{j}(\bw^{n,p,(i)}), \, i=0,1,2,3,$ as in \eqref{eq:Lappr1}-\eqref{eq:Lappr3} and use \eqref{eq:4thapprox1} and \eqref{eq:4thapprox2} to approximate the time derivatives: 
\begin{equation*} \label{eq:approxLp}
L_{j}(\bw^{n,p,(i)}) = L_{j}(\bw^{n,(0)}) + \Delta t^{n,p,(i)} \partial_{tt} w_{j}^{n,(0)} + \frac{1}{2}(\Delta t^{n,p,(i)})^{2} \partial_{ttt} w_{j}^{n,(0)} + O\left (\Delta t^{3}\right ).
\end{equation*}
Using this we can compute $w_{j}^{n,p,(i)},$ $i=1,2,3,4$ with fourth order accuracy in time. 
\begin{remark} \label{rmrk:pred}
By construction, the predictors for SSP-RK$(2,2)$, SSP-RK$(3,3)$ and SSP-RK$(5,4)$ are respectively second, third and fourth order accurate in time. 
\end{remark}
\end{appendices}
\end{document}